\documentclass[hidelinks, 10pt]{amsart}
\usepackage{amsthm,amssymb,amscd,amsmath,amsfonts,amssymb,amscd,stmaryrd,hyperref,mathrsfs,yhmath,mathtools, amsrefs}                                                                                                                    
\usepackage[all]{xy}
\usepackage{hyperref}
\usepackage{enumerate}
\usepackage{physics}
\usepackage{ragged2e}
\usepackage{xfrac, faktor} 
\usepackage{lipsum}  
\usepackage{tikz-cd}
\usepackage{imakeidx}
\usepackage{subfiles}
\usepackage{bbm}
\usepackage{quiver}
\usepackage{mdframed}
\usepackage[margin = 1.2in]{geometry}

\usepackage{amsmath,amsfonts,amscd,amssymb,verbatim,amsthm}

\newcommand{\HH}{\mathrm{H}}
\newcommand{\OO}{\mathcal{O}}

\DeclareMathOperator{\ff}{ff}

\DeclareMathOperator{\Hom}{Hom}

\DeclareMathOperator{\Mod}{\mathbf{Mod}}
\DeclareMathOperator{\sm}{sm}
\DeclareMathOperator{\Gal}{Gal}

\DeclareMathOperator{\Char}{Char}

\DeclareMathOperator{\fd}{fd}
\DeclareMathOperator{\charfield}{char}

\DeclareMathOperator{\Vect}{\mathbf{Vect}}

\DeclareMathOperator{\Rep}{\mathbf{Rep}}
\DeclareMathOperator{\End}{End}
\DeclareMathOperator{\Deg}{Deg}

\DeclareMathOperator{\Ind}{Ind}
\DeclareMathOperator{\Irr}{Irr}
\DeclareMathOperator{\GL}{GL}
\DeclareMathOperator{\Br}{Br}
\DeclareMathOperator{\Cl}{Cl}
\DeclareMathOperator{\Fun}{Fun}

\DeclareMathOperator{\Aut}{Aut}
\DeclareMathOperator{\Stab}{Stab}

\DeclareMathOperator{\Spec}{Spec}

\DeclareMathOperator{\Ord}{Ord}

\DeclareMathOperator{\lcm}{lcm}
\DeclareMathOperator{\hcf}{hcf}

\makeatletter
\newcommand{\colim@}[2]{%
  \vtop{\m@th\ialign{##\cr
    \hfil$#1\operator@font colim$\hfil\cr
    \noalign{\nointerlineskip\kern1.5\ex@}#2\cr
    \noalign{\nointerlineskip\kern-\ex@}\cr}}%
}
\newcommand{\colim}{%
  \mathop{\mathpalette\colim@{\rightarrowfill@\textstyle}}\nmlimits@
}
\makeatother

\newcommand{\bC}{{\mathbb C}}

\newcommand{\bF}{{\mathbb F}}

\newcommand{\bP}{{\mathbb P}}
\newcommand{\bQ}{{\mathbb Q}}
\newcommand{\bR}{{\mathbb R}}
\newcommand{\bZ}{{\mathbb Z}}

\newcommand{\sA}{{\mathcal A}}
\newcommand{\sB}{{\mathcal B}}
\newcommand{\sC}{{\mathcal C}}

\newcommand{\sF}{{\mathcal F}}

\newcommand{\sO}{{\mathcal O}}

\newcommand{\fT}{{\mathfrak T}}

\newtheorem{thm}{Theorem}
\numberwithin{thm}{section}
\newtheorem{lemma}[thm]{Lemma} 
\newtheorem{prop}[thm]{Proposition} 
\newtheorem{cor}[thm]{Corollary} 
 
\newtheorem*{theoremA}{Theorem A}
\newtheorem*{theoremB}{Theorem B}

\theoremstyle{definition}
\newtheorem{defn}[thm]{Definition}
\newtheorem{eg}[thm]{Example}
\newtheorem{remark}[thm]{Remark}
\newtheorem{question}[thm]{Question}

\title{Character Theory for Semilinear Representations}
\date{\today}
\author{James Taylor}
\email{james.taylor@math.unipd.it}
\address{Dipartimento di Matematica, Universit\`{a} degli Studi di Padova, Via Trieste 63, 35131 Padova, Italy}

\begin{document}
\begin{abstract}
Let $G$ be a group acting on a field $L$, and suppose that $L /L^G$ is a finite extension. We show that the category of semilinear representations of $G$ over $L$ can be described in terms of the category of linear representations of $H$, the kernel of the map $G \rightarrow \Aut(L)$. When $G$ is finite and $L$ has characteristic 0 this provides a character theory for semilinear representations of $G$ over $L$, which recovers ordinary character theory when the action of $G$ on $L$ is trivial.
\end{abstract}
\maketitle
\vspace{-1.5em}
\setcounter{tocdepth}{1}
\tableofcontents
\vspace{-3em}
\section{Introduction}

Suppose that $L$ is a field, and $G$ is a group which acts on $L$ through a group homomorphism $\sigma \colon G \rightarrow \Aut(L)$, where $\Aut(L)$ is the group of field automorphisms of $L$. Write $K \coloneqq L^G$, so we can consider the action as a group homomorphism $\sigma \colon G \rightarrow \Aut(L/K)$.

In this context we can consider the category $\Rep_L^{\rtimes}(G)$ of finite-dimensional $L$-vector spaces $V$ equipped with a semilinear action of $G$: a group homomorphism $\rho \colon G \rightarrow \Aut_K(V)$ such that each $\rho(g)$ is $\sigma_g$-semilinear, meaning that
\[
    \rho(g)(\lambda \cdot v) = \sigma_g(\lambda) \cdot \rho(g)(v) \ \mbox{  for all  } \ g\in G, \lambda \in L , v\in V.
\]
In this category the morphisms are those $L$-linear maps which commute with the action of $G$.

Such semilinear representations arise naturally, this category being identified through taking global sections with $\Vect^G(X)$, the category of $G$-equivariant vector bundles on $X \coloneqq \Spec(L)$.

When the action of $G$ on $L$ is trivial, this is simply the category $\Rep_L(G)$ of $L$-linear representations of $G$. When the action of $G$ on $L$ is non-trivial however, $\Rep_L^{\rtimes}(G)$ is not very well understood.

\subsection{Matrix Interpretation}

One can interpret the category $\Rep_L^{\rtimes}(G)$ more concretely in terms of matrices. For any $n \geq 1$, to give a semilinear representation of dimension $n$ over $L$ is the same as giving a family of matrices $(A_g)_{g \in G}$ in $\GL_n(L)$, which satisfy
\[
    A_{g_1 g_2} = A_{g_1} \cdot {}^{g_1 \!} A_{g_2} \ \mbox{  for all  } \ g_1, g_2 \in G,
\]
where ${}^{g_1 \!} A_{g_2}$ is the matrix obtained by applying $\sigma_{g_1}$ to all entries of $A_{g_2}$. A homomorphism from this to another representation given by $(B_g)_{g \in G}$ in $\GL_m(L)$ is a matrix $M \in M_{m \times n}(L)$ such that
\begin{equation}\label{eqn:matrixhomrelation}
B_g \cdot {}^{g} M = M \cdot A_g \ \mbox{  for all  } \ g \in G.
\end{equation}
In particular, the matrix representations equivalent to $(A_g)_{g \in G}$ are all given by $(B_g)_{g \in G}$, where
\[
B_g \coloneqq P \cdot A_g \cdot {}^{g \hspace{-0.1em}}(P^{-1})
\]
for $P \in \GL_n(L)$. This point of view naturally identifies the set of isomorphism classes of $n$-dimensional semilinear representations of $G$ over $L$ with the pointed set $H^1(G, \GL_n(L))$.

Unlike linear representations, one cannot easily use techniques from linear algebra to describe such families of matrices when the action of $G$ on $L$ is non-trivial. For example, even a description of $\Rep_L^{\rtimes}(G)$ when $G$ is finite and cyclic is not apparent, as one can no longer appeal to standard techniques such as diagonalisation.

\subsection{The Extension Problem}

Let $H \coloneqq \ker(G \rightarrow \Aut(L/K))$, a normal subgroup of $G$. As $H$ acts trivially on $L$, any $V \in \Rep_L^{\rtimes}(G)$ natural restricts to an object $V|_H \in \Rep_L(H)$.

Hand in hand with understanding the category $\Rep_L^{\rtimes}(G)$ is the extension problem for $\Rep_L(H)$:
\begin{enumerate}
    \item Which $W \in \Rep_L(H)$ extend to a semilinear representation of $G$ over $L$?
    \item How many different extensions are possible?
\end{enumerate}
The category $\Rep_L^{\rtimes}(G)$ blends the representation theory of $G$ with the arithmetic of $L$ and $K$. Therefore these questions often contain arithmetic information about the extension $L/K$.

\begin{eg}\label{eg:C2C4intro}
    Suppose that $d \in \bQ^{\times} \setminus (\bQ^{\times})^2$, $L / K$ is the extension $\bQ(\sqrt{d}) / \bQ$, and $G = C_4$, acting on $\bQ(\sqrt{d})$ through the natural surjection $C_4 \twoheadrightarrow \Gal(\bQ(\sqrt{d}) / \bQ)$ with kernel $H = C_2$. Write $\overline{\cdot}$ for the non-trivial automorphism of $\Gal(\bQ(\sqrt{d}) / \bQ)$.

    From the matrix description of semilinear representations, to give a one-dimensional semilinear representation of $C_4$ over $\bQ(\sqrt{d})$ is the same as giving an element $a \in \bQ(\sqrt{d})$ with 
    \[
    a \overline{a} a \overline{a} = 1.
    \]
    Similarly, if $W$ is the non-trivial one-dimensional representation of $C_2$ over $\bQ(\sqrt{d})$, then to give a semilinear representation of $C_4$ over $\bQ(\sqrt{d})$ extending $W$ is the same as giving an $a \in \bQ(\sqrt{d})$ with
    \[
    a \overline{a} = -1.
    \]
    Writing $a = x + y \sqrt{d}$, for $x, y \in \bQ$, we see that $W$ extends to a semilinear representation if and only if
    \[
    x^2 - d y^2 = -1
    \]
    for some $x, y \in \bQ$, i.e.\ that $(x,y)$ is a rational solution to the negative Pell equation.

    For such an $a \in \bQ(\sqrt{d})$, let $V_a$ denote the extension of $W$ to a semilinear representation defined by $a$. From the matrix description of semilinear representations, the semilinear extensions isomorphic to $V_a$ are exactly $V_{ba\overline{b}^{-1}}$ for $b \in \bQ(\sqrt{d})^{\times}$. If we write $b = u + v \sqrt{d}$, then we can consider the isomorphic extensions as being parametrised by $[u:v] \in \bP^1(\bQ)$, and one can compute that
    \begin{align*}
    b a \overline{b}^{-1} = \frac{(u^2 + d v^2)x + (2uv)dy}{u^2 - d v^2} + \sqrt{d} \frac{(u^2 + d v^2)y + (2uv)x}{u^2 - d v^2}.
    \end{align*}

    In particular, the statement that there is at most one extension of $W$ to a semilinear representation of $C_4$ over $\bQ(\sqrt{d})$ if equivalent to the claim that if $X^2 - d Y^2 = -1$ has a solution $(x,y) \in \bQ^2$, then the set of all solutions in $\bQ^2$ is parametrised by $[u:v] \in \bP^1(\bQ)$ via the formula
    \[
        \left( \frac{(u^2 + d v^2)x + (2uv)dy}{u^2 - d v^2}, \frac{(u^2 + d v^2)y + (2uv)x}{u^2 - d v^2} \right).
    \]
    Repeating this with $W$ replaced by the trivial representation $\mathbf{1}$ of $C_2$ over $\bQ(\sqrt{d})$, one considers instead the (positive) Pell equation
    \[
        X^2 - d Y^2 = 1.
    \]
    This always has the trivial solution $(1,0)$, which corresponds to the fact that $\mathbf{1}$ always extends to a semilinear representation of $C_4$ over $\bQ(\sqrt{d})$. Similarly, taking $(1,0)$ as an initial solution, this has at most one extension to a semilinear representation if and only if all solutions are given by
    \[
         \left( \frac{u^2 + d v^2}{u^2 - d v^2}, \frac{2uv}{u^2 - d v^2} \right).
    \]
    In particular, taking $d = -1$, this is the standard parametrisation
    \[
         \left( \frac{u^2 - v^2}{u^2 + v^2}, \frac{2uv}{u^2 + v^2} \right).
    \]
    of the $\bQ$-rational points of $X^2 + Y^2 = 1$ by $[u:v] \in \bP^1(\bQ)$.
\end{eg}

\subsection{Contribution} In this paper we provide what we believe to be the first systematic description of the semilinear representation categories $\Rep_L^{\rtimes}(G)$ beyond the linear case of $\Rep_L(G)$. We work under the following mild assumption, which we impose from now until the end of the introduction:
\begin{itemize}
    \item $L / K$ is a finite extension.
\end{itemize}
For example, this is automatically satisfied if either $G$ is finite or $L$ is finite over its prime subfield. We also note this further implies that $L/K$ is Galois and $G \rightarrow \Gal(L/K)$ is surjective (Remark \ref{rem:introsatisfiesassumptions}).
\subsection{Unique Extension}

Our first result concerns the second half of the extension problem, and shows that any $W \in \Rep_L(H)$ has at most one extension to a semilinear representation of $G$ over $L$.

\begin{theoremA}
    Suppose that $V,W \in \Rep_L^{\rtimes}(G)$. Then:
    \begin{itemize}
        \item $V \cong W$ if and only if $V|_H \cong W|_H$ in $\Rep_L(H)$,
        \item The natural map
    \[
        L \otimes_{K} \Hom_{L \rtimes G}(V, W) \rightarrow \Hom_{L[H]}(V|_H, W|_H)
    \]
    is an isomorphism.
    \item In particular, the natural map 
    \[
        L \otimes_{K} \End_{L \rtimes G}(V) \rightarrow \End_{L[H]}(V|_H)
    \]
    is a $K$-algebra isomorphism.
    \end{itemize}
\end{theoremA}

In particular, this tells us that to understand $\Rep_L^{\rtimes}(G)$, we are reduced to classifying which $W \in \Rep_L(H)$ extend to a semilinear representation of $G$ over $L$: the first half of the extension problem.

\subsection{Group Action} One necessary condition for $W \in \Rep_L(H)$ to extend can be described in terms of a group action of $\Gal(L/K)$ on $\Rep_L(H) / \! \cong$, the set of isomorphism classes of objects of $\Rep_L(H)$.

\begin{defn}
For $g \in G$ and $W \in \Rep_{L}(H)$ we set $g * W \in \Rep_L(H)$ to be $W$ as a $K$-vector space, with action of $L$ and $H$ defined by
\begin{itemize}
        \item $h * w = (g^{-1}hg)(w)$,
        \item $\lambda * w = \sigma_{g^{-1}}(\lambda) w$,
\end{itemize}
for $\lambda \in L$, $h \in H$ and $w \in W$.
\end{defn}
This defines an left action of $G$ on $\Rep_{L}(H) / \! \cong$ which is trivial on $H$, and therefore induces an action of the quotient $G/H$, and so also $\Gal(L/K)$ using the isomorphism $G/H \xrightarrow{\sim} \Gal(L/K)$.

If $W = V|_H$ for some $V \in \Rep_L^{\rtimes}(G)$, then for any $g \in G$, the action of $g$ on $V$ defines an isomorphism $\rho(g) \colon g * W \xrightarrow{\sim} W$, and so $W$ is fixed by $\Gal(L/K)$. Therefore a necessary condition for $W$ to extend to a semilinear representation of $G$ over $L$ is that $W$ is fixed by $\Gal(L/K)$.

The converse is not true however: returning to Example \ref{eg:C2C4intro}, the representation $W$ is always fixed by $\Gal(L/K)$, but when $d$ is square-free the negative Pell equation has a solution if and only if all odd prime divisors $p$ of $d$ satisfy $p \equiv 1 \! \mod 4$. 

\subsection{Indecomposable Representations}

The category $\Rep_L^{\rtimes}(G)$ satisfies the Krull-Remak-Schmidt Theorem (Lemma \ref{lem:KRS}), and therefore one is reduced to understanding $\Ind_L^{\rtimes}(G)$, the set of isomorphism classes of indecomposable semilinear representations of $G$ over $L$. We similarly write $\Ind_L(H)$ for the set of isomorphism classes of indecomposable representations of $H$ over $L$.

Our next result describes $\Ind_L^{\rtimes}(G)$ in terms of $\Ind_L(H)$ and its action of $\Gal(L/K)$. To ease notation, we set $\Gamma \coloneqq \Gal(L/K)$, and for $W \in \Ind_L(H)$ write $\Gamma_W$ for the stabiliser of $W$ in $\Gamma$.

Along with the restriction functor
\[
(-)|_H \colon \Rep_{L}^{\rtimes}(G) \rightarrow \Rep_{L}(H),
\]
we can, denoting by $L \rtimes G$ the skew group ring (Definition \ref{def:skewgroupring}), consider the induction functor
\[
\Ind_H^G(-) \coloneqq (L \rtimes G) \otimes_{L[H]} - \colon \Rep_{L}(H) \rightarrow \Rep_{L}^{\rtimes}(G).
\]
\begin{theoremB}
Suppose that $V \in \Ind_{L}^{\rtimes}(G)$ and $W \in \Ind_{L}(H)$ is a direct summand of $V|_H$. Then:
\begin{itemize}
    \item $V|_H$ and $\Ind_H^G(W)$ are described by
    \[
        V|_H \cong \bigoplus_{\gamma \in \Gamma / \Gamma_W} (\gamma * W)^{m(V)}, \qquad \Ind_H^G(W) \cong V^{\oplus |\Gamma_W| / m(V)}
    \]
    for some integer $m(V) \geq 1$ with $m(V) \mid |\Gamma_W|$.
    \item $V$ is irreducible if and only if $W$ is irreducible.
    \item $\Ind_H^G(W) \cong \Ind_H^G(\gamma * W)$ for any $\gamma \in \Gamma$.
    \item Any $W \in \Ind_{L}(H)$ is a direct summand of $V|_H$ for a unique $V \in \Ind_{L}^{\rtimes}(G)$.
    \item This defines bijective correspondences
    \[
        \Ind_{L}^{\rtimes}(G) \leftrightarrow \Ind_{L}(H) / \hspace{0.1em} \Gamma, \qquad \Irr_{L}^{\rtimes}(G) \leftrightarrow \Irr_{L}(H) / \hspace{0.1em} \Gamma.
    \]
\end{itemize}
\end{theoremB}

In particular, this gives a complete description of $\Ind_{L}^{\rtimes}(G)$, up to a knowledge of the numbers $m(V)$: for any orbit of $\Gal(L/K)$ on $\Ind_L(H)$, $m(V)$ copies of the orbit sum extends uniquely to an indecomposable semilinear representation of $G$ over $L$, and all elements of $\Ind_L^{\rtimes}(G)$ arise this way.

\subsection{Semilinear Schur Index} Theorem B allows us to define the following.

\begin{defn}
    For $W \in \Ind_L(H)$, set $m_K^L(W) \coloneqq m(V)$ for the unique corresponding $V \in \Ind_L^{\rtimes}(G)$.
\end{defn}

We refer to $m_K^L(W)$ as the \emph{semilinear Schur index} of $W$, as these generalise the classical Schur indices in an appropriate sense (see Section \ref{sect:classicalSchurIndex}). In light of Theorem B, these measure the failure for a $\Gamma$-fixed $W \in \Ind_L(H)$ to extend to a semilinear representation.

\begin{cor}
    $W \in \Ind_L(H)$ extends to a semilinear representation of $G$ over $L$ if and only if $\Gamma$ fixes $W$ and $m_K^L(W) = 1$.
\end{cor}

In Section \ref{sect:semilinearschurindex} we establish the basic properties of the $m_K^L(W)$, including the fact that all $m_K^L(W) = 1$ whenever $W \in \Irr_L(H)$ and $L$ is a finite field. In general, like their classical counterparts, the numbers $m_K^L(W)$ are hard to determine and often contain arithmetic information. For example, for $W$ as in Example \ref{eg:C2C4intro}, $m_{K}^L(W) \in \{1,2\}$, and $m_K^L(W) = 1$ if and only if $x^2 - dy^2 = -1$ has a solution over $\bQ$.

\subsection{Character Theory}

When $G$ is finite and $L$ has characteristic $0$, character theory provides a powerful and computationally efficient tool for completely describing the category $\Rep_L(G)$.

Theorem A allows us to give a good definition of the character for $\Rep_L^{\rtimes}(G)$ in this setting, which recovers the usual character theory for $\Rep_L(G)$ when $G$ acts trivially on $L$.
\begin{defn}
For $V \in \Rep_L^{\rtimes}(G)$, we set $\chi_V \coloneqq \chi_{V|_H}$, the character of $V|_H$.
\end{defn}
As an immediate consequence of Theorem A we have the following properties.
\begin{cor}
Suppose that $G$ is finite and $L$ has characteristic $0$. If $V, W \in \Rep_{L}^{\rtimes}(G)$, then:
\begin{itemize}
    \item $\chi_V = \chi_W$ if and only if $V \cong W$,
    \item $\langle \chi_V, \chi_W \rangle = \dim_K \Hom_{L \rtimes G}(V,W)$.
\end{itemize}
\end{cor}
Here the inner product is the usual inner product of characters of $H$ over $L$:
\[
\langle \chi_V, \chi_W \rangle = \frac{1}{|H|} \sum_{h \in H} \chi_V(h) \chi_W(h^{-1}).
\]

More generally, when $G$ is finite, $\Rep_L^{\rtimes}(G)$ is semisimple if and only if $|H| \in L^{\times}$ (Corollary \ref{cor:sscriterion}), and in this context one can still use character theory to understand $\Rep_L^{\rtimes}(G)$ (Section \ref{sect:charactertheory}).

For example, this gives a complete description of $\Rep_L^{\rtimes}(G)$ using only characters when $L$ is a finite field, as in this case all $m_K^L(W) = 1$.

As an example of how this character theory works, we give a complete classification of the semilinear representations of $S_3$ acting on a general degree two Galois extension $L/K$ via the natural projection $S_3 \twoheadrightarrow \Gal(L/K)$. This is first done by hand in Section \ref{sect:firstexampleS3}, in the case of characteristic $0$, with explicit semilinear matrix computations. We then return to this example later in Section \ref{sect:secondexampleS3} to recover this classification more conceptually using characters and Theorems A and B.




\subsection{Realisation of Division Algebras}

For $V \in \Irr_L^{\rtimes}(G)$, the $K$-algebra $D \coloneqq \End_{L \rtimes G}(V)$ is a division algebra over $K$, which may or may not be central over $K$. When $\End_{L[H]}(W) = L$, for any corresponding $W \in \Irr_L(H)$, then $m(V)$ is the degree of $D$ in $\Br(Z(D))$, and $Z(D) = K$ exactly when $W$ is fixed by $\Gamma$.

In Section \ref{sect:realisedivisionalg} we prove that this process exhausts $\Br(K)$: for any central division algebra $D$ over a characteristic $0$ field $K$, we show that $D$ is realised as an endomorphism ring of some $V \in \Irr_L^{\rtimes}(G)$, for some finite group $G$ and finite Galois extension $L/K$. We further have that $m(V) = \Deg(D)$, which shows that there are no restrictions on the possible values of the $m(V)$ in general.

This is in contrast to what happens for for linear representations: the classes of $\Br(\bQ)$ realised as endomorphism rings of linear representations of finite groups over $\bQ$ are exactly those classes defined by cyclotomic algebras, which form a proper subgroup of $\Br(\bQ)$.
\subsection{Determination of Semilinear Schur Indices}

By Theorem B, to understand $\Rep_L^{\rtimes}(G)$ one is reduced to understanding the semilinear Schur indices of $W \in \Ind_L(H)$. This is hard in general.

For classical Schur indices, one has Unger's algorithm \cite{UNG} for computing Schur indices over number fields, and it would be interesting to see if this can be generalised to semilinear Schur indices. There is still a local-global principle here (Corollary \ref{cor:localglobal}), and so the problem can similarly be reduced to a problem for local fields. However many results upon which the algorithm rests require deep facts about classical Schur indices, which at present are not established for semilinear Schur indices.

\subsubsection{Other Associated Linear Categories}
One natural way one might try to glean information about the semilinear Schur indices $m_K^L(W)$ is to use other linear representation categories associated with $\Rep_L^{\rtimes}(G)$. Theorem B and the theory developed above can be viewed as a partial description of $\Ind_L^{\rtimes}(G)$ in terms of $\Ind_L(H)$, but there is another linear representation category that one might naturally consider, namely $\Rep_K(G)$. 

We can produce semilinear representations using the induction functor
\[
    L \otimes_K - \colon \Rep_K(G) \rightarrow \Rep_L^{\rtimes}(G),
\]
which is most simply described in terms of matrices: any matrix representation of $G$ over $K$ will define a semilinear matrix representation of $G$ over $L$, as all matrices are valued in $K$. Using this way to construct objects of $\Rep_L^{\rtimes}(G)$, this gives us more information about the numbers $m_K^L(W)$, but does not completely determine them (see Section \ref{sect:secondexampleS3}).

\subsubsection{Linear Characters}
Unlike the classical Schur indices $\overline{m}_{K}^L(W)$, the semilinear Schur indices $m_K^L(W)$ of $1$-dimensional characters can be non-trivial. This is already illustrated in Example \ref{eg:C2C4intro}, for any $d$ for which $x^2 - dy^2 = -1$ has no solution over $\bQ$. In Section \ref{sect:cohomology} we prove some non-trivial properties of $m_K^L(\chi)$ for $1$-dimensional representations $\chi$ (see Corollary \ref{cor:1dimSLSchurrelation}).

\subsubsection{The Case of $\bC / \bR$}
Classically, the biggest success story for determining Schur indices is when $L/K$ is the extension $\bC / \bR$ and $H$ is a finite group. One can associate to each $W \in \Irr_{\bC}(H)$ a number 
\[
    \sF(W) \coloneqq \frac{1}{|H|} \sum_{h \in H} \chi_W(h^2) \in \{-1, 0, 1\},
\]
known as the Frobenius-Schur indicator. It turns out that this number tells you everything: $W$ is fixed by $\Gal(\bC / \bR)$ if and only if $\sF(W) \neq 0$, and such a $W$ is defined by a representation over $\bR$ if and only if $\sF(W) = 1$. In particular, $\overline{m}_{\bR}^{\bC}(W) = 2$ if $\sF(W) = -1$, and $\overline{m}_{\bR}^{\bC}(W) = 1$ otherwise.

Whilst the existence of such a simple way to determine the numbers $\overline{m}_{\bR}^{\bC}(W)$ is already remarkable, what is even more remarkable is that this theory extends to semilinear representations.

Given a finite group $G$, and a surjection $G \twoheadrightarrow \Gal(\bC / \bR)$ with kernel $H$, then for any $W \in \Irr_{\bC}(H)$ one can define the \emph{generalised Frobenius-Schur indicator}:
\[
    \tilde{\sF}(W) \coloneqq \frac{1}{|H|} \sum_{h \in G \setminus H} \chi_W(h^2).
\]
It turns out that, like before,
\begin{itemize}
    \item $\tilde{\sF}(W) \in \{-1,0,1\}$,
    \item $W$ is fixed by the (now twisted) action of $\Gal(\bC / \bR)$ if and only if $\tilde{\sF}(W) \neq 0$,
    \item $W$ extends to $\Rep_{\bC}^{\rtimes}(G)$ if and only if $\tilde{\sF}(W) = 1$ \cite[Thm.\ 4.2]{RUMTAY}.
\end{itemize}
In particular, $m_{\bR}^{\bC}(W) = 2$ if $\tilde{\sF}(W) = -1$, and $m_{\bR}^{\bC}(W) = 1$ otherwise.

An in-depth study of the numbers $m_K^L(W)$ in the smallest non-trivial case, when $L/K$ has degree two, is the subject of the forthcoming work \cite{RUMTAY4}, where alternatives to the generalised Frobenius-Schur indicator for determining the numbers $m_K^L(W)$ are explored.

\subsection{Structure of the Paper}

Let us now describe how the paper is organised, and give an overview of the strategy of the proof of Theorem A and Theorem B.

In Section \ref{sect:semilinearreps}, we recall some basic notions regarding semilinear representations, and how these can be interpreted in terms of matrices.

In Section \ref{sect:semilinearrepsGaloisext} we fix our assumptions for the remainder of the paper. Crucially, even though we are ultimately interested in the case when $L/K$ is an extension of fields, we work with finite Galois extensions $L$ of a field $K$ that are potentially disconnected (so $L$ is a product of fields).

In Section \ref{sect:associatedlinearcats} we study those semilinear representations arising from matrix representations defined over $K$, and give an extended example of how these can be used in some cases to describe all semilinear representations of $G$ over $L$ in Section \ref{sect:firstexampleS3}, which also serves to demonstrate how one works with semilinear representations explicitly.

In Section \ref{sect:splitgroupext}, we consider one case in which the category $\Rep_L^{\rtimes}(G)$ is completely understood, which is when $H \hookrightarrow G$ is split. This recovers the theory of base change for linear representations.

In Section \ref{sect:connectedcomponents} we show that the semilinear representation category $\Rep_L^{\rtimes}(G)$, for $L$ a potentially disconnected finite Galois extension of $K$, is equivalent to $\Rep_F^{\rtimes}(G_e)$, where $F$ is the field defined by a connected component $\Spec(L)$, and $G_e$ is the stabiliser. This reduces general semilinear representation categories to the case when $L$ is a field, and is used in Section \ref{sect:mainresults} to prove Theorems A and B.

In Section \ref{sect:mainresults}, we prove Theorems A and B. The key idea is to interpret the functors $(-)|_H$ and $\Ind_H^G$ between $\Rep_L^{\rtimes}(G)$ and $\Rep_L(H)$ as certain honest base change functors for the field extension $L / K$. This allows us to port many properties of base change to our setting, such as the preservation of semisimplicity, and how the functors $(-)|_H$ and $\Ind_H^G$ compose. Theorem A follows from Corollary \ref{cor:detectisom}, Remark \ref{rem:weakcompatwithcomp}, Remark \ref{rem:introsatisfiesassumptions}, and Remark \ref{rem:weakhomsetisom}, and Theorem B follows from Remark \ref{rem:introsatisfiesassumptions} and Theorem \ref{thm:mainthm}. We also relate the integers $m(V)$ to the $K$-algebras $\End_{L \rtimes G}(V)$ in Proposition \ref{prop:descofcentre}.

In Section \ref{sect:semilinearschurindex} we define the semilinear Schur index of an object of $\Ind_L(H)$, and establish its basic properties, and in Section \ref{sect:classicalSchurIndex} we explain how the semilinear Schur indices specialise to the (classical) Schur indices in the case when $H \hookrightarrow G$ is split. We also highlight which properties the classical Schur indices satisfy which are not satisfied by all semilinear Schur indices.

In Section \ref{sect:basechange} we study the splitting behaviour of the $K$-division algebras $\End_{L \rtimes G}(V)$ for $V \in \Irr_L^{\rtimes}(G)$, and consider base change for semilinear representations. When $K$ is a number field we use the local-global principle for $\Br(K)$ to prove a local-global principle for semilinear representations.

In Section \ref{sect:charactertheory} we reinterpret Theorems A and B in terms of characters, and in Section \ref{sect:secondexampleS3} we revisit the extended example of Section \ref{sect:firstexampleS3} and rederive the classification much more efficiently using this character theory. In Section \ref{sect:countcc} we specialise the framework slightly, and show that when the extension $L/K$ is cyclotomic there is also an action of $\Gal(L/K)$ on $G$, which allows one to compute the order $|\Irr_L^{\rtimes}(G)|$ purely in terms of the group theory of $G$.

In Section \ref{sect:cohomology} we give a cohomological interpretation of our results in terms of Galois cohomology. In particular, we give a new description of the transgression map in our setting, which has consequences for the semilinear Schur indices of $1$-dimensional characters of $H$.

In Section \ref{sect:realisedivisionalg} we show that all central division algebras over $K$ are realised as endomorphism rings of irreducible semilinear representations of $G$ over $L$, for some group $G$ and finite Galois extension $L$. 

Finally, in Section \ref{sect:exttoinfinite} we indicate how the results of this paper extend to infinite Galois extensions.

\subsection{Further Directions} Whilst we have defined the semilinear Schur indices $m_K^L(W)$ here and established their basic properties, it remains to better understand their structure. There is a wealth of literature regarding classical Schur indices, and it is natural to ask how much of this generalises.

For example, one obtains strong properties of the $\overline{m}_K^L(W)$ from Brauer's induction theorem, and it would be natural to try and understand how semilinear induction theory works in this context.

It would also be interesting to work out an integral version of the theory, which in light of Example \ref{eg:C2C4intro} should encode information about integral solutions of certain diophantine equations.

\subsection*{Acknowledgements} I would like to thank Tom Adams, H\r{a}vard Damm-Johnsen, Lorenzo La Porta and Dmitriy Rumynin for interesting conversations regarding the contents of this paper. This research was supported by an LMS Early Career Research Fellowship at the University of Cambridge and the departmental grant Progetto Sviluppo Dipartimentale - UNIPD PSDIP23O88 at the University of Padova.

\subsection*{Notation}

For a ring $R$, we write $\Mod_R$ for the category of left $R$-modules. For a ring homomorphism $R \rightarrow S$ and modules $M \in \Mod_R$, $N \in \Mod_S$, we write $\Ind_R^S M \coloneqq S \otimes_R M$ and $\Res_R^S N$ for $N$ viewed as an $R$-module. When $R \rightarrow S$ is the canonical inclusion $L \rtimes H \hookrightarrow L \rtimes G$ for a ring $L$ and groups $H \leq G$ acting through a group homomorphism $G \rightarrow \Aut(L)$, we use the shorthand notations $\Ind_H^G M$ and $N|_H$ respectively.

\section{Semilinear Representations}\label{sect:semilinearreps}

Suppose that $G$ is a group, $R$ is a commutative ring, and $\sigma \colon G \rightarrow \Aut(R)$ is group homomorphism, where $\Aut(R)$ denote the group of ring automorphisms of $R$.

The central objects we are interested in are semilinear representations of $G$.

\begin{defn}
    A \emph{semilinear representation} of $G$ is a finite-rank free $R$-module $V$ with a group homomorphism $\rho \colon G \rightarrow \Aut_{\bZ}(V)$ such that
    \[
    \rho(g)(\lambda \cdot v) = \sigma_g(\lambda) \cdot \rho(g)(v) \ \mbox{  for all  } \ g\in G, \lambda \in R , v\in V.
    \]
    A \emph{morphism of semilinear representations} from $V$ to $W$ is an $R$-linear map $\phi \colon V \rightarrow W$ such that 
    \[
        \rho_W(g) \circ \phi = \phi \circ \rho_V(g) \ \mbox{  for all  } \ g\in G.
    \]
    We denote the category of such representations by $\Rep_{R}^{\hspace{0.1em} \sigma}(G)$.
\end{defn}

\begin{eg}
    When $\sigma$ is trivial, so $\sigma_g$ is the identity of $R$ for all $g \in G$, $\Rep_{R}^{\hspace{0.1em} \sigma}(G)$ is simply the category of finite-dimensional linear representations of $G$ over $R$. In this case we write omit $\sigma$ from the notation and write $\Rep_{R}(G)$ for $\Rep_{R}^{\hspace{0.1em} \sigma}(G)$.
\end{eg}

We can perform certain constructions with semilinear representations, which generalise those of linear representations.

\begin{eg}\label{eg:tensorandhoms}
    Suppose that $V, W \in \Rep_{R}^{\hspace{0.1em} \sigma}(G)$. Then
    \begin{itemize}
        \item $V \oplus W$ is a semilinear representation of $G$, where $g \in G$ acts by
    \[
        (v, w) \mapsto (\rho_V(g)(v), \rho_W(g)(w)).
    \]
        \item $V \otimes_{R} W$ is a semilinear representation of $G$, where $g \in G$ acts by
    \[
       v \otimes w \mapsto \rho_V(g)(v) \otimes \rho_W(g)(w).
    \]
    \item $\Hom_{R}(V, W)$ is a semilinear representation of $G$, where $g \in G$ acts by
    \[
        f \mapsto \rho_W(g) \circ f \circ \rho_V(g^{-1}).
    \]
    \end{itemize}
\end{eg}

\begin{eg}\label{eg:trivialrep}
The trivial representation is $R$ with its natural action of $G$. This is irreducible if and only if $R$ has no proper non-trivial $G$-stable left ideals, and has endomorphism ring $R^G$.
\end{eg}

\begin{defn}
An $R$-submodule $W$ of a semilinear representation $V$ is a \emph{subrepresentation} if
\[
        \rho(g)(w) \in W \ \mbox{  for all  } \ g\in G, w \in W.
\]

We say that $V$ is \emph{irreducible} if $V$ is non-zero and $0$ and $V$ are the only subrepresentations of $V$, and write $\Irr_{R}^{\hspace{0.1em} \sigma}(G)$ for the set of isomorphism classes of irreducible objects of $\Rep_{R}^{\hspace{0.1em} \sigma}(G)$.

We say that $V$ is \emph{indecomposable} if $V$ is non-zero and $V$ cannot be written as the direct sum of two non-zero subrepresentations, and write $\Ind_{R}^{\hspace{0.1em} \sigma}(G)$ for the set of isomorphism classes of indecomposable objects of $\Rep_{R}^{\hspace{0.1em} \sigma}(G)$.
\end{defn}

\subsection{Matrix Form of Semilinear Representations}\label{sect:matrixdesc}

We can describe the category $\Rep_{R}^{\hspace{0.1em} \sigma}(G)$ more concretely in terms of matrices. We write $\Mod_R^{\ff}$ for the category of finite-rank free $R$-modules.
\subsubsection{Semilinear Algebra}
For any $\gamma \in \Aut(R)$ and $V \in \Mod_R^{\ff}$, we write ${}^{\gamma}V$ for $V$ as an abelian group but with $R$-module structure $\lambda * v = \gamma^{-1}(\lambda)v$. This defines a functor
\[
    {}^{\gamma}(-) \colon \Mod_R^{\ff} \rightarrow \Mod_R^{\ff},
\] 
which is the identity on morphisms, and satisfies ${}^{\gamma}({}^{\mu}V) = {}^{\gamma \mu}V$.

Let $V, W \in \Mod_R^{\ff}$ and fix bases $\sA$ and $\sB$ for $V$ and $W$ respectively, which induce an $R$-module identification
\[
    \End_R(V,W) \rightarrow M_{m \times n}(R), \qquad [T \colon V \rightarrow W] \mapsto {}_{\sB}[T]_{\sA}.
\]
For any $\gamma \in \Aut(R)$, $\sA$ and $\sB$ also form bases for ${}^{\gamma}V$ and ${}^{\gamma}W$ respectively, and there is similarly an $R$-module identification
\[
    \End_R({}^{\gamma}V,{}^{\gamma}W) \rightarrow M_{m \times n}(R), \qquad [T \colon V \rightarrow W] \mapsto {}_{\sB}[T]_{\sA},
\]
In fact, there is an equality of sets
\[
    \End_R(V,W) = \End_R({}^{\gamma}V,{}^{\gamma}W),
\]
and under this equality the two identifications above are related by the commutativity of the diagram
\[\begin{tikzcd}
	{\End_R(V,W)} & {M_{m \times n}(R)} \\
	{\End_R({}^{\gamma}V,{}^{\gamma}W)} & {M_{m \times n}(R)}
	\arrow["\sim", from=1-1, to=1-2]
	\arrow["{=}"', from=1-1, to=2-1]
	\arrow["{{}^{\gamma}(-)}", from=1-2, to=2-2]
	\arrow["\sim", from=2-1, to=2-2]
\end{tikzcd}\]
where ${}^{\gamma}(-) \colon M_{m \times n}(R) \rightarrow M_{m \times n}(R)$ is the ring automorphism defined by applying $\gamma$ to each entry.

Concretely, if $A \in M_{m \times n}(R)$ represents an $R$-linear map $T \colon V \rightarrow W$, then ${}^{\gamma \!}A$ represents the $R$-linear map ${}^{\gamma}T \colon {}^{\gamma}V \rightarrow {}^{\gamma}W$ (which is set-theoretically simply the map $T$).

Suppose now that $V, W \in \Mod_R^{\ff}$, and $S \colon V \rightarrow W$ is a \emph{$\gamma$-semilinear map}, meaning that
\[
S(\lambda v) = \gamma(\lambda) S(v) \ \mbox{  for all  } \ \lambda \in R, v \in V,
\]
or equivalently, that the map $S \colon {}^{\gamma} V \rightarrow W$ is $R$-linear. Let $B$ be the matrix associated to the $R$-linear map $S \colon {}^{\gamma} V \rightarrow W$ with respect to the bases $\sA$ and $\sB$, and write $[-]$ for the identifications $V \xrightarrow{\sim} R^n$, $W \xrightarrow{\sim} R^m$ induced by these choices of basis. Then
\[
[S(v)] = B \cdot {}^{\gamma}[v],
\]
where ${}^{\gamma}[v]$ is $\gamma$ applied to all the entries of $[v]$.

\subsubsection{Semilinear Matrix Representations}
Suppose now that $V \in \Rep_{R}^{\hspace{0.1em} \sigma}(G)$. Then each $\rho(g) \colon V \rightarrow V$ is $\sigma_g$-semilinear and $\rho(g_1 g_2) = \rho(g_1) \circ \rho(g_2)$, so $V$ determines $R$-linear maps
\[
(\rho(g) \colon {}^{\sigma_g}V \rightarrow V)_{g \in G},
\]
which satisfy $\rho(g_1 g_2) =  \rho(g_1) \circ {}^{\sigma_{g_1}}\rho(g_2)$. Conversely, given $V \in \Mod_R^{\ff}$, any such a family of $R$-linear maps defines a semilinear representation structure on $V$.

Let $n \coloneqq \text{rank}_R(V)$, and fix a basis of the $R$-module $V$. For any $g \in G$, write $A_g \in M_{n \times n}(R)$ for the matrix of the linear map $\rho(g) \colon {}^{\sigma_g} V \rightarrow V$. The compatibility relations above in matrix form become
\begin{equation}\label{eqn:matrixrelation}
    A_{g_1g_2} = A_{g_1} \cdot {}^{\sigma_{g_1} \!} A_{g_2} \ \mbox{  for all  } \ g_1,g_2 \in G.
\end{equation}
In particular, as $I_n = A_{g g^{-1}} = A_{g} \cdot {}^{\sigma_g \!} A_{g^{-1}}$, each $A_g \in \GL_n(R)$. The following is direct to verify.

\begin{prop}\label{prop:matrixdesc}
Suppose that $V \in \Mod_R^{\ff}$ has rank $n \geq 1$, and fix a basis of $V$.

Then taking the matrix of a semilinear map defines a canonical bijection between semilinear representation structures $\rho$ on $V$ and families $(A_g)_{g \in G}$ in $\GL_n(R)$ satisfying (\ref{eqn:matrixrelation}). Under this bijection,
\begin{equation}\label{eqn:matrixevalrelation}
[\rho(g)v] = A_g \cdot {}^{\sigma_g}[v] \ \mbox{  for all  } \ g \in G, v \in V.
\end{equation}
In particular, an $R$-submodule $W$ of $V$ is a subobject of $V$ in $\Rep_{R}^{\hspace{0.1em} \sigma}(G)$ exactly when
\[
        A_g \cdot {}^{\sigma_g}[w] \in W \ \mbox{  for all  } \ g\in G, w \in W.
\]

Suppose that $W \in \Mod_R^{\ff}$ has rank $m \geq 1$, and fix a basis of $W$. Let $V$ has semilinear representation structure corresponding to $(A_g)_{g \in G}$ in $\GL_n(R)$ and let $W$ have semilinear representation structure corresponding to $(B_g)_{g \in G}$ in $\GL_m(R)$.

Then an $R$-linear morphism $V \rightarrow W$ corresponding to $M \in M_{m \times n}(R)$ is a morphism in $\Rep_{R}^{\hspace{0.1em} \sigma}(G)$ if and only if $M$ satisfies
\begin{equation}\label{eqn:matrixhomrelation}
B_g \cdot {}^{\sigma_g} M = M \cdot A_g \ \mbox{  for all  } \ g \in G.
\end{equation}
\end{prop}

\begin{eg}
Suppose that $V \in \Rep_{R}^{\hspace{0.1em} \sigma}(G)$ has $R$-rank $n \geq 1$, and $\sA$ and $\sB$ are $R$-bases of $V$. Suppose that $(A_g)_{g \in G}$ and $(B_g)_{g \in G}$ represent $V$ with respect to $\sA$ and $\sB$ respectively, and that $P \in \GL_n(R)$ is the change of basis matrix from $\sA$ to $\sB$, meaning that $P \circ [-]_{\sA} = [-]_{\sB}$. Then
\[
B_g = P \cdot A_g \cdot {}^{\sigma_g}(P^{-1}).
\]
\end{eg}

\begin{eg}
We can give matrix descriptions of the constructions of Examples \ref{eg:tensorandhoms} and \ref{eg:trivialrep}. The trivial representation has matrix representation by $(I_1)_{g \in G}$, for $I_1$ the identity of $M_1(R)$, and if $V$ and $W$ in $\Rep_{R}^{\hspace{0.1em} \sigma}(G)$ have matrix representations $(A_g)_{g \in G}$ and $(B_g)_{g \in G}$ respectively, then 
\begin{itemize}
    \item $V \oplus W$ is represented by the block sum $(A_g \oplus B_g)_{g \in G}$,
    \item $V \otimes_{R} W$ is represented by the Kronecker product $(A_g \otimes B_g)_{g \in G}$,
    \item $\Hom_{R}(V,R)$ is represented by $((A_g^{-1})^T)_{g \in G}$.
\end{itemize}
\end{eg}

\subsection{Semilinear Representations and Twisted Group Rings}

Semilinear representations can also be viewed as modules over twisted group algebras. Suppose as above that $\sigma \colon G \rightarrow \Aut(R)$ is a group homomorphism.

\begin{defn}\label{def:skewgroupring}
    The ring $R \rtimes_{\sigma} \! G$ is $R[G]$ as an abelian group, but with multiplication
    \[
        (r_1 \cdot g_1) * (r_2 \cdot g_2) \coloneqq r_1 \sigma_{g_1}(r_2) \cdot g_1g_2 \ \mbox{  for all  } \ r_1, r_2 \in R, \ g_1,g_2 \in G,
    \]
    extended to all of $R[G]$ by linearly.
\end{defn}

Note that $R$ is not central in $R \rtimes_{\sigma}\! G$, but $R^G$ is. We therefore naturally view $R \rtimes_{\sigma}\! G$ as an $R^G$-algebra. The following is well-known, and direct the verify.

\begin{prop}
    Suppose that $V \in \Mod_R^{\ff}$. Then there is a canonical bijection between $R \rtimes_{\sigma} \! G$-module structures on $V$ and semilinear structures on $V$, under which a morphism in $\Mod_R^{\ff}$ is a morphism in $\Rep_{R}^{\hspace{0.1em} \sigma}(G)$ if and only if it is a morphism in $\Mod_{R \rtimes_{\sigma} \hspace{-0.1em} G}$.
\end{prop}

In particular, we see that $\Rep_{R}^{\hspace{0.1em} \sigma}(G)$ is canonically identified with the full subcategory of $\Mod_{R \rtimes_{\sigma} \hspace{-0.1em} G}$ where the underlying $R$-module is free of finite rank.

\subsection{Finite Products of Fields}

In this paper we are ultimately interested in the case when $R$ is a Galois extension of a field and $G$ acts through the Galois action. We first consider the more general situation when $R$ is any finite product of fields.

Therefore, suppose until the end of this section that $L = \prod_i L_i$ is a product of fields with an action $\sigma \colon G \rightarrow \Aut(L)$ of a group $G$ such that the induced action of $G$ on the set $\{e_i\}_i$ of principal idempotents is transitive. We note that this implies $K := L^G$ is a field by \cite[Lem.\ 3.2]{TAY4}, taking the triple $(F,G,H)$ of \cite[\S 3.1]{TAY4} to be $(L,G,G)$.
\begin{prop}\label{prop:freeandss}
    Suppose that the index set $I$ of the product $L = \prod_i L_i$ is finite. Then any $L \rtimes G$-module is free as an $L$-module. If additionally $G$ is finite and $|G| \in L^\times$, then $\Rep_{L}^{\hspace{0.1em} \sigma}(G)$ is semisimple.
\end{prop}
\begin{proof}
    The first claim follows from \cite[Rem.\ 2.5]{TAY4}, as $|I| < \infty$ and the natural homomorphism $V \rightarrow \prod_i V_i$ is an isomorphism for any $L \rtimes G$-module $V$ (because $|I| < \infty$). Then the second claim follows from \cite[Cor.\ 0.2]{MONT2}, because $|G| \in L^\times$ and any $L \rtimes G$-module $V$ has an $L$-module complement.
\end{proof}

\begin{lemma}\label{lem:homsetinjection}
    If $V,W \in \Rep_{L}^{\hspace{0.1em} \sigma}(G)$, then the natural map
    \begin{equation}\label{eqn:homsetinjection}
    L \otimes_K \Hom_{L \rtimes G}(V, W) \rightarrow \Hom_{L[H]}(V,W)
    \end{equation}
    is injective, where $K = L^G$ and $H = \ker(G \rightarrow \Aut(L/K))$. Furthermore,
    \[
    \dim_K \Hom_{L \rtimes G}(V, W) \leq \rank_L \Hom_{L[H]}(V,W),
    \]
    and (\ref{eqn:homsetinjection}) is an isomorphism if and only if this is an equality.
\end{lemma}

\begin{remark}\label{rem:weakcompatwithcomp}
    This injection is compatible with composition of morphisms in the obvious sense. In particular, when $V = W$ this injection is a ring homomorphism
    \[
    L \otimes_K \End_{L \rtimes G}(V) \hookrightarrow \End_{L[H]}(V).
    \]
\end{remark}

\begin{proof}
    Firstly, because $V$ and $W$ are free of finite rank over $L$, $N \coloneqq \Hom_L(V,W)$ is also free of finite rank over $L$. Furthermore, because $H$ acts trivially on $L$, the isomorphism
    \[
        N \xrightarrow{\sim} \prod_i e_i \cdot N,
    \]
    restricts to an isomorphism
    \[
    N^{H} \xrightarrow{\sim} \prod_i e_i \cdot N^{H}.
    \]
    In particular, $V \coloneqq N^{H}$ if free of finite rank over $L$ by \cite[Lem.\ 2.4]{TAY4}. Therefore we may apply \cite[Cor.\ 3.8]{TAY4} with the $(F, G, H)$ of loc.\ cit.\ taken to be $(L, G, G)$ to deduce the required result, noting that
    \[
    V^G = \Hom_{L[H]}(V,W)^G = \Hom_{L \rtimes G}(V, W). \qedhere
    \]
\end{proof}


\begin{lemma}\label{lem:KRS}
    Suppose that the index set $I$ of the product $L = \prod_i L_i$ is finite. Then the Krull-Remak-Schmidt Theorem holds in $\Rep_{L}^{\hspace{0.1em} \sigma}(G)$: any non-zero $V \in \Rep_{L}^{\hspace{0.1em} \sigma}(G)$ can be written as a direct sum of indecomposable objects
    \[
		V \cong V_1 \oplus \cdots \oplus V_n,
	\]
	and any such description is unique up to isomorphism and permutation.
\end{lemma}
	
\begin{proof}
	This follows from \cite[Thm.\ 1]{ATI}, using \cite[\S 3 Cor.]{ATI} and the fact that $\Hom(V ,W)$ is finite dimensional over $K$ for any $V, W \in \Rep_{L}^{\hspace{0.1em} \sigma}(G)$ because $|I| < \infty$.
\end{proof}

\section{Semilinear Representations for Galois Extensions}\label{sect:semilinearrepsGaloisext}

\subsection{Galois Extensions of a Field}

\begin{defn}
If $K$ is a field, and $L$ is a finite-dimensional $K$-algebra with an action of a finite group $\Gamma$, we say that $L / K$ is a \emph{Galois extension} with Galois group $\Gamma$ if $K = L^\Gamma$ and the natural map
\[
L \otimes_K L \rightarrow \prod_{\gamma \in \Gamma} L, \qquad a \otimes b \mapsto (a \gamma(b))_{\gamma \in \Gamma}
\]
is an isomorphism \cite[Def.\ 2.47]{TAY3}.
\end{defn}

Note that unlike the definition of a Galois field extension, the action of $\Gamma$ is part of the data of a Galois extension. However, the following shows that when $L$ is a field this essentially coincides with the notion of a Galois extension of fields.

\begin{eg}\label{eg:classicalgaloisext}
Any Galois field extension $L / K$ with its natural action of $\Gal(L/K)$ is a Galois extension with Galois group $\Gal(L/K)$. Conversely, if $L/K$ is a Galois extension with Galois group $\Gamma$ and $L$ is a field, then $L/K$ is Galois and the action map $\Gamma \rightarrow \Gal(L/K)$ is an isomorphism.
\end{eg}

\begin{eg}\label{eg:splitgalext}
For a Galois extension $L / K$, $\Spec(L)$ need not be connected, or equivalently, $L$ need not be a field. For example, for any finite group $\Gamma$ one has the split Galois extension defined by setting $K_{\gamma} = K$ for any $\gamma \in \Gamma$, with diagonal inclusion
\[
K \hookrightarrow \prod_{\gamma \in \Gamma} K_{\gamma}
\]
and action of $\Gamma$ by permutation $\sigma \cdot (\lambda_{\gamma})_{\gamma} \coloneqq (\lambda_{\sigma \gamma})_{\gamma}$. This is a Galois extension with Galois group $\Gamma$.
\end{eg}

\begin{remark}\label{rem:generalform}
Generally, if $L / K$ is a Galois extension with Galois group $\Gamma$, then for any connected component $\Spec(L_i)$ of $\Spec(L)$, $L_i / K$ is a Galois field extension with Galois group $G_i \coloneqq \Stab_{\Gamma}(\Spec(L_i))$, and the action of $\Gamma$ on $\pi_0(\Spec(L))$ is transitive \cite[Lem.\ 2.52]{TAY3}.
\end{remark}

\subsection{Assumptions and Notation}\label{sect:notation}
\emph{From now on until the end of the paper we suppose that:}
\begin{enumerate}
    \item $K$ is a field, and $K \rightarrow L$ is a Galois extension with Galois group $\Gamma$,
    \item $G$ is a group which acts on $L$ through a surjective homomorphism $\sigma \colon G \rightarrow \Gamma$ with kernel $H$.
\end{enumerate}

As $\sigma \colon G \rightarrow \Gamma$ is fixed, we suppress it from the notation and write:
\begin{itemize}
    \item $g(\lambda) \coloneqq \sigma_g(\lambda)$ for $\lambda \in L, g \in G$,
    \item $L \rtimes G \coloneqq L \rtimes_{\sigma} \! G$,
    \item $\Rep_L^{\rtimes}(G)$ for $\Rep_{L}^{\hspace{0.1em} \sigma}(G)$,
    \item $\Irr_{L}^{\rtimes}(G)$ for $\Irr_{L}^{\hspace{0.1em} \sigma}(G)$.
\end{itemize}
We also fix the following notation:
\begin{itemize}
    \item $e$ is a primitive idempotent of $L$,
    \item $F \coloneqq e \cdot L$,
    \item $\Gamma_e \coloneqq \Stab_{\Gamma}(e)$, and
    \item $G_e \coloneqq \Stab_G(e)$.
\end{itemize}

From Remark \ref{rem:generalform}, the induced action of $\Gamma_e$ on $F$ makes $F / K$ a Galois extension with Galois group $\Gamma_e$, and the action of $\Gamma$ on the connected components of $\Spec(L)$ is transitive. In particular, the tuple
\[
(G_e,H,F/K, \sigma \colon G_e \rightarrow \Gamma_e)
\]
also satisfies the above assumptions, with the additional property that $F$ is a field.

\begin{remark}\label{rem:introsatisfiesassumptions}
    Suppose momentarily that $L$ is a field and $G$ is a group that acts on $L$ through a group homomorphism $G \rightarrow \Aut(L)$ (as in the introduction). Setting $K \coloneqq L^G$, then the assumption that $L/K$ is a finite extension automatically implies that $L/K$ is Galois and that the induced map $G \rightarrow \Gal(L/K)$ is surjective (and therefore satisfies the assumptions of this section). Indeed, $K$ is the fixed field of a subgroup of the finite group $\Aut(L/K)$, so $L/K$ is Galois, and if $S \subset \Gal(L/K)$ denotes the image of $G$ then $L^S = L^G = K$, hence $S = \Gal(L/K)$ by the Galois correspondence.
\end{remark}

We also note that this setup is stable under base change, which is an important tool for our analysis, and why we don't consider only Galois extensions of fields.

More precisely, if $K' / K$ is a field extension, then the base change $L' = L \otimes_K K'$ with is a Galois extension of $K'$ with Galois group $\Gamma$, where $\Gamma$ acts through the left factor of $L \otimes_K K'$. With the induced action of $G$ on $L'$ the tuple
\[
(G, H, L'/K', \sigma \colon G \rightarrow \Gamma)
\]
also satisfies the assumptions of this section. 

\begin{lemma}\label{lem:injofbasechangeSL}
    If $K' / K$ is a field extension and $V,W \in \Rep_L^{\rtimes}(G)$, then the natural map
    \[
    K' \otimes_K \Hom_{L \rtimes G}(V, W) \rightarrow \Hom_{L' \rtimes G}(V_{L'}, W_{L'})
    \]
    is an isomorphism.
\end{lemma}

\begin{proof}
    As an $L$-module $V$ is free of finite rank, hence by \cite[Ch. 2, \S 5.3, Prop. 7(ii)]{BOUR} the natural map
    \[
        K' \otimes_K \Hom_L(V,W) \equiv L' \otimes_L \Hom_L(V,W) \rightarrow \Hom_{L'}(V_{L'}, W_{L'})
    \]
    is an isomorphism, and we are done after taking $G$-invariants.
\end{proof}

We are interested in describing the semilinear representation category $\Rep_L^{\rtimes}(G)$ in terms of associated linear representation categories, in particular $\Rep_{F}(H)$. Let us now examine these categories.

\section{Associated Linear Representation Categories}\label{sect:associatedlinearcats}

There is a diagram of $K$-algebra inclusions
\begin{equation}\label{eqn:4rings}
\begin{tikzcd}
	{K[G]} & {L \rtimes G} \\
	{K[H]} & {L[H]}
	\arrow[hook, from=1-1, to=1-2]
	\arrow[hook, from=2-1, to=1-1]
	\arrow[hook, from=2-1, to=2-2]
	\arrow[hook, from=2-2, to=1-2]
\end{tikzcd}
\end{equation}
Recall that an homomorphism of (unital associative) rings $R \rightarrow S$ is called \emph{Frobenius} if the induction
\[
S \otimes_R - \colon \Mod_R \rightarrow \Mod_S
\]
and co-induction
\[
\Hom_R(S, -) \colon \Mod_R \rightarrow \Mod_S
\]
functors are naturally isomorphic.
\begin{prop}\label{prop:allfrob}
    All inclusions in the diagram (\ref{eqn:4rings}) are Frobenius.
\end{prop}
\begin{proof}
    We use the criterion of \cite[Thm.\ 1.2]{KAD}, that $R \rightarrow S$ is Frobenius if and only if there is an $(R,R)$-bimodule homomorphism $E \colon S \rightarrow R$ and elements $x_1, ... ,x_n, y_1, ... , y_n \in S$ with 
\begin{equation}\label{eqn:frobprop}
\sum_{i = 1}^n E(sx_i)y_i = s = \sum_{i = 1}^n x_i E(y_i s)
\end{equation}
for all $s \in S$.

    To define these maps and elements in our context, let $g_1, ... ,g_n$ be a set of left coset representatives for $H$ in $G$. For $K[H] \hookrightarrow K[G]$ or $L[H] \hookrightarrow L \rtimes G$, one can take $E \colon S \rightarrow R$ to be the $K$-linear (resp.\ $L$-linear) extension of the map which has $E(h) = h$ for $h \in H$ and $E(g) = 0$ if $g \not\in H$, with $x_i = g_i$ and $y_i = g_i^{-1}$. For $K[H] \hookrightarrow L[H]$ or $K[G] \hookrightarrow L \rtimes G$, let $E \colon S \rightarrow R$ be the linear extension of $E(\lambda g) = \sum_{\gamma \in \Gamma}\gamma(\lambda) g$. The $\{x_i\}_i$ and $\{y_i\}_i$ are those given by \cite[Thm.\ 1.3(b)]{CHR}. It is direct to verify that in each case $E$ is $(R,R)$-bilinear and with the $\{x_i\}_i$ and $\{y_i\}$ satisfying the identity (\ref{eqn:frobprop}). 
\end{proof}

The induction and restriction functors between the rings of diagram (\ref{eqn:4rings}) restrict to functors between the categories
\[\begin{tikzcd}
	{\Rep_K(G)} & {\Rep_L^{\rtimes}(G)} \\
	{\Rep_K(H)} & {\Rep_L(H)}
	\arrow[tail reversed, from=1-1, to=1-2]
	\arrow[tail reversed, from=1-1, to=2-1]
	\arrow[tail reversed, from=2-1, to=2-2]
	\arrow[tail reversed, from=2-2, to=1-2]
\end{tikzcd}\]
and therefore one can naturally use induction from the linear representation categories $\Rep_K(G)$ and $\Rep_L(H)$ to construct objects of $\Rep_L^{\rtimes}(G)$.

\subsection{Semilinear Representations from $\Rep_K(G)$}\label{sect:classicalreps}

In this section we consider the most natural way to construct semilinear representations, and consider semilinear representations arising from $\Rep_K(G)$. The base change functor $(L \rtimes G) \otimes_{K[G]} -$ is naturally isomorphic to the functor
\[
L \otimes_K - \colon \Rep_K(G) \rightarrow \Rep_L^{\rtimes}(G)
\]
defined for $V \in \Rep_K(G)$ by letting $g \in G$ act diagonally on $L \otimes_K V$. The isomorphism is the composition of identifications
\[
L \otimes_K V \equiv L \otimes_K K[G] \otimes_{K[G]} V \equiv (L \rtimes G) \otimes_{K[G]} V,
\]
the composition of which is given by $\lambda \otimes v \mapsto \lambda \otimes v$. This is $L \rtimes G$-linear, as in $(L \rtimes G) \otimes_{K[G]} V$,
\[
g \cdot (\lambda \otimes v) = g(\lambda)g \otimes v = g(\lambda) \otimes g(v).
\]

Concretely, in terms of matrices (cf.\ Section \ref{sect:matrixdesc}), if $\rho \colon G \rightarrow \GL_n(K)$ corresponds to $V \in \Rep_K(G)$, then $L \otimes_K V$ has matrix representation $\rho \colon G \rightarrow \GL_n(K) \hookrightarrow \GL_n(L)$. The relations of (\ref{eqn:matrixrelation}) are satisfied because all matrices are valued in $K$.

It turns out that the functor $L \otimes_K -$ is neither fully faithful nor essentially surjective in general. However we can at least say the following.
\begin{cor}\label{cor:allsubobjects}
    For any irreducible $V \in \Rep_L^{\rtimes}(G)$,
    \[
    V \hookrightarrow L \otimes_K (V|_K),
    \]
    hence any irreducible $V \in \Rep_L^{\rtimes}(G)$ is a sub-object of an object in the essential image of $L \otimes_K -$.
\end{cor}
\begin{proof}
    Because $K[G] \hookrightarrow L \rtimes G$ is Frobenius by Proposition \ref{prop:allfrob}, $L \otimes_K -$ is right-adjoint to restriction, hence there is a $K$-linear isomorphism
    \[
        \Hom_{L \rtimes G} (V , L \otimes_K (V|_K)) \equiv 
        \Hom_{L \rtimes G} (V|_K , V|_K) \neq 0,
    \]
    and therefore, as $V$ is irreducible, $V \hookrightarrow L \otimes_K (V|_K)$.
\end{proof}

\section{Example: $\Rep_L^{\rtimes}(S_3)$ for $L / K$ of Degree $2$}\label{sect:firstexampleS3}

In this section we give an example of how one can use the functor
\[
L \otimes_K - \colon \Rep_K(G) \rightarrow \Rep_L^{\rtimes}(G)
\]
to construct semilinear representations. This example also illustrates how one practically works with semilinear matrices in terms of matrices, as described in Section \ref{sect:matrixdesc}. This example will also help to motivate the general theory we will develop in Section \ref{sect:mainresults}, using which we will rederive the results of this section more conceptually in Section \ref{sect:secondexampleS3}.

Let $L/K$ be a degree two Galois field extension, and let $S_3$ act on $L$ via the non-trivial homomorphism $\sigma \colon S_3 \twoheadrightarrow \Gal(L/K)$, which has kernel $H = \langle (123) \rangle$. Write $\alpha \colon L \rightarrow L$ for the non-trivial element of $\Gal(L/K)$. We restrict to characteristic $0$ for simplicity. 

Because $K$ has characteristic $0$, the irreducible $K$-representations of $S_3$ are the trivial representation $\rho_1$, the sign representation $\rho_2$, and the unique two-dimensional irreducible representation $\rho_3$.

\subsection{Representations from $\rho_1$ and $\rho_2$}
Let us first understand the semilinear representations $L \otimes_K \rho_1$ and $L \otimes_K \rho_2$. $L \otimes_K \rho_1$ has matrix representation $A_g = 1$ for all $g \in S_3$, and $L \otimes_K \rho_2$ has matrix representation $B_g = 1$ if $g \in \langle (123) \rangle$, and $B_g = -1$ if $g \not\in \langle (123) \rangle$. By Proposition \ref{prop:matrixdesc}, an isomorphism between these representations would be a scalar $\lambda \in L^\times$ with the property that 
\[
 B_g \cdot {}^{\sigma_g} \lambda = \lambda A_g
\]
for all $g \in G$, or in other words, with $\alpha(\lambda) = -\lambda$. This can be found, as the $K$-linear operator $\alpha \colon L \rightarrow L$ has order $2$, and $L$ splits into eigenspaces
\[
L = \{x \in L \mid \alpha(x) = x\} \oplus \{x \in L \mid \alpha(x) = - x\} = K \oplus \{x \in L \mid \alpha(x) = - x\}.
\]
Therefore, we see that the two non-isomorphic representations of $S_3$ $\rho_1$ and $\rho_2$ become isomorphic as $L$-linear representations of $S_3$. Using Proposition \ref{prop:matrixdesc} we can also compute
\[
\End_{L \rtimes S_3}(L \otimes_K \rho_1) \subset L
\]
as those $\lambda \in L$ with $A_g \cdot {}^{\sigma_g} \lambda = \lambda A_g$ for all $g \in G$, or in other words, with ${}^{\sigma_g} \lambda = \lambda$ for all $g \in G$, which, because $\sigma$ is surjective, is simply $K$.

\subsection{Decomposing $L \otimes_K \rho_3$} Now let us consider the irreducible two dimensional representation $\rho_3$. From the natural permutation representation with basis $e_1,e_2,e_3$ we may obtain $\rho_3$ as the quotient by the $L$-subspace spanned by $e_1 + e_2 + e_3$, which has matrix representation with respect to the basis $\overline{e_1}, \overline{e_2}$ given by
\[
A_{(12)} = \begin{pmatrix}
    0 & 1 \\ 1 & 0
\end{pmatrix}, \qquad A_{(123)} = \begin{pmatrix}
    0 & -1 \\ 1 & -1
\end{pmatrix}.
\]
Because $\rho_3$ has dimension $2$, if $V \coloneqq L \otimes_K \rho_3$ is not irreducible then there is some $1$-dimensional $L$-subspace preserved by the action of $G$, and so some vector $\mathbf{v} \in L^2$ with 
\[
A_g \cdot {}^{g} \mathbf{v} = \lambda_g \mathbf{v}
\]
for some $\lambda_g \in L^\times$, by Proposition \ref{prop:matrixdesc}. In particular, there are $\lambda, \mu \in L^\times$ with
\[
\begin{pmatrix} \lambda v_1 \\ \lambda v_2 \end{pmatrix}
=
A_{(12)} \cdot {}^{\alpha} \mathbf{v}
=
\begin{pmatrix}
    0 & 1 \\ 1 & 0
\end{pmatrix} \cdot {}^{\alpha} \mathbf{v}
=
\begin{pmatrix}
    \alpha(v_2) \\ \alpha(v_1) 
\end{pmatrix},
\]
and 
\[
\begin{pmatrix}
    \mu v_1 \\ \mu v_2
\end{pmatrix}
=
A_{(123)}  \mathbf{v}
=
\begin{pmatrix}
    0 & -1 \\ 1 & -1
\end{pmatrix} \mathbf{v}
=
\begin{pmatrix}
    -v_2 \\ v_1 - v_2
\end{pmatrix}.
\]
From the second pair of equations, we have that $v_2(\mu^2 + \mu + 1) = 0$, and from the first pair that $v_2 \neq 0$ (otherwise $\mathbf{v} = \mathbf{0}$). Therefore $\mu^2 + \mu + 1 = 0$ and $L$ contains a primitive third root of $1$. 

In particular, if $L$ contains no primitive third root of $1$, then $V$ is irreducible. 

\subsubsection{The case when $\omega \in K$} Suppose then, that $L$ contains $\omega$, a primitive third root of $1$. Then there are two cases. In the first case, where $\omega \in K$, we may diagonalise over $K$ the matrix of $(123)$ (which has characteristic polynomial $x^2 + x + 1$) to obtain a matrix representation of $\rho_3$ over $K$ defined by
\[
B_{(12)} = \begin{pmatrix}
    0 & 1 \\ 1 & 0
\end{pmatrix}, \qquad B_{(123)} = \begin{pmatrix}
    \omega & 0 \\ 0 & \omega^{-1}
\end{pmatrix}.
\]
In particular, similarly to above, if $V$ was reducible then one could obtain equations
\[
\begin{pmatrix}
    \alpha(v_2) \\ \alpha(v_1) 
\end{pmatrix} = \begin{pmatrix} \lambda v_1 \\ \lambda v_2 \end{pmatrix}, \qquad \begin{pmatrix}
    \omega v_1 \\ \omega^{-1} v_2
\end{pmatrix} = \begin{pmatrix} \mu v_1 \\ \mu v_2 \end{pmatrix}
\]
for some $\lambda, \mu \in L^\times$. From the first pair of equations we must have that $v_1, v_2$ are both non-zero, and therefore from the second equations we must have that $\omega = \mu = \omega^{-1}$, a contradiction. So also in this case $V$ is also irreducible. 

\subsubsection{The case when $\omega \in L$ and $\omega \not\in K$}
The most interesting case is the third, when $\omega \in L$ but $\omega \not\in K$. In this case, $L = K(\omega)$, and the automorphism $\alpha \colon L \rightarrow L$ swaps $\omega$ and $\omega^{-1}$. We can then define $\mathbf{v} = (-1, \omega)$, $\mathbf{w} = (-1, \omega^{-1})$. The $L$-span of each defines a $G$-stable subspace because
\[
A_{(12)} \cdot {}^{\alpha} \mathbf{v} = \begin{pmatrix} \alpha(\omega) \\ \alpha(-1) \end{pmatrix} = \begin{pmatrix}\omega^{-1} \\ -1 \end{pmatrix}  = - \omega^{-1} \begin{pmatrix} -1 \\ \omega \end{pmatrix} = \omega^{-1} \mathbf{v}
\]
and
\[
A_{(123)} \cdot \mathbf{v} = \begin{pmatrix}
    -\omega \\ -1 - \omega
\end{pmatrix} = \begin{pmatrix}
    -\omega \\ \omega^2
\end{pmatrix} = \omega \begin{pmatrix}
    -1 \\ \omega
\end{pmatrix} = \omega \mathbf{v},
\]
and similarly for $\mathbf{w}$. We also see that these two subrepresentations restrict to the two non-trivial (linear) representations of $H$ over $L$, and therefore must be non-isomorphic in $\Rep_L^{\rtimes}(S_3)$.

\subsection{Summary} To summarise:
\begin{itemize}
    \item $L \otimes_K \rho_1 \cong L \otimes_K \rho_2$, with endomorphism ring $K$,
    \item If there is some $\omega \in L$ which is a primitive third root of $1$ and $\omega \not\in K$, then $L \otimes_K \rho_3$ is the direct sum of two non-isomorphic $1$-dimensional irreducible representations,
    \item Otherwise $V = L \otimes_K \rho_3$ is irreducible.
\end{itemize}
By Corollary \ref{cor:allsubobjects} we know that these are all the irreducible representations, and that any object of $\Rep_L^{\rtimes}(S_3)$ is a direct sum of these irreducible representations by Proposition \ref{prop:freeandss}.

\subsection{Endomorphism Rings}
We can also compute endomorphism rings for these representations.

First note that each irreducible semilinear representation $\sigma_i$ of $S_3$ over $L$ corresponds to a matrix factor in the $K$-algebra $L \rtimes S_3$, which is semisimple by Proposition \ref{prop:freeandss}. This is of the form $M_{n_i}(D_i)$ where $D_i \coloneqq \End_{L \rtimes S_3}(\sigma_i)$. The simple module $\sigma_i$ is recovered from the action of $L \rtimes S_3$ on the first column, and so by comparing $K$-dimensions we have that
\[
n_i \cdot \dim_K D_i = [L:K]\cdot \dim_L \sigma_i.
\]
\subsubsection{The case when $\omega \in L \setminus K$.}
In the case where there is a primitive third root $\omega \in L$ but $\omega \not\in K$, we have seen that there are two other non-isomorphic representations $\sigma_2$ and $\sigma_3$ with $\sigma_2 \oplus \sigma_3 = V$. There is an injection $L \otimes_K D_i \hookrightarrow \End_{L[H]}(\sigma_i)$ from Lemma \ref{lem:homsetinjection}, and each $\End_{L[H]}(\sigma_i) = L$ as each $\sigma_i$ restricted to $L[H]$-module is absolutely irreducible. In particular, $D_2 = K = D_3$, each $n_i = n_i \dim_K D_i = [L:K]\dim_L \sigma_i = 2$, and we have a decomposition
\[
L \rtimes S_3 \cong M_2(K) \times M_2(K) \times M_2(K).
\]
\subsubsection{The case when $V$ is irreducible} In the other case, where $V$ is irreducible, the above expression reduces to
\[
4 = n_V \dim_K D_V,
\]
so $n_V = 2$, and $\dim_K D_V = 2$. In particular, $D_V$ is a degree $2$ field extension of $K$ and
\[
L \rtimes S_3 \cong M_2(K) \times M_2(D_V).
\]

We can be more precise about $D_V$, using the injection $L \otimes_K D_V \hookrightarrow \End_{L[H]}(V)$ from Lemma \ref{lem:homsetinjection}. If $\omega \in L$ (and hence $\omega \in K$), then $\End_{L[H]}(V) = L \times L$, so $L \otimes_K D_V \xrightarrow{\sim} L \times L$ and thus $D_V \cong L$ by \cite[Thm.\ 2.2]{COHN}. Conversely, if $\omega \not\in L$, then $\End_{L[H]}(V) = L(\omega)$, and all we can say in this level of generality is that $L \otimes_K D_V \xrightarrow{\sim} L(\omega)$.

\section{Split Extensions}\label{sect:splitgroupext}

In Section \ref{sect:mainresults} we will describe the category $\Rep_L^{\rtimes}(G)$ in terms of $\Rep_{F}(H)$. Before this, we describe one situation in which we can already give a complete description of $\Rep_L^{\rtimes}(G)$, namely when the sequence
\[
1 \rightarrow H \rightarrow G \rightarrow \Gamma \rightarrow 1
\]
is left-split, so that $G \cong H \times \Gamma$ and $H \rightarrow G$ and $G \rightarrow \Gamma$ are the natural inclusion and projection.
    \begin{prop}\label{prop:galoisdescent}
        Suppose that $H$ is a group and $L / K$ is a Galois extension with Galois group $\Gamma$. Then the functors 
    \begin{align*}
        L \otimes_K - &\colon \Rep_K(H) \rightarrow \Rep_L^{\rtimes}(H \times \Gamma), \\
        (-)^\Gamma &\colon \Rep_L^{\rtimes}(H \times \Gamma) \rightarrow \Rep_K(H)
    \end{align*}
        define mutually inverse equivalences of categories.
    \end{prop}
    \begin{proof}
        This is just Galois descent, see for example \cite[Prop.\ 2.53]{TAY3}.
    \end{proof}
    
The following lemma shows that under these identifications, \emph{restriction} from $\Rep_L^{\rtimes}(H \times \Gamma)$ to $\Rep_L(H)$ corresponds to \emph{induction} from $\Rep_K(H)$ to $\Rep_L(H)$.

    \begin{lemma}\label{lem:classicalinduction}
        Suppose that $H$ is a group and $L / K$ is a Galois extension with Galois group $\Gamma$. Then the diagram
\[\begin{tikzcd}
	{\Rep_K(H)} && {\Rep_L^{\rtimes}(H \times \Gamma)} \\
	& {\Rep_L(H)}
	\arrow["{L \otimes_K -}", shift left, from=1-1, to=1-3]
	\arrow["{(-)_L}"', from=1-1, to=2-2]
	\arrow["{(-)^\Gamma}", shift left, from=1-3, to=1-1]
	\arrow["{(-)|_H}", from=1-3, to=2-2]
\end{tikzcd}\]
        commutes up to natural isomorphism, where $H \times \Gamma$ acts on $L$ through the projection to $\Gamma$.
    \end{lemma}

    \begin{proof}
         For $V \in \Rep_K(H)$, the two actions of $L[H]$ on $L \otimes_K V$ coincide.
    \end{proof}
    
    Similarly, \emph{induction} from $\Rep_L(H)$ to $\Rep_L^{\rtimes}(G)$ corresponds to \emph{restriction} from $\Rep_L(H)$ to $\Rep_K(H)$.

    \begin{lemma}\label{lem:classicalrestriction}
        Suppose that $H$ is a group and $L / K$ is a Galois extension with Galois group $\Gamma$. Then the diagram
\[\begin{tikzcd}
	{\Rep_K(H)} && {\Rep_L^{\rtimes}(H \times \Gamma)} \\
	& {\Rep_L(H)}
	\arrow["{L \otimes_K -}", shift left, from=1-1, to=1-3]
	\arrow["{(-)^\Gamma}", shift left, from=1-3, to=1-1]
	\arrow["{\Res^L_K -}", from=2-2, to=1-1]
	\arrow["{\Ind_{H}^{H \times \Gamma} -}"', from=2-2, to=1-3]
\end{tikzcd}\]
        commutes up to natural isomorphism, where $H \times \Gamma$ acts on $L$ through the projection to $\Gamma$.
    \end{lemma}

\begin{proof}
        Suppose that $V \in \Rep_L(H)$. We have $K[H]$-linear isomorphisms
        \[
            L \otimes_K V \: \xlongrightarrow{\sim} \: \prod_{\gamma \in \Gamma} V \: \xlongrightarrow{\sim} \: (L \rtimes \Gamma) \otimes_L V \: \xlongrightarrow{\sim} \: L \rtimes (H \times \Gamma) \otimes_{L[H]} V,
        \]
        defined by $ \lambda \otimes v \mapsto (\gamma(\lambda)v)_{\gamma}$, $(v_{\gamma})_{\gamma} \mapsto \gamma^{-1} \otimes v_{\gamma}$ and $x \otimes v \mapsto x \otimes v$ respectively, where the first map comes from applying $- \otimes_L V$ to the defining isomorphism of the Galois extension $L / K$. It is direct to check that this composition commutes with the action of $L$ and $\Gamma$, and thus is $L \rtimes (H \times \Gamma)$-linear.
\end{proof}

\section{Connected Components}\label{sect:connectedcomponents}

In this section, we show that the study of $\Rep_L^{\rtimes}(G)$ reduces to the case that $L$ is connected. The results of this section are used crucially in Section \ref{sect:mainresults} when we describe the category $\Rep_L^{\rtimes}(G)$ in terms of $\Rep_{F}(H)$ and prove Theorems A and B.

Recall that $e$ denotes a primitive idempotent of $L$, $F = e \cdot L$, and $G_e = \Stab_G(e)$. From Remark \ref{rem:generalform}, the induced action of $G_e$ on $F$ makes $F / K$ a Galois extension with Galois group $G_e$, and the action of $G$ on the connected components of $\Spec(L)$ is transitive.

There is a functor
\[
    e \cdot (-) \colon \Rep_L^{\rtimes}(G) \rightarrow \Rep_{F}^{\rtimes}(G_e),
\]
and a functor in the other direction defined by
\[
    (L \rtimes G) \otimes_{L \rtimes G_e} - \colon \Rep_{F}^{\rtimes}(G_e) \rightarrow \Rep_{L}^{\rtimes}(G),
\]
where we view any $F \rtimes G_e$-module as a $L \rtimes G_e$-module via the natural projection $L \rtimes G_e \rightarrow F \rtimes G_e$.

\begin{prop}\label{prop:connequiv}
    The functors 
\begin{align*}
    e \cdot (-) &\colon \Rep_L^{\rtimes}(G) \rightarrow \Rep_{F}^{\rtimes}(G_e), \\
    (L \rtimes G) \otimes_{L \rtimes G_e} - &\colon \Rep_{F}^{\rtimes}(G_e) \rightarrow \Rep_{L}^{\rtimes}(G)
\end{align*}
    define mutually inverse equivalences of categories.
\end{prop}

\begin{proof}
    Suppose first that $V \in \Rep_L^{\rtimes}(G)$ and define 
    \[
        \phi \colon (L \rtimes G) \otimes_{L \rtimes G_e} e \cdot V \rightarrow V, \qquad \phi(\lambda g \otimes v) = \lambda g(v).
    \]  
    To see that this gives a well-defined map out of the tensor product, let $\mu h \in L \rtimes G_e$ and note that
    \[
        \phi((\lambda g)(\mu h) \otimes v) = \phi(\lambda g(\mu) gh \otimes v) = \lambda g(\mu) (gh)(v),
    \]    
    which, noting that $e \cdot v = v$ and hence
    \[
    (gh)(v) = (gh)(e \cdot v) = g(h(e) \cdot h(v)) = g(e \cdot h(v)) = g(e) \cdot (gh)(v),
    \]
    is equal to
    \[
        \phi(\lambda g \otimes (e \cdot \mu) h(v)) = \lambda g(e \cdot \mu) (gh)(v) = \lambda g(e) g(\mu) (gh)(v).
    \]    
    It is direct to see that $\phi$ is $L \rtimes G$-linear. Furthermore, $e \cdot V$ has $K$-dimension $\dim_K(V) / [L : F]$, so the left-hand side has $K$-dimension $[G:G_e] \dim_K(V) / [L : F] = \dim_K(V)$ and it is sufficient to show that $\phi$ is surjective. For this, let $g_i$ be a set of left coset representatives of $G_e$ in $G$. Because the action of $G$ on $\pi_0(\Spec(L))$ is transitive, the primitive idempotents of $L$ are $g_i(e)$. Any $v \in V$ can be written as $\sum_i g_i(e)v$, which is in the image of $\phi$ as $\phi(g_i \otimes e \cdot g_i^{-1}(v)) = g_i(e) v$.

    Suppose now that $W \in \Rep_{F}^{\rtimes}(G_e)$, and define 
    \[
        \psi \colon W \rightarrow e \cdot (L \rtimes G \otimes_{L \rtimes G_e} W), \qquad \psi(w) = e \otimes w,
    \] 
    which easily seen to be $F \times G_e$-linear. Similarly to above, to show this is an isomorphism it is sufficient to show that $\psi$ is surjective. To see this, note that any element of $L \rtimes G \otimes_{L \rtimes G_e} W$ is a sum of elements of the form $\lambda g \otimes w$ for $\lambda \in L$, $g \in G$, $w \in W$, and on these
    \[
    e \cdot (\lambda g \otimes w) = e \lambda g \otimes w = g g^{-1}(e) g^{-1}(\lambda) \otimes w = g \otimes e g^{-1}(e) g^{-1}(\lambda) w,
    \] 
    which is zero whenever $g \not\in G_e$, as in this case $e g^{-1}(e) = 0$.
\end{proof}

For example, this allows us to give a complete description of $\Rep_L^{\rtimes}(G)$ in certain cases.

\begin{defn}
The extension $L / K$ is called \emph{split} if there is a $K$-algebra homomorphism $L \rightarrow K$.
\end{defn}

This can be characterised in other ways. The following are equivalent:
\begin{itemize}
    \item $L / K$ is split,
    \item $H = G_e$,
    \item The inclusion $K \rightarrow F$ is an isomorphism,
    \item $L$ is the product of copies of $K$,
    \item $\dim_K(L) = |\pi_0(\Spec(L))|$.
\end{itemize}

Split extensions are therefore essentially all of the form as described in Example \ref{eg:splitgalext}. For such extensions one obtains a complete description of $\Rep_L^{\rtimes}(G)$, as in this case $\Rep_{F}^{\rtimes}(G_e) = \Rep_K(H)$ is a linear representation category.

Therefore, in summary, together with the results of Section \ref{sect:splitgroupext} we see that when either the group extension $1 \rightarrow H \rightarrow G \rightarrow \Gamma \rightarrow 1$ or the Galois extension $L / K$ is split we have an equivalence from $\Rep_L^{\rtimes}(G)$ to $\Rep_K(H)$. The following shows that if both are true, then these are essentially the same equivalence.

\begin{prop}
Suppose that $L/K$ is split Galois extension with Galois group $\Gamma$, $H$ is a group, and $H \times \Gamma$ acts on $L$ through the projection to $\Gamma$. Then the diagram
\[\begin{tikzcd}
	& {\Rep_L^{\rtimes}(H \times \Gamma)} \\
	{\Rep_K(H)} && {\Rep_{F}(H)}
	\arrow["{(-)^{\Gamma}}"', from=1-2, to=2-1]
	\arrow["{e \cdot (-)}", from=1-2, to=2-3]
	\arrow["{F \otimes_K -}"', from=2-1, to=2-3]
	\arrow["\sim", from=2-1, to=2-3]
\end{tikzcd}\]
commutes up to natural isomorphism.
\end{prop}

\begin{proof}
    For $V \in \Rep_L^{\rtimes}(H \times \Gamma)$, this is defined by
    \[
        F \otimes_K V^{\Gamma} \rightarrow e \cdot V, \qquad \lambda \otimes v \mapsto \lambda v.
    \]
    which is directly seen to be $K[H]$-linear. Both sides have the same $F$-dimension, and therefore this map is an isomorphism as it is surjective: for $v \in e \cdot V$,
    \[
    e \otimes \left( \sum_{\gamma \in \Gamma} \gamma(v) \right) \mapsto \sum_{\gamma \in \Gamma} e \gamma(v) = \sum_{\gamma \in \Gamma} e \gamma(e) v = v,
    \]
    using that $e \cdot v = v$ and that $e \gamma(e) = 0$ for $\gamma \neq 1$.
\end{proof}

\section{Restriction and Induction between $\Rep_L^{\rtimes}(G)$ and $\Rep_{F}(H)$}\label{sect:mainresults}

In this section we come to the heart of the paper and study restriction and induction between $\Rep_L^{\rtimes}(G)$ and $\Rep_{F}(H)$, which we use to give a description of $\Irr_L^{\rtimes}(G)$ in terms of $\Irr_{F}(H)$. We continue with our running notation, where $e$ denote a primitive idempotent of $L$, $F = e \cdot L$, and $\Gamma_e = \Stab_{\Gamma}(e)$.

The restriction and induction functors we are interested in are:
\begin{align*}
    e \cdot (-) &\colon \Rep_{L}^{\rtimes}(G) \rightarrow \Rep_{F}(H), \\
    (L \rtimes G) \otimes_{L[H]} - &\colon \Rep_{F}(H) \rightarrow \Rep_{L}^{\rtimes}(G).
\end{align*}
Here any $V \in \Rep_{F}(H)$ is viewed as an $L[H]$-module via the projection $L[H] \rightarrow F[H]$ and then $L \rtimes G$ acts on the left of $(L \rtimes G) \otimes_{L[H]} V$.

\subsection{Base Change}
The main idea is to interpret these functors in terms of base change.

\begin{defn}
To ease notation, we set
\[
R \coloneqq \prod_{\gamma \in \Gamma} F.
\]
\end{defn}

We consider $R$ with an action of of $\Gamma$ by
\[
\sigma * (\lambda_{\gamma})_{\gamma} \coloneqq (\lambda_{\sigma^{-1} \gamma})_{\gamma} \ \mbox{  for all  } \ \sigma \in \Gamma, \ (\lambda_{\gamma})_{\gamma} \in R
\]
and action of $\Gamma_e$ by
\[
\sigma *  (\lambda_{\gamma})_{\gamma} \coloneqq (\sigma(\lambda_{\gamma \sigma}))_{\gamma} \ \mbox{  for all  } \ \sigma \in \Gamma_e, \ (\lambda_{\gamma})_{\gamma} \in R.
\]
These actions commute.

We also consider $F \otimes_K L$ with its natural action of $\Gamma_e \times \Gamma$, where $(\gamma_e, \gamma)$ acts as $\gamma_e \otimes \gamma$.

\begin{lemma}\label{lem:basechangedesc}
    The map of $F$-algebras
    \[
    F \otimes_K L \rightarrow R, \qquad \lambda \otimes \mu \mapsto (\lambda \gamma^{-1}(\mu))_{\gamma \in \Gamma},
    \]
    is a $\Gamma_e \times \Gamma$-equivariant isomorphism.
\end{lemma}

\begin{proof}
    This is an isomorphism, being the composition of the chain of isomorphisms
    \[
        F \otimes_K L \rightarrow F \otimes_K \left( \prod_{\alpha \in \Gamma / \Gamma_e} e_{\alpha} \cdot L \right) \rightarrow \prod_{\alpha \in \Gamma / \Gamma_e} F \otimes_K F \rightarrow \prod_{\alpha \in \Gamma / \Gamma_e} \prod_{\beta \in \Gamma_e} F \equiv \prod_{\gamma \in \Gamma} F.
    \]
    making sure to choose a consistent choice of representatives $\alpha \in \Gamma / \Gamma_e$. Here $e_{\alpha} = \alpha(e)$, and we use the isomorphisms $\alpha^{-1} \colon e_{\alpha} \cdot L \rightarrow e \cdot L = F$. It is direct to check that it is $\Gamma_e \times \Gamma$-equivariant with respect to the action of $\Gamma_e \times \Gamma$ on $R$ described above.
\end{proof}

Letting $G$ act on $F \otimes_K L$ and $R$ through the action of $\Gamma$, this induces an isomorphism
\[
F \otimes_K (L \rtimes G) \xrightarrow{\sim} (F \otimes_K L) \rtimes G \xrightarrow{\sim} R \rtimes G.
\]
In this setting, the tuple $(G,H, R/F, \sigma \colon G \rightarrow \Gamma)$ satisfies the hypothesis of Section \ref{sect:notation}. A primitive idempotent of $R$ is given by $f \coloneqq (\delta_{\gamma, 1})_{\gamma}$. This is a split Galois extension, $f \cdot R = F$, and $H = \Stab_{G}(f)$. In particular, from Proposition \ref{prop:connequiv} we have an equivalence of categories 
\begin{align*}
    f \cdot (-) &\colon \Rep_{R}^{\rtimes}(G) \rightarrow \Rep_{F}(H), \\
    (R \rtimes G) \otimes_{R[H]} - &\colon \Rep_{F}(H) \rightarrow \Rep_{R}^{\rtimes}(G).
\end{align*}
We also have base change and restriction functors
\begin{align*}
    F \otimes_K - &\colon \Rep_{L}^{\rtimes}(G) \rightarrow \Rep_{F \otimes_K L}^{\rtimes}(G), \\
    \Res_{L \rtimes G}^{(F \otimes_K L) \rtimes G} - &\colon \Rep_{F \otimes_K L}^{\rtimes}(G) \rightarrow \Rep_{L}^{\rtimes}(G).
\end{align*}
Our next results allow us to interpret our restriction functor $e \cdot (-) \colon \Rep_{L}^{\rtimes}(G) \rightarrow \Rep_{F}(H)$ as this base change functor $F \otimes_K -$, and our induction functor $(L \rtimes G) \otimes_{F[H]} - \colon \Rep_{F}(H) \rightarrow \Rep_{L}^{\rtimes}(G)$ as this restriction functor.

\begin{prop}\label{prop:trivialisationind}
    The diagram
\[\begin{tikzcd}[column sep=3.15em]
	{\Rep_L^{\rtimes}(G)} & {\Rep_{F \otimes_K L}^{\rtimes}(G)} \\
	{\Rep_{F}(H)} & {\Rep_{R}^{\rtimes}(G)}
	\arrow["{\Res_{L \rtimes G}}"', from=1-2, to=1-1]
	\arrow["{(L \rtimes G) \otimes_{L[H]} -}", from=2-1, to=1-1]
	\arrow["{(R \rtimes G) \otimes_{R[H]} -}"', from=2-1, to=2-2]
	\arrow["\sim", from=2-1, to=2-2]
	\arrow["\sim"', from=2-2, to=1-2]
\end{tikzcd}\]
commutes up to natural transformation.
\end{prop}

\begin{proof}
    Suppose that $W \in \Rep_{F}(H)$. The induced $L \rtimes G$-module is defined by first viewing $W$ as a $L[H]$-module via the projection $L[H] \rightarrow F[H]$, and then letting $L \rtimes G$ act on the left of $(L \rtimes G) \otimes_{L[H]} W$. On the other hand, we obtain a $R \rtimes G$-module by letting $R[H]$ act on $W$ through the projection $R[H] \rightarrow F[H]$ and letting $R \rtimes G$ act on the left of $(R \rtimes G) \otimes_{R[H]} W$. There is a natural map
    \[
        (L \rtimes G) \otimes_{L[H]} W \rightarrow (R \rtimes G) \otimes_{R[H]} W
    \]
    induced by the inclusions $L[H] \hookrightarrow R[H]$ and $L \rtimes G \hookrightarrow R \rtimes G$, which is $L \rtimes G$-linear. To see that it is an isomorphism, first note that it is surjective, as any pure tensor of $r g \otimes w \in (R \rtimes G) \otimes_{R[H]} W$ can be expressed as
    \[
    r g \otimes w = g g^{-1}(r) \otimes w = g \otimes g^{-1}(r) w,
    \]
    and thus lies in the image. There it is an isomorphism, being $K$-linear and both modules having $K$-dimension $|\Gamma_e| \dim_K W$.
\end{proof}

\begin{prop}\label{prop:trivialisationres}
    The diagram
\[\begin{tikzcd}[column sep=3.15em]
	{\Rep_L^{\rtimes}(G)} & {\Rep_{F \otimes_K L}^{\rtimes}(G)} \\
	{\Rep_{F}(H)} & {\Rep_{R}^{\rtimes}(G)}
	\arrow["{F \otimes_K -}", from=1-1, to=1-2]
	\arrow["{e \cdot(-)}"', from=1-1, to=2-1]
	\arrow["\sim"', tail reversed, no head, from=2-2, to=1-2]
	\arrow["{f \cdot (-)}", from=2-2, to=2-1]
	\arrow["\sim"', from=2-2, to=2-1]
\end{tikzcd}\]
commutes up to natural transformation.
\end{prop}

\begin{proof}
    Suppose that $V \in \Rep_L^{\rtimes}(G)$. There is a well-defined action of $(F \otimes_K L) \rtimes G$ on $(F \otimes_K L) \otimes_L V$ defined by $x g \cdot (y \otimes v) \coloneqq x g(y) \otimes g(v)$, for which is natural identification
    \[
        F \otimes_K V \xrightarrow{\sim} (F \otimes_K L) \otimes_L V
    \]
    is $F \otimes_K (L \rtimes G)$-equivariant, where $F \otimes_K (L \rtimes G)$ acts on $(F \otimes_K L) \otimes_L V$ via the isomorphism $F \otimes_K (L \rtimes G) \xrightarrow{\sim} (F \otimes_K L) \rtimes G$. Then under the above functor diagram, $V$ maps to the $R \rtimes G$-module $R \otimes_L V$, where similarly $R \rtimes G$ acts by $x g \cdot (y \otimes v) \coloneqq x g(y) \otimes g(v)$. Then
    \[
    f \cdot (R \otimes_L V) = (f \cdot R) \otimes_L V = F \otimes_L V \xrightarrow{\sim} e \cdot V,
    \]
    as $F[H]$-modules, where the last isomorphism is defined by $\lambda \otimes v \mapsto \lambda v$. The map $L \rightarrow R$ is defined by $\mu \mapsto (e\gamma^{-1}(\mu))_{\gamma}$, and so the map to $L \rightarrow f \cdot R = F$ is $\mu \mapsto e \cdot \mu$, so this map is well-defined, and an isomorphism, having inverse $v \mapsto e \otimes v$.
\end{proof}

\subsection{Consequences}
This allows us to transport results about base change functors to our setting.

\begin{cor}\label{cor:indres}
For any $V \in \Rep_L^{\rtimes}(G)$,
\[
(L \rtimes G) \otimes_{L[H]} e \cdot V \cong V^{\oplus |\Gamma_e|}.
\]
\end{cor}

\begin{proof}
    This follows from the commutativity of the diagrams of Propositions \ref{prop:trivialisationind} and \ref{prop:trivialisationres}, and the fact that for a $K$-algebra $A$, and $V \in \Mod_A$, $F \otimes_K V \cong V^{\oplus |\Gamma_e|}$ as $A$-modules.
\end{proof}

\begin{cor}\label{cor:detectisom}
    $V \cong W$ in $\Rep_L^{\rtimes}(G)$ if and only if $e \cdot V \cong e \cdot W$ in $\Rep_{F}(H)$. 
\end{cor}

\begin{proof}
    This follows directly from Corollary \ref{cor:indres} and Lemma \ref{lem:KRS}.
\end{proof}

\begin{cor}\label{cor:homsetisom}
For $V,W \in \Rep_L^{\rtimes}(G)$, the natural map
\[
F \otimes_{K} \Hom_{L \rtimes G}(V,W) \rightarrow \Hom_{F[H]}(e \cdot V,e \cdot W) 
\]
is an isomorphism.
\end{cor}

\begin{remark}
This isomorphism is compatible with composition of morphisms in the obvious sense. In particular, when $V = W$ this map is a ring isomorphism
    \[
    F \otimes_K \End_{L \rtimes G}(V) \xrightarrow{\sim} \End_{F[H]}(e \cdot V).
    \]
\end{remark}

\begin{proof}
    We have that 
    \[
    F \otimes_{K} \Hom_{L \rtimes G}(V,W) \rightarrow \Hom_{F[H]}(e \cdot V,e \cdot W),
    \]
    corresponds under the $K$-linear equivalences of Proposition \ref{prop:trivialisationres} to the natural map
    \[
        F \otimes_K \Hom_{L \rtimes G}(V,W) \xrightarrow{\sim} \Hom_{(L \rtimes G)_{F}}(F \otimes_K V, F \otimes_K W),
    \]
    which is an isomorphism by \cite[Lem.\ 29.5]{CR}.
\end{proof}

\begin{remark}\label{rem:weakhomsetisom}
    Using this we can also show that the natural map
    \[
    L \otimes_{K} \Hom_{L \rtimes G}(V,W) \rightarrow \Hom_{L[H]}(V,W), \qquad \lambda \otimes f \mapsto \lambda f
    \]
    of Lemma \ref{lem:homsetinjection} is also an isomorphism, as this factors as
    \begin{align*}
        L \otimes_{K} \Hom_{L \rtimes G}(V,W) &\xrightarrow{\sim} L \otimes_{F} \Hom_{F[H]}(e \cdot V,e \cdot W), \\
        &\xrightarrow{\sim}  \Hom_{L[H]}(L \otimes_{F} e \cdot V, L \otimes_{F} e \cdot W), \\
        &\xrightarrow{\sim} \Hom_{L[H]}(V,W),
    \end{align*}
    using the $L[H]$-module isomorphisms $L \otimes_{F} e \cdot V \xrightarrow{\sim} V$, $L \otimes_{F} e \cdot W \xrightarrow{\sim} W$.
\end{remark}

 \begin{eg}\label{eg:basechangeisom}
    Suppose that $L$ is connected, so $L = F$. Then by Lemma \ref{lem:classicalinduction} we see that Corollary \ref{cor:homsetisom} generalises the fact that for any group $H$ the map
    \[
        L \otimes_K \Hom_{K[H]}(V, W) \rightarrow \Hom_{L[H]}(L \otimes_K V, L \otimes_K W).
    \]
    is an isomorphism for $V,W \in \Rep_K(H)$.
\end{eg}

\begin{remark}
    Corollary \ref{cor:indres} and Corollary \ref{cor:homsetisom} generalise \cite[Prop.\ 2.9]{RUMTAY} and \cite[Thm.\ 5.1]{RUMTAY} respectively, which consider the special case where $L / K = \bC / \bR$ and $G$ is finite.
\end{remark}

Corollary \ref{cor:homsetisom} also has implications for the base change functor $L \otimes_K -$ considered in Section \ref{sect:classicalreps}.

\begin{cor}\label{cor:dimequal}
    For $V, W \in \Rep_K(G)$, we have an equality
    \[
    \dim_K\Hom_{L \rtimes G}(L \otimes_K V, L \otimes_K W) = \dim_K\Hom_{K[H]}(V,W).
    \]
\end{cor}

\begin{remark}
    We stress that this is only an equality of dimensions - there is only a canonical map between these spaces after we tensor with $L$ over $K$, and not before.
\end{remark}

\begin{proof}
    From Remark \ref{rem:weakhomsetisom} and Example \ref{eg:basechangeisom} we have isomorphisms
    \[
    L \otimes_K \Hom_{L \rtimes G}(L \otimes_K V, L \otimes_K W)
        \xrightarrow{\sim} \Hom_{L[H]} (L \otimes_K V, L \otimes_K W) \xleftarrow{\sim} L \otimes_K \Hom_{K[H]}(V, W). \qedhere
    \]
\end{proof}

\subsection{Preservation of Semisimplicity} Suppose now that $A$ is a $K$-algebra, and $M$ is a simple $A$-module which is finite dimensional over $K$. Then because $M$ is finite dimensional, the $K$-algebra homomorphism $A \rightarrow \End_K(M)$ factors as
\[
A \twoheadrightarrow B \hookrightarrow \End_K(M),
\]
where $B$ is the image of $A$, a finite dimensional $K$-algebra. Note that because $M$ is a simple $B$-module, $B$ is a simple $K$-algebra by \cite[Prop.\ 3.31]{CR1}. Furthermore, because $F / K$ is separable, $F \otimes_K M$ is a semisimple $F \otimes_K B$-module by \cite[Cor.\ 7.8(ii)]{CR2} and hence a semisimple $F \otimes_K A$-module.

Conversely, if $N$ is a simple $F \otimes_K A$-module, which is finite dimensional over $K$, then writing $B$ for the image of $A$ in $\End_K(N)$, $B$ is a finite dimensional $K$-algebra with $J(B)$ acting trivially on $N$ (because $J(F \otimes_K B) = F \otimes_K J(B)$ \cite[Thm.\ 7.9]{CR1} and $N$ is a simple $F \otimes_K B$-module), hence $N$ is a semisimple $B$-module, and thus a semisimple $A$-module.

\begin{cor}\label{cor:presss}
The functors
\begin{align*}
    e \cdot (-) &\colon \Rep_{L}(G) \rightarrow \Rep_{F}(H), \\
    (L \rtimes G) \otimes_{L[H]} - &\colon \Rep_{F}(H) \rightarrow \Rep_{L}(G)
\end{align*}
preserve semisimplicity.
\end{cor}

\begin{proof}
    Taking $A = L \rtimes G$, this follows from the above discussion and the commutativity of the diagrams of Propositions \ref{prop:trivialisationind} and \ref{prop:trivialisationres}, the equivalences of which preserve semisimplicity.
\end{proof}

\subsection{Action of $\Gamma_e$}

In this section we consider the two natural actions of the group $\Gamma_e$ on $\Rep_{R}^{\rtimes}(G)$ and $\Rep_{F}(H)$, and show that these agree under the equivalences Propositions \ref{prop:trivialisationind} and \ref{prop:trivialisationres}.

\subsubsection{Action of $\Gamma_e$ on $\Rep_{R}^{\rtimes}(G)$}

For any $\sigma \in \Gamma_e$, we have a $K$-algebra isomorphism
\[
\sigma \otimes 1 \colon F \otimes_K (L \rtimes G) \rightarrow F \otimes_K (L \rtimes G).
\]
Under the isomorphism $F \otimes_K (L \rtimes G) \xrightarrow{\sim} R \rtimes G$ induced from Lemma \ref{lem:basechangedesc}, this corresponds to
\[
\sigma \colon R \rtimes G \rightarrow R \rtimes G, \qquad (\lambda_{\gamma})_{\gamma} \cdot g \mapsto (\sigma(\lambda_{\gamma \sigma}))_{\gamma} \cdot g.
\]
In particular, we obtain a left action on $\Rep_{R}^{\rtimes}(G) / \! \cong$, the set of isomorphism classes of objects:
\begin{defn}
$M \in \Rep_{R}^{\rtimes}(G)$, ${}^\sigma M$ is $M$ but with $R \rtimes G$ action through $\sigma^{-1} \colon R \rtimes G \rightarrow R \rtimes G$.
\end{defn}

\subsubsection{Action of $\Gamma_e$ on $\Rep_{F}(H)$}

Recall $\Gamma_e = \Stab_{\Gamma}(e)$ and $G_e = \sigma^{-1}(\Gamma_e)$, where $\sigma$ denotes the surjection $\sigma \colon G \rightarrow \Gal(L/K)$. We have a left action of $G_e$ on $\Rep_{F}(H) / \! \cong$ defined as follows.

\begin{defn}\label{def:Gaction}
    For $g \in G_e$ and $W \in \Rep_{F}(H)$, $g * W \in \Rep_{F}(H)$ is $W$ as an abelian group, with
    \begin{itemize}
        \item $h * w = (g^{-1}hg)(w)$,
        \item $\lambda * w = \sigma_{g^{-1}}(\lambda) w$,
    \end{itemize}
    for $\lambda \in F$, $h \in H$ and $w \in W$.
\end{defn}
Note that the induced action on isomorphism classes of objects is trivial when restricted to $H$, and thus becomes an action of the finite group $G_e / H$ and therefore of $\Gamma_e$ on $\Rep_{F}(H) / \! \cong$ using the isomorphism $G_e / H \xrightarrow{\sim} \Gamma_e$.

Using this action, we can describe the other composition of $L \rtimes G \otimes_{F[H]} -$ and $e \cdot (-)$. We could do this using Propositions \ref{prop:trivialisationind} and \ref{prop:trivialisationres}, but it is easier to see directly.

\begin{prop}\label{prop:resind}
    For any $W \in \Rep_{F}(H)$ there is an isomorphism
    \[
        e \cdot ((L \rtimes G) \otimes_{L[H]} W) \xrightarrow{\sim} \bigoplus_{g \in G_e/H} \: g * W.
    \]
\end{prop}

\begin{proof}
    Pick coset representatives $g_i$ for $G/H$. First, $L \rtimes G = \oplus_i \: g_i L[H]$ as an $L$-module, hence
    \[
    (L \rtimes G) \otimes_{L[H]} W = \bigoplus_i g_i L[H] \otimes_{L[H]} W.
    \]
    If $g_i \not\in G_e$, then $\sigma_{g_i^{-1}}(e)$ is sent to zero under $L \rightarrow F$, and so $e \cdot (g_i L[H] \otimes_{L[H]} W) = 0$ as
    \[
    e \cdot g_i x \otimes w = g_i \sigma_{g_i^{-1}}(e) x \otimes w = g_i x \otimes 0 = 0.
    \]
    Therefore
    \[
        e \cdot ((L \rtimes G) \otimes_{L[H]} W) = \bigoplus_{g_i \in G_e} g_i L[H] \otimes_{L[H]} W.
    \]
    Each summand $g_i L[H] \otimes_{L[H]} W$ is stable under the action of $F[H]$, and for $h \in H$, $\lambda \in F$ and $w \in W$,
    \begin{align*}
        \lambda (g_i \otimes w) &= \lambda g_i \otimes w = g_i \sigma_{g_i^{-1}}(\lambda)\otimes w = g_i \otimes \sigma_{g_i^{-1}}(\lambda)w, \\
        h (g_i \otimes w) &= hg_i \otimes w = g_i (g_i^{-1} h g_i) \otimes w = g_i \otimes (g_i^{-1} h g_i)w,
    \end{align*}
    which is exactly the action of $F[H]$ on $g_i * W$.
\end{proof}

The next lemma shows that the natural action of of $\Gamma_e$ on $\Rep_R^{\rtimes}(G) / \cong$ agrees under the equivalences of Propositions \ref{prop:trivialisationind} and \ref{prop:trivialisationres} with the action of $\Gamma_e$ on $\Rep_F(H) / \cong$ described above. We will use this later in Proposition \ref{prop:descofcentre}, to describe the endomorphism ring of irreducible objects of $\Rep_L^{\rtimes}(G)$.

\begin{lemma}\label{lem:groupsactionssame}
    For any $g \in G_e$ and $W \in \Rep_{F}(H)$,
    \[
    \begin{matrix}
        \phi : & {}^{\sigma_g}( (R \rtimes G) \otimes_{R[H]} W ) &\rightarrow & (R \rtimes G) \otimes_{R[H]} g * W \\
        & x \otimes w & \mapsto & \sigma_g(x) \cdot g^{-1} \otimes w
    \end{matrix}
    \]
    is a natural isomorphism in $\Rep_R^{\rtimes}(G)$.
\end{lemma}
\begin{proof}
    To see that this induces a well-defined map on the tensor product, let $\mu h \in R[H]$, which acts on $W$ by $w \mapsto \mu_1 h(w)$. Then 
    \[
        x \otimes \mu_1 h(w) \mapsto \sigma_g(x) \cdot g^{-1} \otimes \mu_1 h(w).  
    \]
    On the other hand,
    \begin{align*}
        x \cdot \mu h \otimes w
        &\mapsto \sigma_g(x) \cdot \sigma_g(\mu) h \cdot g^{-1} \otimes w, \\
        &= \sigma_g(x) \cdot \sigma_g(\mu) g^{-1}  \cdot g h g^{-1} \otimes w, \\
        &= \sigma_g(x) \cdot g^{-1} \cdot g(\sigma_g(\mu)) \cdot g h g^{-1} \otimes w, \\
        &= \sigma_g(x) \cdot g^{-1} \cdot g(\sigma_g(\mu_{\gamma \sigma_g}))_{\gamma} \cdot g h g^{-1} \otimes w, \\
        &= \sigma_g(x) \cdot g^{-1} \cdot(\sigma_g(\mu_{\sigma_g^{-1} \gamma \sigma_g}))_{\gamma} \cdot g h g^{-1} \otimes w, \\
        &= \sigma_g(x) \cdot g^{-1} \otimes \sigma_g(\mu_1) (g h g^{-1}) * w, \\
        &= \sigma_g(x) \cdot g^{-1} \otimes \mu_1 h(w).
    \end{align*}
    Where here $*$ denote the action of $F[H]$ on $g * W$. For $y \in R$,
    \[
    y \cdot \phi(x \otimes w) = y \cdot \sigma_g(x) \cdot g^{-1} \otimes w,
    \]
    whereas
    \[
    \phi(y * (x \otimes w)) = \phi(\sigma_{g^{-1}}(y) \cdot x \otimes w) = y \cdot \sigma_g(x) \cdot g^{-1} \otimes w
    \]    
    and thus the map is $R$-linear. It also directly seen to be an isomorphism, natural in $W$.
\end{proof}

\subsection{Description of Indecomposables}

We can use Corollary \ref{cor:presss} to obtain a precise description of the restriction $e \cdot V$ for indecomposable $V \in \Rep_{L}^{\rtimes}(G)$. Recall that $\Ind_{L}^{\rtimes}(G)$ and $\Ind_{F}(H)$ denote the sets of isomorphism classes of indecomposable objects of $\Rep_{L}^{\rtimes}(G)$ and $\Rep_{F}(H)$ respectively.

\begin{defn}
    For $W \in \Ind_{F}(H)$, we write $\Gamma_{e,W} \coloneqq \Stab_{\Gamma_e}(W)$.
\end{defn}

\begin{thm}\label{thm:mainthm}
    Suppose that $V \in \Ind_L^{\rtimes}(G)$, and $W$ is any indecomposable direct summand of $e \cdot V$. Then
    \[
        e \cdot V = \bigoplus_{\gamma \in \Gamma_{e}/\Gamma_{e,W}} (\gamma * W)^{m(V)}
    \]
    for some $m(V) \geq 1$. Furthermore, if $W \in \Ind_{F}(H)$ then
    \begin{itemize}
        \item $W$ appears in $e \cdot V$ for a unique $V \in \Ind_L^{\rtimes}(G)$,
        \item This $V$ is irreducible if and only if $W$ is irreducible,
        \item For this $V$, $(L \rtimes G) \otimes_{L[H]} (g*W) = V^{\oplus r}$ for any $g \in G_e$, where $r = |\Gamma_{e,W}| / m(V)$.
    \end{itemize}
    In particular, $m(V) \mid |\Gamma_{e,W}|$, and sending $W$ to this unique $V$ defines bijective correspondences
    \[
    \Ind_L^{\rtimes}(G) \xleftrightarrow{\sim} \Ind_{F}(H) / \Gamma_e, \qquad \Irr_L^{\rtimes}(G) \xleftrightarrow{\sim} \Irr_{F}(H) / \Gamma_e
    \]
    where conversely $V \in \Ind_L^{\rtimes}(G)$ is sent to any indecomposable direct summand of $e \cdot V$.
\end{thm}

\begin{proof}
    For the first claim, write $e \cdot V$ as a direct sum of indecomposable objects of $\Rep_L(H)$ using Lemma \ref{lem:KRS}, and let $W$ be an indecomposable direct summand of $e \cdot V$. For any $g \in G_e$, the number of copies of each $g * W$ in $e \cdot V$ must be equal to the number of copies of $W$, using Lemma \ref{lem:KRS} and the fact that the action $\rho(g)$ of $g$ on $V$ defines an isomorphism $\rho(g) \colon g * (e \cdot V) \xrightarrow{\sim} e \cdot V$.
    
    It therefore remains remains to see that only one orbit of the action of $\Gamma_e$ on $\Ind_{F}(H)$ appears in $e \cdot V$. Indeed, suppose that $W_1, W_2 \in \Ind_L(H)$ both appear in $e \cdot V$, so $e \cdot V = W_1 \oplus W_2 \oplus U$ for some $U \in \Rep_L(H)$. Then $(L \rtimes G) \otimes_{L[H]} (e \cdot V) \cong V^{\oplus |\Gamma_e|}$ by Corollary \ref{cor:indres}, and so by Lemma \ref{lem:KRS} both $(L \rtimes G) \otimes_{L[H]} W_1$ and $(L \rtimes G) \otimes_{L[H]} W_2$ are isomorphic to a direct sum of a number of copies of $V$. In particular, there are $a,b \geq 1$ with 
    \[
        \left((L \rtimes G) \otimes_{L[H]} W_1 \right)^{\oplus a} \cong \left((L \rtimes G) \otimes_{L[H]} W_2 \right)^{\oplus b}
    \]
    and applying $e \cdot V$ we see that 
    \[
        \bigoplus_{x \in \Gamma_e} (x \cdot W_1)^{\oplus a} \cong \bigoplus_{y \in \Gamma_e} (y \cdot W_2)^{\oplus b}
    \]
    by Proposition \ref{prop:resind} and therefore $W_2 \cong \gamma \cdot W_1$ for some $\gamma \in \Gamma_e$, by Lemma \ref{lem:KRS}.

    For the remaining claims, suppose that $W \in \Ind_{F}(H)$. From Proposition \ref{prop:resind}, $W$ is a direct summand of $e \cdot ((L \rtimes G) \otimes_{L[H]} W)$, and thus of $e \cdot V$ for some $V \in \Ind_L^{\rtimes}(G)$ after writing $(L \rtimes G) \otimes_{L[H]} W$ as a direct sum of indecomposable objects and using Lemma \ref{lem:KRS}. For this $V$, from the above we have that
    \[
        V^{|\Gamma_e|} \cong (L \rtimes G) \otimes_{L[H]} (e \cdot V) \cong \bigoplus_{\gamma \in \Gamma_e/\Gamma_{e,W}} ((L \rtimes G) \otimes_{L[H]}(\gamma * W))^{m(V)}
    \]
    and therefore $(L \rtimes G) \otimes_{L[H]} (\gamma*W) = V^{\oplus r}$ for any $\gamma \in \Gamma_e$ by Lemma \ref{lem:KRS}, where $r = |\Gamma_{e,W}| / m(V)$.
    
    In particular, the $V \in \Ind_L^{\rtimes}(G)$ for which $W$ is a direct summand of $e \cdot V$ is unique. Finally, from Corollary \ref{cor:presss} we have that $V$ is irreducible if and only if $W$ is irreducible. 
\end{proof}

We can use this to characterise when $\Rep_L^{\rtimes}(G)$ is semisimple, improving on Proposition \ref{prop:freeandss}. It also possible to show this directly, using a generalisation of Maschke's Theorem \cite[Lem.\ 1.3]{KUN}.

\begin{cor}\label{cor:sscriterion}
    Suppose that $G$ is finite. Then $\Rep_L^{\rtimes}(G)$ is semisimple if and only if $|H| \in L^{\times}$.
\end{cor}

\begin{proof}
    By Lemma \ref{lem:KRS}, the category $\Rep_L^{\rtimes}(G)$ is semisimple if and only if $\Irr_L^{\rtimes}(G) = \Ind_L^{\rtimes}(G)$, which by Theorem \ref{thm:mainthm} is equivalent to the statement that $\Irr_F(H) = \Ind_F(H)$, or in other words that $\Rep_F(H)$ is semisimple. By Maschke's Theorem this is equivalent to $|H|$ being invertible in $F$, which is the same as $|H|$ being invertible in $L$, as $L$ is the product of copies of $F$.
\end{proof}

In general, we will see that any positive integer can occur as some $m(V)$ (Section \ref{sect:realisedivisionalg}). One case when we know that the numbers $m(V)$ are equal to $1$ is the following. This applies in particular when $K$ is a finite field, by Wedderburn's Little Theorem \cite[Thm.\ IV.4.1]{milneCFT}.

\begin{prop}\label{prop:finitefieldsmequals1}
Suppose that $K$ is a field such that $\Br(K') = 0$ for any finite extension $K'/K$, or equivalently, such that any finite dimensional division algebra over $K$ is a field. Then for any $V \in \Irr_L^{\rtimes}(G)$, $m(V) = 1$. 
\end{prop}

\begin{proof}
    By Corollary \ref{cor:homsetisom}, there is a $K$-algebra isomorphism 
    \[
        L \otimes_K \End_{L \rtimes G}(V) \xrightarrow{\sim} \End_{F[H]}(e \cdot V),
    \]
    and thus $\End_{F[H]}(e \cdot V)$ is commutative because $\End_{L \rtimes G}(V)$ is a field. In particular, $e \cdot V$ cannot contain any factors with multiplicity greater than $1$, so $m(V) = 1$.
\end{proof}

We also record here that the induction and restriction functors of Theorem \ref{thm:mainthm} are compatible with the equivalence of Proposition \ref{prop:connequiv}, the proof of which is direct. We use the alternative notation for the functors between $\Rep_F^{\rtimes}(G_e)$ and $\Rep_{F}(H)$, described at the start of Section \ref{sect:whenLisafield} below.

\begin{prop}\label{prop:indandrescompwithconnequiv}
    The diagrams 
\[\begin{tikzcd}
	{\Rep_L^{\rtimes}(G)} && {\Rep_F^{\rtimes}(G_e)} & {\Rep_L^{\rtimes}(G)} && {\Rep_F^{\rtimes}(G_e)} \\
	& {\Rep_F(H)} &&& {\Rep_F(H)}
	\arrow["{e \cdot (-)}", from=1-1, to=1-3]
	\arrow["\sim"', from=1-1, to=1-3]
	\arrow["{e \cdot(-)}"', from=1-1, to=2-2]
	\arrow["{(-)|_H}", from=1-3, to=2-2]
	\arrow["{(L \rtimes G) \otimes_{L \rtimes G_e} -}"', from=1-6, to=1-4]
	\arrow["\sim", from=1-6, to=1-4]
	\arrow["{(L \rtimes G) \otimes_{L[H]} -}", from=2-5, to=1-4]
	\arrow["{\Ind_H^G(-)}"', from=2-5, to=1-6]
\end{tikzcd}\]
    both commute up to natural isomorphism.
\end{prop}

\subsection{The Case When $L$ is a Field}\label{sect:whenLisafield}
\[
\text{\emph{From now on in this paper on we assume that $L$ is a field.}}
\]
Note that one is always reduced to this case in practice by Proposition \ref{prop:connequiv} and Proposition \ref{prop:indandrescompwithconnequiv}. As $L$ is a field, the action map $\Gamma \rightarrow \Gal(L/K)$ is an isomorphism (cf.\ Example \ref{eg:classicalgaloisext}), and so without loss of generality one could assume that $\Gamma = \Gal(L/K)$. Furthermore, $G_e = G$, $\Gamma_e = \Gamma$, $L = F$, and the functors of the previous section become
\begin{align*}
    (-)|_H &\colon \Rep_{L}^{\rtimes}(G) \rightarrow \Rep_{L}(H), \\
    \Ind_{H}^G(-) &\colon \Rep_{L}(H) \rightarrow \Rep_{L}^{\rtimes}(G),
\end{align*}
where $\Ind_H^G$ is as defined as described at the end of the introduction.

We first give a strengthening of Theorem \ref{thm:mainthm} to intermediate extensions, which we will use in Section \ref{sect:semilinearschurindex} when we consider the semilinear Schur index.

Let $S$ be any intermediate subgroup $H \leq S \leq G$, which corresponds to a unique subgroup $\Gamma_S \leq \Gamma$ with $S = \sigma^{-1}(\Gamma_S)$. If $F$ is the fixed field of $\Gamma_S$, then $K \leq F \leq L$ and $L/F$ is Galois. In particular, the tuple $(S,H, L/F, \sigma \colon S \rightarrow \Gamma_S)$ satisfies the assumptions of Section \ref{sect:notation}, and so by Theorem \ref{thm:mainthm} there are canonical bijections
\[
\begin{array}{cc}
\begin{aligned}
\Ind_L^{\rtimes}(G) &\xleftrightarrow{\sim} \Ind_{L}(H) / \Gamma, \\
\Ind_L^{\rtimes}(S) &\xleftrightarrow{\sim} \Ind_{L}(H) / \Gamma_S.
\end{aligned}
&
\begin{aligned}
\Irr_L^{\rtimes}(G) &\xleftrightarrow{\sim} \Irr_{L}(H) / \Gamma, \\
\Irr_L^{\rtimes}(S) &\xleftrightarrow{\sim} \Irr_{L}(H) / \Gamma_S.
\end{aligned}
\end{array}
\]

Theorem \ref{thm:mainthm} further provides a description of induction and restriction between $\Rep_{L}^{\rtimes}(G)$ and $\Rep_{L}(H)$, and induction and restriction between $\Rep_{L}^{\rtimes}(S)$ and $\Rep_{L}(H)$. We now wish to understand induction and restriction between $\Rep_{L}^{\rtimes}(G)$ and $\Rep_{L}^{\rtimes}(S)$.

\begin{prop}\label{prop:generalintermediatedecomp}
    Let $\sA$ be an orbit of $G$ on $\Ind_L(H)$, corresponding to $V \in \Ind_L^{\rtimes}(G)$, and let
    \[
        \sA_1 \sqcup \cdots \sqcup \sA_r
    \]
    be a decomposition of $\sA$ into $S$-orbits. Let $U_i \in \Ind_L^{\rtimes}(S)$ correspond to $\sA_i$, so
    \[
    U_i|_H \cong \bigoplus_{X \in \sA_i} X^{\oplus m(U_i)}
    \]
    for some $m(U_i) \geq 1$ (by Theorem \ref{thm:mainthm}). Then $m(U_i) \mid m(V)$, $m(V) \mid m(U_i) \cdot [\Gamma_X: \Gamma_{S,X}]$,
    \[
    V|_S \cong \bigoplus_{i = 1}^r U_i^{\oplus \frac{m(V)}{m(U_i)}},
    \]
    and 
    \[
    \Ind_S^G(U_i) \cong V^{\oplus \frac{[\Gamma_X : \Gamma_{S,X}]\cdot m(U_i)}{m(V)}}
    \]
    for any $X \in \sA_i$. Furthermore, $V$ is irreducible if and only if each $U_i$ is irreducible.
\end{prop}

\begin{proof}
    For the first claim, note that using Corollary \ref{cor:detectisom} for $\Rep_L^{\rtimes}(S)$,
    \[
    (V|_S)^{m(U_1) \cdots m(U_r)} \cong \bigoplus_{i = 1}^r U_i^{\oplus m(V) m(U_1) \cdots m(U_{i-1}) m(U_{i+1}) \cdots m(U_r)}
    \]
    because their restrictions to $H$ are isomorphic. Then the description of $V|_S$ follows from Lemma \ref{lem:KRS}.
    
    For the second claim, take any $X \in \sA_i$, and note that by Theorem \ref{thm:mainthm}
    \[
        V^{\oplus \frac{|\Gamma_X|}{m(V)}} \cong \Ind_{H}^G (X) \cong \Ind_S^G (\Ind_{H}^S (X)) \cong \Ind_S^G(U_i)^{\oplus \frac{|\Gamma_{S,X}|}{m(U_i)}}.
        \]
    The description of $\Ind_S^G(U_i)$ then simiarly follows directly from Lemma \ref{lem:KRS}.
    
    For the final claim, note that $V$ is irreducible if and only if each element of $\sA$ is irreducible, and $U_i$ is irreducible if and only if each element of $\sA_i$ is irreducible, in each case by Theorem \ref{thm:mainthm}.
\end{proof}

\begin{eg}
    If $S = G$, then $r = 1$, $m(V) = m(U_1)$, and $[\Gamma_X : \Gamma_{S,X}] = 1$. If $S = H$, then $r = [\Gamma : \Gamma_X]$, $m(U_i) = 1$, and $[\Gamma_X: \Gamma_{S,X}] = [\Gamma_X : 1]$.
\end{eg}

The main case of interest is the following. To state it, note that if $S$ is a normal subgroup of $G$, then there is an action of $G/S$ on $\Ind_L^{\rtimes}(S)$, where for $g \in G$, and $U \in \Ind_L^{\rtimes}(S)$, $g*U$ is the abelian group $U$ but with $L \rtimes S$-module structure
\begin{itemize}
        \item $s * u = (g^{-1}sg)(u)$,
        \item $\lambda * u = \sigma_g^{-1}(\lambda) u$,
\end{itemize}
for $\lambda \in L$, $s \in S$ and $u \in U$. This action preserves the subset $\Irr_L^{\rtimes}(S)$ of $\Ind_L^{\rtimes}(S)$.

\begin{remark}
    The action of $G/H$ on $\Ind_L(H)$ is the restriction to the diagonal of an action of the product $G/H \times G/H$ on $\Ind_L(H)$ (where one copy acts by conjugation in $G$ and the other acts on scalars - see Section \ref{sect:countcc}). When $S > H$, however, this action of $G/S$ on $\Ind_L^{\rtimes}(S)$ is not the restriction of an action of the product $G/S \times G/S$ in the same way: one only has an action of the diagonal.
\end{remark}

\begin{cor}\label{cor:splittingovergalois}
    Suppose that $V \in \Ind_L^{\rtimes}(G)$, and that for some corresponding $W \in \Ind_L(H)$, $G_W$ is normal in $G$. Set $G_V \coloneqq G_W$, which is independent of the choice of $W$. Let $U \in \Ind_L^{\rtimes}(G_V)$ correspond to $W$. Then
    \begin{align*}
        V|_{G_W} &\cong \bigoplus_{g \in G/G_V} g * U, \\
        V &\cong \Ind_{G_V}^G(g * U),
    \end{align*}
    $V$ is irreducible if and only if $U$ is irreducible, and for any $g \in G$, $g * U$ corresponds to $g * W$, so
    \begin{align*}
        (g * U)|_H &\cong (g * W)^{\oplus m(V)},\\
        \Ind_{H}^{G_V}(g * W) &\cong (g * U)^{\oplus \frac{[G_V: H]}{m(V)}}.
    \end{align*}
\end{cor}

\begin{proof}
    We have that for any $g \in G$, $G_{g*W} = g G_W g^{-1} = G_W$. In particular, the $G$-orbit $\{g * W\}_{g \in G/G_V}$ decomposes into singleton $G_V$-orbits, which give the $\{U_1, ... , U_r\}$ from Proposition \ref{prop:generalintermediatedecomp}. We have that $m(g*U) | m(V)$ and $m(V) \mid m(g*U) \cdot [G_V : G_V] = m(g * U)$, and therefore $m(V) = m(g * V)$. Then the statements follow from Theorem \ref{thm:mainthm} and Proposition \ref{prop:generalintermediatedecomp}.
\end{proof}

Now we are interested in the condition that $L$ is furthermore a splitting field for $W \in \Ind_L(H)$ (meaning that $\End_{L[H]}(W) = L$, or equivalently that $W_{L'}$ is irreducible for any field extension $L'/L$).

\begin{prop}\label{prop:whencentral}
    Suppose that $W \in \Ind_L(H)$, corresponding to $V \in \Ind_L^{\rtimes}(G)$. Then the $K$-algebra $\End_{L \rtimes G}(V)$ is central over $K$ if and only if $\Gamma$ fixes $W$ and $Z(\End_{L[H]}(W)) = L$.
\end{prop}

\begin{proof}
    From the isomorphism
    \[
        L \otimes_{K} \End_{L \rtimes G}(V) \rightarrow \End_{L[H]}(V|_H)
    \]
    of Corollary \ref{cor:homsetisom}, $Z(\End_{L \rtimes G}(V)) = K$ if and only if $Z(\End_{L[H]}(V|_H)) = L$, by \cite[Prop.\ IV.2.3]{milneCFT}. From Theorem \ref{thm:mainthm}, this occurs if and only if $\Gamma$ fixes $W$ and $Z(\End_{L[H]}(W)) = L$.
\end{proof}

\begin{prop}\label{prop:descofcentre}
Suppose that $W \in \Irr_L(H)$, corresponding to $U \in \Irr_L^{\rtimes}(G_W)$ and $V \in \Irr_L^{\rtimes}(G)$, and that $\End_{L[H]}(W) = L$. Set $L_W \coloneqq L^{G_W}$, and $D \coloneqq \End_{L \rtimes G}(V)$. Then
\begin{enumerate}
    \item $Z(D) \cong L_W$,
    \item $\dim_{Z(D)} D = m(V)^2$,
    \item $\End_{L \rtimes G_W}(U)$ is central over $L_W$,
    \item $\dim_{L_W}(\End_{L \rtimes G_W}(U)) = m(V)^2$.
\end{enumerate}
\end{prop}

\begin{proof}
    For the first point, let $A \coloneqq (L \rtimes G)/\text{Ann}_{L \rtimes G}(V)$, which, as $V$ is finite-dimensional over $L$, is a finite dimensional $K$-algebra. Note that because $V$ is a simple $A$-module, $A$ is furthermore simple by \cite[Prop.\ 3.31]{CR1}. As $A$ and $A_L$ are both quotients, they define full-subcategories
\[\begin{tikzcd}
	{\Mod^{\fd}_A} & {\Mod^{\fd}_{A_L}} \\
	{\Rep_L^{\rtimes}(G)} & {\Rep_{L \otimes_K L}^{\rtimes}(G)}
	\arrow["{L \otimes_K -}", shift left, from=1-1, to=1-2]
	\arrow[hook, from=1-1, to=2-1]
	\arrow["{(-)|_A}", shift left, from=1-2, to=1-1]
	\arrow[hook, from=1-2, to=2-2]
	\arrow["{L \otimes_K -}", shift left, from=2-1, to=2-2]
	\arrow["{\Res_{L \rtimes G}}", shift left, from=2-2, to=2-1]
\end{tikzcd}\]
which are closed under sub-objects, and the base change and restriction functors restrict to the base change and restriction functors between $A$ and $A_L$. Let $M \in \Mod^{\fd}_{A_L}$ correspond to $W$. The $K$-algebra $\End_{A_L}(M)$ contains $L$, and has the same $K$-dimension as $\End_{L[H]}(W) = L$, and therefore $\End_{A_L}(M) = L$ and $A_L$ is split over $L$. We may therefore apply \cite[Thm.\ 74.4 (ii)]{CR2} to obtain that $Z(\End_{L \rtimes G}(V)) \cong L^{\Gamma_M}$, where $\Gamma_M$ is the stabiliser of $M$ in $\Gamma$. The action of $\Gamma$ on $\Irr_L(H)$ corresponds to that of $\Gamma$ on $\Irr(A_L)$ by Lemma \ref{lem:groupsactionssame}, and therefore by Proposition \ref{prop:trivialisationres} $\Gamma_W = \Gamma_M$ and 
\[
    Z(\End_{L \rtimes G}(V)) \cong L^{\Gamma_M} = L^{\Gamma_W} = L_W.
\]
For the second point, letting $D \coloneqq \End_{L \rtimes G}(V)$, then applying Corollary \ref{cor:homsetisom},
    \begin{align*}
        \dim_{Z(D)} (D) &= [Z(D):K] \dim_K (D), \\
        &= \dim_L \End_{L[H]}(V|_H), \\
        &= \dim_L \left(\prod_{i = 1}^{[G:G_W]} M_{m(V)}(L) \right), \\
        &= [G:G_W] \cdot m(V)^2,
    \end{align*}
    and hence point (2) follows from point (1), as $[G:G_W] = [L_W:K]$. Similarly, by Corollary \ref{cor:homsetisom}, there is an isomorphism
    \[
        L \otimes_{L_W} \End_{L \rtimes G_W}(U) \xrightarrow{\sim} \End_{L[H]}(U|_H) \cong M_{m(U)}(L),
    \]
    because the $G_W$ orbit of $W$ is a singleton. Then taking the centre, by \cite[Prop.\ IV.2.3]{milneCFT}, we have that
    \[
    L \otimes_{L_W} Z(\End_{L \rtimes G_W}(U)) \xrightarrow{\sim} Z(\End_{L[H]}(U|_H) \cong M_{m(U)}(L)) = L,
    \]
    hence $\dim_{L_W} Z(\End_{L \rtimes G_W}(U)) = 1$. The final claim again follows Corollary \ref{cor:homsetisom}, using the fact that 
    \[
        U|_H \cong W^{\oplus m(V)}
    \]
    by Proposition \ref{prop:generalintermediatedecomp}, as $m(U) \mid m(V)$ and $m(V) \mid m(U) \cdot [G_W : G_W] = m(U)$.
\end{proof}

\begin{remark}
    Note that $Z(\End_{L \rtimes G}(V)) \cong L_W$, and so the field $L_W$ only depends on the $\Gamma$-orbit of $W$. This can also be seen directly, as the $L_W$ for $W$ in a fixed $\Gamma$-orbit are all Galois conjugate in $L$.
\end{remark}

\section{Semilinear Schur Index}\label{sect:semilinearschurindex}

\begin{defn}
For $W \in \Ind_L(H)$, set $m_K^L(W) \coloneqq m(V)$ for the unique corresponding $V \in \Ind_L^{\rtimes}(G)$.
\end{defn}

As a direct consequence of Theorem \ref{thm:mainthm} we have the following.

\begin{cor}\label{cor:whenextends}
Let $W \in \Ind_L(H)$. Then $W$ is extendable to a semilinear representation of $G$ if and only if $\Gamma$ fixes $W$ and $m_K^L(W) = 1$.
\end{cor}

From the results of the previous section, the numbers $m_K^L(W)$ have the following properties. Recall that $G_W$ is the preimage of $\Gamma_W$ in $G$, and we write $L_W \coloneqq L^{\Gamma_W}$. For any intermediate field $E$ with $K \leq E \leq L$, we write $\Gamma_E$ for the preimage of $\Gal(L/E)$ under the isomorphism $\Gamma \xrightarrow{\sim} \Gal(L/K)$.

\begin{prop}\label{prop:propertiesSchurIndex}
    Suppose that $W \in \Ind_L(H)$, and set $\OO(W) \coloneqq \bigoplus_{\gamma \in \Gamma / \Gamma_W} \gamma * W$. Then:
    \begin{enumerate}
        \item $\OO(W)^{\oplus m_K^L(W)} = V|_H$ for a unique $V \in \Ind_L^{\rtimes}(G)$,
        \item $m_K^L(W)$ is the unique $m \geq 1$ such that $\OO(W)^{\oplus m}$ extends to an indecomposable object of~$\Rep_{L}^{\rtimes}(G)$,
        \item $m_K^L(W)$ is the smallest $m \geq 1$ such that $\OO(W)^{\oplus m}$ extends to an object of $\Rep_{L}^{\rtimes}(G)$,
        \item For any $V \in \Rep_L^{\rtimes}(G)$, the number of times $W$ appears as a direct summand in the Krull-Remak-Schmidt decomposition of $V|_H$ is a multiple of $m_K^L(W)$.
    \end{enumerate}
    We furthermore have that
    \begin{enumerate}
        \item[(5)] $m_K^L(W) = m_K^L(g * W)$ for any $g \in G$.  
        \item[(6)] $m_K^L(W) \mid |\Gamma_W|$.
        \item[(7)] $m_{L_W}^{L}(W) = m_{K}^L(W)$.
    \end{enumerate} 
        In particular, applying the above to $G_W \twoheadrightarrow \Gamma_W$,
    \begin{enumerate}
        \item[(1')] $W^{\oplus m_{L_W}^{L}(W)} = U|_H$ for a unique $U \in \Ind_L^{\rtimes}(G_W)$,
        \item[(2')] $m_{L_W}^{L}(W)$ is the unique $m \geq 1$ such that $W^{\oplus m}$ extends to an indecomposable object of~$\Rep_{L}^{\rtimes}(G_W)$,
        \item[(3')] $m_{L_W}^{L}(W)$ is the smallest $m \geq 1$ such that $W^{\oplus m}$ extends to an object of $\Rep_{L}^{\rtimes}(G_W)$,
        \item[(4')] For any $U \in \Rep_L^{\rtimes}(G_W)$, the number of times $W$ appears as a direct summand in the Krull-Remak-Schmidt decomposition of $U|_H$ is a multiple of $m_{L_W}^{L}(W)$.
    \end{enumerate}

    Now suppose that $K \subset E \subset L$ is an intermediate field. Then
        \begin{enumerate}
            \item[(8)] $m_{E}^L(W) \mid m_K^L(W)$,
            \item[(9)] $m_{K}^L(W) \mid m_{E}^L(W) \cdot [\Gamma_W : \Gamma^E_W]$, and
            \item[(10)] $m_{K}^L(W) \leq m_{E}^L(W) \cdot [E:K]$.
        \end{enumerate}
    If additionally either $E$ or $L_W$ is Galois over $E \cap L_W$, then $[\Gamma_W : \Gamma^E_W] = [E : E \cap L_W]$, so
    \begin{enumerate}
        \item[(11)] $m_{K}^L(W) \mid  m_{E}^L(W) \cdot [E : E \cap L_W]$, and in particular
        \item[(12)] $m_{K}^L(W) \mid m_{E}^L(W) \cdot [E:K]$.
    \end{enumerate}
\end{prop}

\begin{proof}
    Points (1), (2), (3), (4), (5) and (6) follow from Theorem \ref{thm:mainthm} and Corollary \ref{cor:detectisom}. For point (7), we may take $S$ to be $G_W$ in Proposition \ref{prop:generalintermediatedecomp}, to obtain that $m_{L_W}^{L}(W) \mid m_{K}^L(W)$ and $m_K^L(W) \mid m_{L_W}^{L}(W) \cdot [\Gamma_W : (\Gamma_W)_W] = m_{L_W}^{L}(W)$.

    Now suppose that $E$ is an intermediate field between $K$ and $L$. Taking $S \coloneqq G^E$, the preimage of $\Gamma^E$ in $G$, we have that $m_{E}^L(W) \mid m_K^L(W)$ and $m_{K}^L(W) \mid m_{E}^L(W) \cdot [\Gamma_W : \Gamma^E_W]$. Furthermore,
    \[
    [\Gamma : \Gamma_W][\Gamma_W : \Gamma_W^E] = [\Gamma : \Gamma^E] [\Gamma^E : \Gamma^E_W], 
    \]
    and $[\Gamma^E : \Gamma^E_W] \leq [\Gamma : \Gamma_W]$, hence $[\Gamma_W : \Gamma_W^E] \leq [\Gamma : \Gamma^E] = [E:K]$.

    Suppose now that $E$ or $L_W$ is Galois over $E \cap L_W$. Then in the square
\[\begin{tikzcd}
	{\Gamma^E_W} & {\Gamma^E} \\
	{\Gamma_W} & {\Gamma^{E \cap L_W}}
	\arrow[hook, from=1-1, to=1-2]
	\arrow[hook, from=1-1, to=2-1]
	\arrow[hook, from=1-2, to=2-2]
	\arrow[hook, from=2-1, to=2-2]
\end{tikzcd}\]
    at least one of $\Gamma_W$ or $\Gamma^E$ is normal in $\Gamma^{E \cap L_W}$, and so by the second isomorphism theorem,
    \[
    [\Gamma_W : \Gamma^E_W] = [\Gamma^{E \cap L_W} : \Gamma^E] = [E : E \cap L_W]. \qedhere
    \]
\end{proof}

As a direct consequence of Proposition \ref{prop:finitefieldsmequals1}, we also have the following.

\begin{cor}\label{cor:finitefield}
    If $L$ is a finite field, $m_K^L(W) = 1$ for any $W \in \Irr_L(H)$.
\end{cor}

The following example shows how the numbers $m_K^L(W)$ capture arithmetic information.

\begin{eg}\label{eg:C2C4}
    Let $K$ be any field, and let $L = K(\sqrt{d})$, where $d \in K^{\times} \setminus (K^{\times})^2$. Taking $G = C_4 = \langle y \rangle$, there is a natural homomorphism
    \[
    C_4 \twoheadrightarrow \Gamma \coloneqq \Gal(K(\sqrt{d}) / K),
    \]
    with kernel $C_2 = \langle x \rangle$, where $y^2 = x$. Let $W$ be the $1$-dimensional $L$-representation of $C_2$ where $x$ acts by scalar multiplication by $-1$. Then, $m_K^{L}(W)$ is either $1$ or $2$ by Proposition \ref{prop:propertiesSchurIndex}(6). Explicitly, using the matrix description of Section \ref{sect:matrixdesc}, $W$ extends to a semilinear representation of $G$ over $K(\sqrt{d})$ if and only if there is an $a = u + v \sqrt{d} \in K(\sqrt{D})^{\times}$ (defining the matrix of the action of $y$) with
    \[
        (u + v \sqrt{d})(u - v \sqrt{d}) = -1,
    \]
    this condition arising from the relation that $y^2 = x$. In particular, as $W$ is fixed by the action of $\Gamma$, we see from Corollary \ref{cor:whenextends} that $m_K^{L}(W) = 1$ if and only if $-1 = u^2 - d v^2$ for some $u, v \in K$.
\end{eg}

This can be extended to produce arbitrarily large values of $m_K^L(W)$ (see also Corollary \ref{cor:realisationdivalg}).

\begin{eg}
Let $p$ be a prime, and let $L/K$ be a degree $p$ Galois field extension. For any $n \geq 1$, there is a quotient map
\[
    C_{np} \twoheadrightarrow \Gal(L/K),
\]
with kernel $C_n = \langle x \rangle \subset C_{np} = \langle y \rangle$, where $y^p = x$. Then for any $1$-dimensional representation $W$ of $H$ valued in $K$, then similarly to above, $m_K^L(W) \in \{1,p\}$, and $m_K^L(W) = 1$ if and only if $N_K^L(a) = \chi_W(x)$ for some $a \in L^{\times}$.
\end{eg}

\section{Recovery of the Classical Schur Index}\label{sect:classicalSchurIndex}

Suppose now that $H$ is a group, and $L/K$ is a Galois extension of fields. Then, as in Section \ref{sect:splitgroupext}, we may take $G \coloneqq H \times \Gal(L/K)$ and $G \twoheadrightarrow \Gal(L/K)$ to be the natural projection, and consider the category $\Rep_{L}^{\rtimes}(G)$. By Proposition \ref{prop:galoisdescent} there are equivalences
\[\begin{tikzcd}[column sep=huge]
	{\Rep_{K}(H)} & {\Rep_L^{\rtimes}(H \times \Gal(L/K))}
	\arrow["{L \otimes_K -}", shift left, from=1-1, to=1-2]
	\arrow["{(-)^{\Gal(L/K)}}", shift left, from=1-2, to=1-1]
\end{tikzcd}\]
under which \emph{induction}
\[
(-)_L \colon \Rep_K(H) \rightarrow \Rep_L(H)
\]
corresponds to \emph{restriction}
\[
(-)|_H \colon \Rep_L^{\rtimes}(H \times \Gal(L/K)) \rightarrow \Rep_L(H),
\]
by Lemma \ref{lem:classicalinduction}, and \emph{restriction}
\[
(-)|_{K[H]} \colon \Rep_L(H) \rightarrow \Rep_K(H)
\]
corresponds to \emph{induction}
\[
\Ind_{H}^{H \times \Gal(L/K)} \colon \Rep_L(H) \rightarrow \Rep_L^{\rtimes}(H \times \Gal(L/K))
\]
by Lemma \ref{lem:classicalrestriction}. For example, given $V \in \Rep_L(H)$, $V$ is induced from a $K$-linear representation of $H$ if and only if $V$ extends to a semilinear representation of $H \times \Gal(L/K)$.

As $H$ commutes with $\Gal(L/K)$, the action of $\Gal(L/K)$ of Definition \ref{def:Gaction} on $\Rep_L(H) / \cong$ simplifies: for $\gamma \in \Gal(L/K)$ and $W \in \Rep_L(H)$, $\gamma * W$ is $W$ as a $K$-vector space, with same action of $H$, and $L$-linear structure $\lambda \cdot w \coloneqq \gamma^{-1}(\lambda) w$. In terms of matrices, if $W$ has matrix representation $(A_g)_{g \in G}$, then $\gamma * W$ has matrix representation $({}^{\gamma \!} A_g)_{g \in G}$ (in the notation of Section \ref{sect:matrixdesc}).

From Theorems A and B, we recover the following relationship between $\Rep_K(H)$ and $\Rep_L(H)$.

\begin{cor}\label{cor:classicalmainthm}
Let $V_1, V_2 \in \Rep_{K}(H)$. Then:
\begin{itemize}
    \item If $V_{1, L} \cong V_{2,L}$ in $\Rep_L(H)$ then $V_1 \cong V_2$.
    \item The natural map
    \[
    L \otimes_{K} \Hom_{K[H]}(V_1, V_2) \rightarrow \Hom_{L[H]}(V_{1,L}, V_{2,L})
    \]
    is an isomorphism.
    \item In particular, if $V \in \Rep_{K}(H)$ then the natural map
    \[
        L \otimes_{K} \End_{K[H]}(V) \rightarrow \End_{L[H]}(V_L)
    \]
    is an $L$-algebra isomorphism.
\end{itemize}
Let $V \in \Ind_{L}^{\rtimes}(G)$ and let $W \in \Ind_{L}(H)$ be an indecomposable direct summand of $V_L$. Then:
\begin{itemize}
    \item $V_L$ and $W|_{K[H]}$ are described by
    \[
        V_L \cong \bigoplus_{\gamma \in \Gal(L/K) / \Gal(L/K)_V} (\gamma * W)^{\overline{m}_K^L(W)}, \qquad W|_{K[H]} \cong V^{\oplus |\Gal(L/K)_W| / \overline{m}_K^L(W)}
    \]
    for some integer $\overline{m}_K^L(W) \geq 1$. In particular, $\overline{m}_K^L(W) \mid |\Gal(L/K)_W|$.
    \item $W$ is irreducible if and only if $V$ is irreducible.
    \item $W|_{K[H]} \cong (\gamma * W)|_{K[H]}$ and $\overline{m}_K^L(W) = \overline{m}_K^L(\gamma * W)$ for any $\gamma \in \Gal(L/K)$.
    \item Any $W \in \Ind_{L}(H)$ is a direct summand of $V_L$ for a unique $V \in \Ind_{K}(H)$. Sending $W$ to this unique $V$ defines bijections
    \[
       \Ind_{K}(H) \leftrightarrow \Ind_{L}(H) / \Gal(L/K), \qquad \Irr_{K}(H) \leftrightarrow \Irr_{L}(H) / \Gal(L/K).
    \]
\end{itemize}
\end{cor}

\subsection{Classical Schur Index}

For $L / K$ a finite Galois extension and $W \in \Ind_L(H)$, we call the number $\overline{m}_K^L(W)$ appearing in Corollary \ref{cor:classicalmainthm} the (classical) Schur index of $W$ over $K$. When $W$ is irreducible (which is the setting in which Schur indices are usually considered) this agrees by definition with the usual definition of the Schur index \cite[Thm.\ 38.1]{HUP}.

We following record how Proposition \ref{prop:propertiesSchurIndex} recovers the basic properties of the Schur index, and further generalises them to indecomposable modules.

\begin{cor}
Let $W \in \Ind_L(H)$. Then $W$ is defined over $K$ if and only if $\Gal(L/K)$ fixes $W$ and $\overline{m}_K^L(W) = 1$.
\end{cor}

\begin{cor}\label{cor:classicalSchurindexproperties}
Suppose that $W \in \Ind_L(H)$, and set $\OO(W) \coloneqq \bigoplus_{\gamma \in \Gal(L/K) / \Gal(L/K)_W} \gamma * W$. Then:
    \begin{enumerate}
        \item $\OO(W)^{\oplus \overline{m}_K^L(W)} = V_L$ for a unique $V \in \Ind_K(H)$,
        \item $\overline{m}_K^L(W)$ is the unique $m \geq 1$ such that $\OO(W)^{\oplus m}$ is induced from an indecomposable object of $\Rep_{K}(H)$,
        \item $\overline{m}_K^L(W)$ is the smallest $m \geq 1$ such that $\OO(W)^{\oplus m}$ is induced from an object of $\Rep_{K}(H)$,
        \item For any $V \in \Irr_K(H)$, the number of times $W$ appears as a direct summand in the Krull-Remak-Schmidt decomposition of $V_L$ is a multiple of $\overline{m}_K^L(W)$.
    \end{enumerate}
    We furthermore have that
    \begin{enumerate}
        \item[(5)] $\overline{m}_K^L(W) = \overline{m}_K^L(\gamma * W)$ for any $\gamma \in \Gal(L/K)$.  
        \item[(6)] $\overline{m}_K^L(W) \mid |\Gal(L/K)_W|$.
        \item[(7)] $\overline{m}_{L_W}^L(W) = \overline{m}_K^L(W)$.
    \end{enumerate} 
        In particular, applying the above to the Galois extension $L / L_W$,
    \begin{enumerate}
        \item[(1')] $W^{\oplus \overline{m}_{L_W}^L(W)} = U_L$ for a unique $U \in \Ind_{L_W}(H)$,
        \item[(2')] $\overline{m}_{L_W}^L(W)$ is the unique $m \geq 1$ such that $W^{\oplus m}$ is induced from an indecomposable object of $\Rep_{L_W}(H)$,
        \item[(3')] $\overline{m}_{L_W}^L(W)$ is the smallest $m \geq 1$ such that $W^{\oplus m}$ is induced from an object of $\Rep_{L_W}(H)$,
        \item[(4')] For any $U \in \Rep_{L_W}(H)$, the number of times $W$ appears as a direct summand in the Krull-Remak-Schmidt decomposition of $U_L$ is a multiple of $\overline{m}_{L_W}^L(W)$.
    \end{enumerate}

    Now suppose that $K \subset E \subset L$ is an intermediate field. Then
        \begin{enumerate}
            \item[(8)] $\overline{m}_{E}^L(W) \mid \overline{m}_{K}^L(W)$,
            \item[(9)] $\overline{m}_{K}^L(W) \mid \overline{m}_{E}^L(W) \cdot [\Gal(L/L_W) : \Gal(L/(E \cdot L_W))]$, and
            \item[(10)] $\overline{m}_{K}^L(W) \leq \overline{m}_{E}^L \cdot [E:K]$.
        \end{enumerate}
    If additionally either $E$ or $L_W$ is Galois over $E \cap L_W$, then
    \[
    [\Gal(L/L_W) : \Gal(L/(E \cdot L_W))] = [E : E \cap L_W],
    \]
    thus
    \begin{enumerate}
        \item[(11)] $\overline{m}_{K}^L(W) \mid  \overline{m}_{E}^L \cdot [E : E \cap L_W]$, and in particular
        \item[(12)] $\overline{m}_{K}^L(W) \mid \overline{m}_{E}^L \cdot [E:K]$.
    \end{enumerate}
\end{cor}

The above properties are all special cases of properties of all semilinear Schur indices. We now highlight some properties that the classical Schur indices $\overline{m}_K^L(W)$ satisfy, which are not true for the semilinear Schur indices $m_K^L(W)$. 

\subsubsection{Galois Splitting Field} When $H$ is a finite group (no assumptions on the characteristic of $L$), the field $L_W$ admits a nicer description. For a character $\chi$ of an irreducible representation of $H$ over $L$, we write $K(\chi)$ for the subfield of $L$ generated by the set $\{ \chi(h) \mid h \in H\}$.

\begin{lemma}
    Suppose that $H$ is finite, and $W \in \Irr_L(H)$. Then $L_W = K(\chi_W)$.
\end{lemma}
\begin{proof}
If $\gamma \in \Gal(L/K)$ fixes $W$, then $\gamma * \chi_W = \chi_W$, where $(\gamma * \chi_W)(h) \coloneqq \gamma(\chi_W(h))$ for any $h \in H$. In particular, $\gamma$ fixes $K(\chi)$, and so $K(\chi_W) \subset L_W$. Conversely, if $\gamma$ fixes $K(\chi_W)$, then $\gamma * \chi_W = \chi_W$, and so, as $\chi_{\gamma * W} = \gamma * \chi_W$, then $W \cong \gamma * W$ by \cite[Cor.\ 9.22]{ISA}. In particular, $\Gal(L / K(\chi)) \subset \Gal(L/K)_W$, and so $L_W \subset K(\chi_W)$.
\end{proof}

\begin{remark}
In particular, when $H$ is finite and $W \in \Irr_L(H)$, $L_W = K(\chi_W)$ is normal over $K$, being a subfield of a cyclotomic extension. In particular, from Corollary \ref{cor:classicalSchurindexproperties},
\[
\overline{m}_{K}^L(W) \mid \overline{m}_{E}^L(W) \cdot [E : E \cap K(\chi_W)]
\]
for any intermediate field $K \subset E \subset L$.
\end{remark}

\subsubsection{Divisibility by Dimension}

The following result, amongst other things, implies that any linear character of $H$ over $L$ has Schur index $1$. This makes sense, as any such representation fixed by $\Gal(L/K)$ is easily seen to be defined over $K$. The proof does not extend to the non-split setting, and indeed the statement for semilinear Schur indices is false. This can be seen from Example \ref{eg:C2C4}, taking $L / K$ to be $\bQ(i) / \bQ$, which exhibits linear characters of $H$ over $L$ with Schur index $2$.

\begin{prop}\label{prop:classicalSIdividesdim}
Suppose that $W \in \Irr_L(H)$, and let $D_W \coloneqq \End_{L[H]}(W)$, a (potentially non-central) division algebra over $L$. Then 
\[
\overline{m}_{K}^{L}(W) \mid \frac{\dim_L(W)}{\dim_L D_W}.
\]
\end{prop}
\begin{proof}
    Let $V \in \Irr_K(H)$ correspond to $W$, and set $m \coloneqq \overline{m}_K^L$. Writing $D_V \coloneqq \End_{K[G]}(V)$, we have that $\dim_K D_V \mid \dim_K V$ as $V$ is a free module over $D_V$. From Corollary \ref{cor:classicalmainthm} we have that
    \[
    V_L = \oplus_{\gamma \in \Gamma / \Gamma_W} (\gamma * W)^{m}
    \]
    and thus
    \[
    \dim_K D_V = \dim_L \End_{L[H]}(V_L) = \dim_L \prod_{\Gamma / \Gamma_W} M_{m}(D_W) = [\Gamma : \Gamma_W] \cdot m^2 \cdot \dim_L D_W.
    \]
    From $\dim_K D_V \mid \dim_K V$ and noting that $\dim_K V = \dim_L V_L$ we have that
    \[
        [\Gamma : \Gamma_W] \cdot m^2 \cdot \dim_L D_W \mid [\Gamma: \Gamma_W] \cdot m \cdot \dim_L W,
    \]
    and therefore $m \cdot \dim_L D_W \mid \dim_L W$ as required.
\end{proof}

\section{Base Change for Semilinear Representations}\label{sect:basechange}

We would like to better understand the division algebras $\End_{L \rtimes G}(V)$ for $V \in \Irr_L^{\rtimes}(G)$, and more specifically their splitting behaviour under base extension. In this section only we briefly relax the condition, in place since the start of Section \ref{sect:whenLisafield}, that $L$ is a field - this is only needed for the final result of this Section, Corollary \ref{cor:localglobal}. We also suppose in this section that $K'$ is an arbitrary, potentially infinite, field extension of $K$.

In this section we consider the base change functor
\[
(-)_{L'} \colon \Rep_{L}^{\rtimes}(G) \rightarrow \Rep_{L'}^{\rtimes}(G)
\]
as defined in Section \ref{sect:semilinearrepsGaloisext}, where $L' \coloneqq K' \otimes_K L$, which has Galois action of $\Gamma$ through the action of $\Gamma$ on $L$. We let $f$ denote a primitive idempotent of $K' \otimes_K F$, which unlike $F$ need no longer be connected, and set $F' \coloneqq f \cdot (K' \otimes_K F) \hookrightarrow L'$. Note that $F' \neq K' \otimes_K F$ in general, unlike what the notion might suggest, but instead $F'$ defines a connected component of the extension $L' / K'$, just as $F$ defines a connected component of the extension $L/K$.

We first show that base change is compatible with the induction and restriction functors appearing in Theorem \ref{thm:mainthm}.

\begin{prop}\label{prop:basechangecomm}
    The diagrams
\[\begin{tikzcd}
	{\Rep_L^{\rtimes}(G)} & {\Rep_{L'}^{\rtimes}(G)} && {\Rep_L^{\rtimes}(G)} & {\Rep_{L'}^{\rtimes}(G)} \\
	{\Rep_F(H)} & {\Rep_{F'}(H)} && {\Rep_F(H)} & {\Rep_{F'}(H)}
	\arrow["{(-)_{L'}}", from=1-1, to=1-2]
	\arrow["{(-)_{L'}}", from=1-4, to=1-5]
	\arrow["{e \cdot (-)}"', from=1-4, to=2-4]
	\arrow["{f \cdot (-)}", from=1-5, to=2-5]
	\arrow["{(L \rtimes G) \otimes_{L[H]} -}", from=2-1, to=1-1]
	\arrow["{(-)_{F'}}"', from=2-1, to=2-2]
	\arrow["{(L' \rtimes G) \otimes_{L'[H]} -}"', from=2-2, to=1-2]
	\arrow["{(-)_{F'}}"', from=2-4, to=2-5]
\end{tikzcd}\]
    both commute up to natural isomorphism.
\end{prop}

\begin{proof}
    For the first diagram, if $W \in \Rep_{F}(H)$, there is a natural morphism between both compositions
    \[
        L' \otimes_L ((L \rtimes G) \otimes_{L[H]} W) \rightarrow (L' \rtimes G) \otimes_{L'[H]}(F' \otimes_{F} W), \qquad \lambda \otimes x \otimes w \mapsto \lambda x \otimes 1 \otimes w.
    \] 
    Every pure tensor of the right-hand term can be put in the form $x \otimes 1 \otimes w$, and therefore this is surjective. If $W$ has rank $n \geq 0$ over $F$, both sides have rank $|\Gamma|n$ over $L'$, so the map is an isomorphism.

    For the second diagram, if $V \in \Rep_{L}^{\rtimes}(G)$, there is a natural morphism between both compositions
    \[
        F' \otimes_F (e \cdot V) \rightarrow f \cdot (L' \otimes_L V), \qquad \lambda \otimes v \mapsto \lambda \otimes v.
    \]
    Because $L'/ L$ is free, $f \cdot (L' \otimes_L V) = (f \cdot L') \otimes_L V$. Furthermore, as $fe = f$ in $F'$, any pure tensor of $(f \cdot L') \otimes_L V$ is of the form $\lambda \otimes v$, for $v \in e \cdot V$, and so the map is surjective. If $W$ has rank $n \geq 0$ over $L$, then both sides have rank $n$ over $F'$, and so the map is an isomorphism.
\end{proof}

\begin{cor}\label{cor:basechangedescofSL}
    Suppose that $W \in \Ind_F(H)$, and that $W_{F'} = W_1^{n_1} \oplus \cdots \oplus W_r^{n_r}$ is a direct sum decomposition of $W_{F'}$, where $W_i \in \Ind_{F'}(H)$ are pairwise non-isomorphic. Suppose that $V \in \Ind_L^{\rtimes}(G)$ corresponds to $W$, and $V_i \in \Ind_{L'}^{\rtimes}(G)$ corresponds to $W_i$ for $i = 1, ... , r$. Then 
    \[
        V_{L'} \cong \bigoplus_{i = 1}^r V_i^{\frac{|\Gamma_{f, W_i}| \cdot m(V) \cdot n_i}{|\Gamma_{e, W}| \cdot m(V_i)}}.
    \]
    In particular, when $W$ is absolutely irreducible, if $V_0 \in \Irr_{L'}^{\rtimes}(G)$ corresponds to $W_{F'}$, then
    \[
        V_{L'} \cong V_0^{\frac{|\Gamma_{f, W}| \cdot m(V)}{|\Gamma_{e, W}| \cdot m(V_0)}}
    \]
    and so $m(V_0) \cdot [\Gamma_{e, W} : \Gamma_{f,W}] \mid m(V)$. If additionally, $W$ is fixed by $\Gamma$, then
    \[
    [\Gamma_{e, W} : \Gamma_{f,W}] = [\Gamma_e : \Gamma_f] =  \deg_{F'}(K' \otimes_K F),
    \]
    thus
    \[
    V_{L'} \cong V_0^{\frac{m(V)}{\deg_{F'}(K' \otimes_K F) \cdot m(V_0)}}
    \]
    and so $m(V_0) \cdot \deg_{F'}(K' \otimes_K F) \mid m(V)$.
\end{cor}

\begin{proof}
    From the first diagram of Proposition \ref{prop:basechangecomm} and Theorem \ref{thm:mainthm} we have that
    \[
    V_{L'}^{\frac{|\Gamma_{e,W}|}{m(V)}} \cong \bigoplus_{i = 1}^r V_i^{\frac{|\Gamma_{f, W_i}| \cdot n_i}{m(V_i)}},
    \]
    from which the description of $V_{L'}$ follows from Lemma \ref{lem:KRS}. For the second statement, note that $\Gamma_{f, W} = \Gamma_{f, W_{F'}}$: if $\gamma \in \Gamma$ and $\gamma * W_{F'} \cong W_{F'}$, then as $\gamma * W_{F'} = (\gamma * W)_{F'}$, $(\gamma * W)_{F'} \cong W_{F'}$, hence $\gamma * W \cong W$ by Corollary \ref{cor:classicalmainthm}.
\end{proof}

\begin{remark}
    If $F' / F$ is a finite Galois extension of fields, then from Corollary \ref{cor:classicalmainthm} we have that
    \[
        W_{F'} = \bigoplus_{\gamma \in \Gal(F' / F) / \Gal(F' / F)_{W_0}} (\gamma * W_0)^{\overline{m}_{F}^{F'}(W)}
    \]
    where $W_0$ is taken to be any of the $W_i$. In particular,
    \[
    n_1 = \cdots = n_s = \overline{m}_{F}^{F'}(W_0) \ \text{ and } \ s = [\Gal(F'/F) : \Gal(F'/F)_{W_0}].
    \]
\end{remark}

\begin{cor}\label{cor:descdegreebasechange}
    Suppose that $W \in \Irr_{L}(H)$ is fixed by $\Gamma$ and $\End_{L[H]}(W) = L$. Let $V \in \Irr_L^{\rtimes}(G)$ correspond to $W$, and set $D \coloneqq \End_{L \rtimes G}(V)$, a central division algebra over $K$. Then
    \[
        \Deg(D_{K'}) = m_{K'}^{F'}(W_{F'}),
    \]
    and in particular $\Deg(D_{K'}) \cdot \deg_{F'}(K' \otimes_K F) \mid m_K^F(W)$ and $\Deg(D_{K'}) \mid [F' : K']$.
\end{cor}

\begin{proof}
    From Lemma \ref{lem:injofbasechangeSL}, there is an isomorphism
    \[
        D_{K'} \xrightarrow{\sim} \End_{L' \rtimes G}(V_{L'}),
    \]
    and from the equivalence of Proposition \ref{prop:connequiv} an isomorphism
    \[
        \End_{L' \rtimes G}(V_{L'}) \xrightarrow{\sim} \End_{F' \rtimes G_f}(f \cdot V_{L'}).
    \]
    Considering the tuple $(G_f, H, F'/K')$, let $V_0 \in \Irr_{F'}^{\rtimes}(G_f)$ correspond to $W_{F'}$. We have by Corollary \ref{cor:basechangedescofSL} that
    \[
        f \cdot V_{L'} = V_0^{\frac{m_K^F(W)}{[F' : F] \cdot m_{K'}^{F'}(W_{F'})}},
    \]
    and therefore that
    \[
    \Deg(\End_{F' \rtimes G_f}(f \cdot V_{L'})) = \Deg(\End_{F' \rtimes G_f}(V_0)) = m_{K'}^{F'}(W_{F'}),
    \]
    the final equality by Proposition \ref{prop:descofcentre}, as $\End_{F'[H]}(W_{F'}) = F'$ and $W_{F'}$ is fixed by $\Gamma_f$. The divisibility relations come from Proposition \ref{prop:propertiesSchurIndex} and Corollary \ref{cor:basechangedescofSL}.
\end{proof}

We can use this to give a local-global principle in our context.

\begin{cor}\label{cor:localglobal}
Suppose that $L/K$ is a Galois extension of number fields, $W \in \Irr_{L}(H)$ is fixed by $\Gamma$ and $\End_{L[H]}(W) = L$. For each place of $K$, let $w$ be a place of $L$ over $v$, and let $G_w$ be the stabiliser of of $L_w$ in $K_v \otimes_K L$. Then
\[
m_K^F(W) = \lcm_v m_{K_v}^{L_w}(W_{L_w}).
\]
In particular, $W$ extends to a semilinear representation of $G$ over $L$ if and only if $W_{L_w}$ extends to a semilinear representation of $G_w$ for any place $v$ of $K$.
\end{cor}

\begin{proof}
    Letting $D$ denote the central simple $K$-algebra $\End_{L \rtimes G}(V)$, $D$ is split over $K$ if and only if $m_{K}^L(W) = 1$, as $\Deg(D) = m_{K}^L(W)$ by Proposition \ref{prop:descofcentre}. For any place $v$ of $K$, similarly $m_{K_v}^{L_w}(W_{L_v}) = 1$ if and only if $W_{L_w}$ extends to a semilinear representation of $G_w$. We have that $\Deg(D_{K_v}) = m_{K_v}^{L_w}(W_{L_v})$ by Corollary \ref{cor:descdegreebasechange}, noting that $L_w$ is a connected component of $K_v \otimes_K L$, so plays the role of $F'$. The statement then follows from the fact that $\Deg(D)$ is the lowest common multiple of the $\Deg(D_{K_v})$ by \cite[Thm.\ VIII.2.6]{milneCFT}.
\end{proof}

We can use the results of this section to compute endomorphism rings of semilinear representations explicitly.

\begin{eg}\label{eg:C2C4ctd}
We continue Example \ref{eg:C2C4}, and suppose additionally that $K = \bQ$, so $L = \bQ(\sqrt{d})$ for some $d \in \bQ^{\times} \setminus (\bQ^{\times})^2$. The group $G = C_4$, with subgroup $H$, and natural quotient map
\[
C_4 / C_2 \xrightarrow{\sim} \Gamma \coloneqq \Gal(\bQ(\sqrt{d}) / \bQ).
\]
We take $W$ to be the non-trivial $1$-dimensional representation of $C_2$, and suppose that $m_K^L(W) = 2$, so $D \coloneqq \End_{L \rtimes C_4}(V)$ for corresponding $V \in \Irr_L^{\rtimes}(C_4)$. From Example \ref{eg:C2C4}, this is equivalent to the claim that $-1 \neq x^2 - d y^2$ for any $x, y \in \bQ^{\times}$. We wish to describe the quaternion algebra $D$ over $\bQ$.

For any place $v$ of $\bQ$, by Corollary \ref{cor:descdegreebasechange} $D_{\bQ_v}$ is split if and only if $m_{K_v}^{L_w}(W_{L_w}) = 1$, where $w$ is a place of $L$ above $v$. Set $L_v \coloneqq K_v \otimes_K L$.

When $L_v$ is a field, $L_v = K_v(i)$, and therefore by Example \ref{eg:C2C4}, $m_{K_v}^{L_v}(W_{L_v}) = 1$ if and only if $-1 \neq x^2 - d y^2$ for any $x, y \in \bQ_v^{\times}$.

When $L_v$ is not a field, or equivalently the place $v$ splits completely, we see that $m_{K_v}^{L_w}(W_{L_w}) = 1$ by Corollary \ref{cor:descdegreebasechange}, taking $F = K_v$ and $F' = L_w$. We can rephrase this condition as follows. Writing $V^v \in \Irr_{L_v}^{\rtimes}(C_4)$ for the irreducible semilinear representation of $C_4$ over $L_v$ corresponding to $W$, $m(V^v) = m_{K_v}^{L_w}(W_{L_w})$ from the compatibility of Proposition \ref{prop:indandrescompwithconnequiv}. Similarly to Example \ref{eg:C2C4}, $m(V^v) = 1$ if and only if $W \cong f \cdot V^v$, if and only if there is an $a = x + \sqrt{d} y \in L_v$, $x, y \in K_v$, with
\[
f \cdot (x + \sqrt{d}y)(x - \sqrt{d}y) = f \cdot a \overline{a} = -1
\]
which, as $(x + iy)(x - iy) = x^2 - d y^2 \in K_v$, is the same as the condition that $x^2 - d y^2 = -1$ in $K_v$. Of course, as $m_{K_v}^{L_w}(W_{L_w}) = 1$ this is always satisfied, and indeed $\sqrt{d} \in L \subset L_w$ and $L_w = K_v$, so we can take $x = 0$ and $y = 1 / \sqrt{d}$.

Altogether, we see that for any place $v$ of $K$, $D_{K_v}$ is split if and only if $-1 = x^2 - d y^2$ for some $x,y \in K_v$. Therefore $D$ has the same splitting behaviour at all places as $(-1,d)_{\bQ}$, so $D \cong (-1,d)_{\bQ}$.
\end{eg}

\section{Character Theory for Semilinear Representations}\label{sect:charactertheory}

Suppose in this section that $G$ is finite, and $|H| \in L^{\times}$, which by Corollary \ref{cor:sscriterion} is equivalent to the claim that $\Rep_L^{\rtimes}(G)$ is semisimple, or that $\Rep_L(H)$ is semisimple. For $V,W \in \Rep_L(H)$, we can consider the inner product of characters
\[
    \langle \chi_V, \chi_W \rangle = \frac{1}{|H|} \sum_{h \in H}\chi_V(h) \chi_W(h^{-1}),
\]
which satisfies \cite[Lem.\ 3.20(b)]{LOR}
\[
\langle \chi_V, \chi_W \rangle = \dim_L \Hom_{L[H]}(V,W) \cdot 1_L,
\]
where the multiplication by $1_L$ is to ensure this is interpreted as an equality in $L$.
\begin{defn}
    For $V \in \Rep_L^{\rtimes}(G)$, we define \emph{the character of $V$}
    \[
        \chi_V \colon H \rightarrow L
    \]
    to be the character of the linear representation $V|_H \in \Rep_L(H)$.
\end{defn}

From Corollary \ref{cor:detectisom} and Corollary \ref{cor:homsetisom} the following are immediate.

\begin{cor}\label{cor:innerproducthomset}
    For $V, W \in \Rep_L^{\rtimes}(G)$, $ \langle \chi_V, \chi_W \rangle = \dim_K \Hom_{L \rtimes G}(V,W) \cdot 1_L$.
\end{cor}

\begin{cor}
If $V, W \in \Rep_L^{\rtimes}(G)$ and $\charfield(L) = 0$, then $V \cong W$ if and only if $\chi_V = \chi_W$.
\end{cor}

We can also describe the characters of the irreducible objects of $\Rep_L^{\rtimes}(G)$. Let $\Fun(H/\!/H, L)$ denote the $L$-algebra of $L$-valued functions on $H$ which are invariant under conjugation by $H$.

\begin{defn}
We let $G$ act on $\Fun(H/\!/H, L)$ by
\[
g * f \coloneqq \sigma_{g}(f(g^{-1} - g))
\]
for any $g \in G$.
\end{defn}
On $H$ this action is trivial, and therefore this induces an action of $\Gamma$ on $\Fun(H /\!/ H, L)$ via the isomorphism $G / H \xrightarrow{\sim} \Gamma$. This is compatible with the action of $\Gamma$ on isomorphism classes of objects in $\Rep_L^{\rtimes}(G)$, in the sense that for any $V \in \Rep_{L}^{\rtimes}(G)$,
\[
    \chi_{g * V} = g * \chi_V.
\]
In particular, we see that the action of $\Gamma$ on $\Fun(H /\!/ H, L)$ preserves the subset $\Char(H,L)$ of characters of objects of $\Rep_{L}(H)$, and the smaller subset $\Irr(H,L)$ of characters of of irreducible representations of $H$ over $L$. Note that $\Irr(H,L)$ forms a basis for $\Fun(H /\!/ H, L)$ whenever $L$ splits $H$ \cite[Cor.\ 3.21]{LOR}, but need not span $\Fun(H /\!/ H, L)$ in general. However, it is still true that the character of any irreducible representation of $H$ over $L$ is non-zero, irreducible representations are determined by their characters, and that the set $\Irr(H,L)$ is $L$-linearly independent \cite[Cor.\ 9.22]{ISA}.

\begin{defn}
    For $\chi \in \Irr(H,L)$, we set $m_K^L(\chi) \coloneqq m_K^L(W)$, where $W \in \Irr_L(H)$ has $\chi_W = \chi$.
\end{defn}

The numbers $m_K^L(\chi)$ inherit all the properties of $m_K^L(W)$ from Proposition \ref{prop:propertiesSchurIndex}. For example, they are constant on the orbits under the action of $\Gamma$, and satisfy certain divisibility relations.

\begin{thm}\label{thm:mainthmchartheory}
    Let $\sC_1, ... , \sC_r$ denote the orbits of the action of $\Gamma$ on $\Irr(H,L)$. Then there are $r$ irreducible semilinear representations $V_1, ... , V_r$ of $G$ over $L$, and for any $k = 1, ... , r$, $V_k$ has character
    \[
        \psi_k \coloneqq m_K^L(\chi_k) \cdot \sum_{\chi \in \sC_k} \chi,
    \]
    where $\chi_k$ is any representative of the orbit $\sC_k$. Moreover, writing any $V, W \in \Rep_{L}^{\rtimes}(G)$ as
    \[
        V \cong V_1^{\oplus a_1} \oplus \cdots \oplus V_r^{a_r}, \qquad W \cong V_1^{\oplus b_1} \oplus \cdots \oplus V_r^{b_r}
    \]
    for unique $a_i, b_i \geq 1$, and setting $D_k \coloneqq \End_{L \rtimes G}(V_k)$ for $k = 1, ... ,r$, 
    \[
        \langle \chi_V, \chi_W \rangle = (a_1 b_1 \dim_K D_1 + \cdots + a_r b_r \dim_K D_r) \cdot 1_L, 
    \]
    and each
    \[
        \dim_K D_k \cdot 1_L = \langle \psi_k, \psi_k \rangle = m_K^L(\chi_k)^2 \cdot |\sC_k| \cdot \langle \chi_k, \chi_k \rangle.
    \]
\end{thm}

\begin{proof}
    This follows directly from Theorem \ref{thm:mainthm}, Corollary \ref{cor:innerproducthomset}, and the fact that $\Rep_L^{\rtimes}(G)$ is semisimple. The only thing to note is that $\langle \chi, \chi \rangle = \langle \chi', \chi' \rangle$ for any $\chi, \chi'$ in the same $\Gamma$-orbit on $\Irr(H,L)$, which follows from the fact that
    \[
    \End_{L[H]}(W) = \End_{L[H]}(g * W)
    \]
    for any $W \in \Rep_{L}(H)$ and $g \in G$, which is clear from the definition of $g*W$.
\end{proof}

From Corollary \ref{cor:finitefield}, we can go further is when $L$ is finite.

\begin{thm}\label{thm:finitefieldtrivialschurindex}
    Suppose that $L$ is finite. Then $m_K^L(\chi_k) = 1$ for all $k = 1, ... , r$.
\end{thm}

In particular, the $\Gamma$-set $\Irr(H,L)$ completely describes the category $\Rep_L^{\rtimes}(G)$ when $L$ is finite.

\begin{cor}
    Suppose that $W \in \Irr_L^{\rtimes}(G)$. Then $\chi_W \neq 0$ if and only if $m(W) \cdot 1_L \neq 0$, and those characters $\chi_W$ where $\chi_W \neq 0$ are distinct and $L$-linearly independent. In particular, the association
    \[
        \Irr_L^{\rtimes}(G) \rightarrow \Fun(H/\!/H, L), \qquad W \mapsto \chi_W,
    \]
    is injective if $L$ has characteristic $0$ or if $L$ is finite.
\end{cor}
\begin{proof}
    This follows directly from Theorem \ref{thm:mainthmchartheory}, Theorem \ref{thm:finitefieldtrivialschurindex}, and the fact that $\chi_V$ for $V \in \Irr_L(H)$ are all non-zero, distinct, and $L$-linearly independent \cite[Cor.\ 9.22]{ISA}.
\end{proof}

\section{Example: $\Rep_L^{\rtimes}(S_3)$ for $L / K$ of Degree $2$ Revisited}\label{sect:secondexampleS3}

In this section we give an example of the theory of the previous sections, to describe $\Rep_L^{\rtimes}(S_3)$ for any field $L$ where $S_3$ acts on $L$ with kernel $H \coloneqq \langle (123) \rangle$ through an order two automorphism, so that $L / K$ has degree two. We will see that this recovers the classification of Section \ref{sect:firstexampleS3} in characteristic $0$ deduced by more hands-on methods.

Let $x \in S_3 \setminus H$ be any element, and let $\omega$ be a primitive third root of unity in $\overline{L}$.

\subsection{Characteristic $\neq 3$} Let us first suppose that $\text{char}(K) \neq 3$, so $\Rep_L^{\rtimes}(S_3)$ is semisimple by Corollary \ref{cor:sscriterion} and we may use the character theory of Section \ref{sect:charactertheory}. From Theorem \ref{thm:mainthm} we need to understand the character table of $H$ over $L$, and its action of $\Gamma = \Gal(L/K)$. There are three possible cases, which exactly match and explain the three cases encountered in Section \ref{sect:firstexampleS3}. 

\vspace{0.6em}

\textbf{Case 1:} $\omega \in K$. Then the character table of $H$ is
\[
\begin{array}{c|ccc}
 & 1 & (123) & (132) \\ \hline
\chi_1 & 1 & 1 & 1 \\
\chi_{\omega} & 1 & \omega & \omega^2 \\
\chi_{\omega^2} & 1 & \omega^2 & \omega \\
\end{array}
\]
and, as all entries are contained inside $K$, the action of $\Gal(L/K)$ sends $\chi \mapsto \chi^x$. Therefore, there are two orbits, $\{\chi_1\}$ and $\{\chi_{\omega}, \chi_{\omega^2}\}$, corresponding to some $V_1$ and $V_2$ in $\Irr_L^{\rtimes}(S_3)$ respectively. For each we can compute $m(V_i)$.

For $V_1$, $\chi_1$ is the restriction of the trivial semilinear representation $L_{\text{triv}}$. In particular, $\chi_1$ is extendable to a semilinear representation of $S_3$ over $L$, and so $m(V_1)= 1$. Furthermore, $V_1 = L_{\text{triv}}$.

For $V_2$, note that $|\Gamma_{\chi_{\omega}}| = 1$, and so $m(V_2) = 1$ as $m(V_2) \mid |\Gamma_{\chi_{\omega}}|$ by Theorem A. What this means concretely is that $\chi_{\omega} + \chi_{\omega^2}$ uniquely extends to a semilinear representation, $V_2$, of $S_3$ over $L$.

Because $L$ splits $H$, we can also compute the endomorphism rings of $V_1$ and $V_2$. We have
\[
\End_{L \rtimes S_3}(V_1) \cong L^{\Gal(L/K)} = K
\]
by Proposition \ref{prop:descofcentre}, and similarly
\[
\End_{L \rtimes S_3}(V_2) \cong L^{\Gal(L/L)} = L.
\]
Note that this second isomorphism $L \xrightarrow{\sim} \End_{L \rtimes S_3}(V_2)$ is not the natural action of $L$ on $V_2$.

\vspace{0.6em}

\textbf{Case 2:} $\omega \in L \setminus K$. In this case, the character table of $H$ is the same as in Case 1, but now the natural Galois action of $\Gal(L/K)$ on $L$ swaps $\omega$ and $\omega^2$. In particular, the induced action of $\Gal(L/K)$ on the character table of $H$ is trivial. Therefore, all orbits of the action of $\Gal(L/K)$ are singletons, and there are three irreducible representations $V_1$, $V_{\omega}$ and $V_{\omega^{2}}$, which corresponds each to the characters $\chi_1, \chi_{\omega}$ and $\chi_{\omega^{2}}$ respectively.

As $m(V_{\omega^i}) \mid |\Gamma_{\chi_{\omega^i}}|$ by Theorem A, we have that $m(V_{\omega^i}) \in \{1,2\}$. Similarly to Case 1 above, $\chi_1$ is the restriction of the trivial representation $L_{\text{triv}}$, and so $m(V_1) = 1$. To compute $m(V_{\omega})$ and $m(V_{\omega^2})$ however, this is as far as Theorem A alone gets us.

To determine $m(V_{\omega})$ and $m(V_{\omega^{2}})$ we use the induction functor
\[
L \otimes_K - \colon \Rep_K(S_3) \rightarrow \Rep_L^{\rtimes}(S_3)
\]
of Section \ref{sect:classicalreps}. The character table of $S_3$ over $K$ is
\[
\begin{array}{c|rrr}
 & (1) & (12) & (123) \\ \hline
\psi_1 & 1 & 1 & 1 \\
\psi_2 & 1 & -1 & 1 \\
\psi_3 & 2 & 0 & -1 \\
\end{array}
\]
Let $U$ denote the irreducible $K$-linear representation of $S_3$ with character $\psi_3$, and suppose that
\[
L \otimes_K U = V_1^{\oplus a_1} \oplus V_{\omega}^{\oplus a_{\omega}} \oplus V_{\omega^2}^{\oplus a_{\omega^2}}
\]
for integers $a_1, a_{\omega}, a_{\omega^2} \geq 0$. Considering the character of the restriction to $H$, we therefore have that
\begin{align*}
    \chi_{\omega} + \chi_{\omega^2} = \psi_3|_H = a_1 m(V_1) \chi_1 + a_{\omega} m(V_{\omega}) \chi_{\omega} + a_{\omega^2} m(V_{\omega^2}) \chi_{\omega^2}.
\end{align*}
Comparing coefficients, we have that $a_{\omega} m(V_{\omega}) = 1 = a_{\omega^2} m(V_{\omega^2})$, and therefore $m(V_{\omega}) = 1 =  m(V_{\omega^2})$.

Because $L$ splits $H$, we can use Proposition \ref{prop:descofcentre} to see that the endomorphism rings of $V_1$, $V_{\omega}$ and $V_{\omega^{2}}$ are all $K$.

\vspace{0.6em}

\textbf{Case 3:} $\omega \not\in L$. Now, for $\chi \coloneqq \chi_{\omega} + \chi_{\omega^2}$ the character table of $H$ over $L$ is
\[
\begin{array}{c|rrr}
 & 1 & (123) & (132) \\ \hline
\chi_1 & 1 & 1 & 1 \\
\chi & 2 & -1 & -1 \\
\end{array}
\]
which has trivial action of $\Gal(L/K)$. Therefore, both irreducible characters $\chi_1, \chi$ correspond to a unique irreducible semilinear representations $V_1, V$ respectively of $S_3$ over $L$. As in Case 1, $V_1$ is the trivial semilinear representation $L_{\text{triv}}$ of $S_3$, with endomorphism ring $K$ and $m(V_1) = 1$, and similarly to above we either have that $m(V) = 1$ or $m(V) = 2$.

To determine $m(V)$, we again use the functor $L \otimes_K -$. Observing that $\chi = \psi_3|_H$, we see that $L \otimes_K U$ extends $\chi$, and therefore $m(V) = 1$ and $V = L \otimes_K U$ is irreducible.

\subsection{Characteristic 3}

When $L$ has characteristic $3$, $\Rep_L^{\rtimes}(S_3)$ is no longer semisimple, and we need to understand $\Ind_L^{\rtimes}(S_3)$. From Theorem \ref{thm:mainthm}, this is identified with $\Ind_L(H) / \Gal(L/K)$. The set $\Ind_L(H)$ consists of three indecomposble representations $V_1, V_2, V_3$ of dimensions $1,2,3$ respectively, and thus the action of $\Gal(L/K)$ on $\Ind_L(H)$ must be trivial as these dimensions are distinct. It remains to determine the semilinear Schur indices $m(V_i)$, which we now show are all equal to $1$.

Each $V_i$ is defined over $K$, being given by the matrix representations where $(123)$ acts via
\[
(1), \qquad \begin{pmatrix} 1 & 1 \\ 0 & 1 \end{pmatrix}, \qquad \begin{pmatrix} 1 & 1 & 0 \\ 0 & 1 & 1 \\ 0 & 0 & 1 \end{pmatrix}
\]
respectively. Each extends to a linear representation of $S_3$ over $K$, where $(12)$ acts via the matrices
\[
(1), \qquad \begin{pmatrix} 1 & 0 \\ 0 & 2 \end{pmatrix}, \qquad \begin{pmatrix} 2 & 1 & 0 \\ 0 & 1 & 1 \\ 0 & 0 & 2 \end{pmatrix}
\]
respectively, and therefore the same matrices define extensions $\tilde{V_i}$ of each $V_i$ to $\Rep_L^{\rtimes}(S_3)$.

In particular, $\tilde{V_1}, \tilde{V_2}, \tilde{V_3}$ is a complete list of all indecomposable objects of $\Rep_L^{\rtimes}(S_3)$.

\subsection{General Algorithm}\label{sect:simplealgorithm}
It is natural to ask if one can always follow the same process used in this section to determine all $m_K^L(W)$ when $G$ is finite and $L$ has characteristic $0$. In general, for an orbit $\sO$ of the action of $\Gamma$ on $\Irr(H,L)$ we can consider
    \[
    \chi_{\sO} = \sum_{\chi \in \OO} \chi.
    \]
    Then for any $U \in \Irr_K(G)$, with character $\psi_U$, $(L \otimes_K U)|_H$ has character $\psi_U|_H$, and
    \[
        \psi_U|_H = \sum_{\sO} a_{U, \sO} \cdot \chi_{\OO}
    \]
    for unique $a_{U, \sO} \geq 0$. From Proposition \ref{prop:propertiesSchurIndex}(4) we have that $m_K^L(\OO) \mid a_{U, \sO}$, and so
    \[
        m_K^L(\OO) \mid \hcf \{ a_{U, \sO} \mid U \in \Irr_K(G)\}
    \]
    where $m_K^L(\OO)$ is the semilinear Schur index of an element of $\OO$. We also know that for a fixed orbit $\sO$, not all $a_{U, \sO}$ can be zero by Corollary \ref{cor:allsubobjects}.
    
    In the examples of this section, this divisibility relation always turned out to always be an equality. However, this is not true in general, as the following example shows.

    \begin{eg}\label{eg:C2C4algfails}
        Let $L / K$ be $\bQ(\sqrt{2}) / \bQ$, $G = C_4 = \langle y \rangle$ with natural surjection to $\Gamma = \Gal(\bQ(\sqrt{2}) / \bQ)$ with kernel $H = \langle x \rangle$, $y^2 = x$. The three irreducible characters of $C_4$ over $\bQ$ are 
        \[
        \begin{array}{c|rrrr}
        & (1) & (y) & (y^2) & (y^3) \\ \hline
        \psi_1 & 1 & 1 & 1 & 1 \\
        \psi_2 & 1 & -1 & 1 & -1 \\
        \psi_3 & 2 & 0 & -2 & 0 \\
        \end{array}
        \]  
        The two orbits of $\Gamma$ on $\Irr_L(C_2)$ are $\chi_1$ and $\chi_{-1}$, and we can compute that
        \begin{align*}
            \psi_1|_H &= \chi_1, \\
            \psi_2|_H &= \chi_1, \\
            \psi_3|_H &= 2 \cdot \chi_{-1}.
        \end{align*}
        However $m_{K}^L(\chi_{-1}) = 1$, by Example \ref{eg:C2C4}, as $-1 = x^2 - 2 y^2$ has a solution over $\bQ$, with $x = 1, y = 1$.
    \end{eg}

    In particular, this algorithm, whilst still giving restrictions on the $m_K^L(W)$, won't determine all semilinear Schur indices in general. A more precise process for determining the semilinear Schur indices $m_K^L(W)$ from linear representation-theoretic data when $L/K$ of degree two is the subject of the forthcoming work \cite{RUMTAY4}.

\section{Conjugacy Classes and the Number of Semilinear Representations}\label{sect:countcc}

In this section we suppose that
\begin{enumerate}
    \item $G$ is finite, $L$ is a field, with $|H| \in L^\times$, and
    \item $L = K(\mu_n)$ for some $n \geq 1$ with $\text{exp}(H) \mid n$.
\end{enumerate}
In this situation, there is a canonical embedding
\[
\epsilon \colon \Gamma \hookrightarrow (\bZ / n \bZ)^\times.
\]
which is determined by the property that $\gamma(\zeta) = \zeta^{\epsilon(\gamma)}$ for any $n$th root of unity $\zeta$ and all $\gamma \in \Gamma$. Using this embedding, there is a left action of $G \times G$ on $H$ by the formula:
\[
    (g_1, g_2) * h = g_1 h^{\epsilon(\sigma_{g_2}^{-1})}g_1^{-1},
\]
This induces an action of $G / H \times G / H$ and hence $\Gamma \times \Gamma$ on $\Cl(H)$. Similarly, there is a left action of $G \times G$ on $\Irr_L(H)$, where $(g_1, g_2) * V$ is $V$ as an abelian group but with $L \rtimes H$-module structure defined by
\[
\lambda h * v \coloneqq \sigma_{g_2^{-1}}(\lambda) \cdot (g_1^{-1} h g_1)(v),
\]
which induces an action of $G / H \times G / H$ and hence $\Gamma \times \Gamma$ on $\Irr_L(H)$. The diagonal action is the action of $\Gamma$ on $\Irr_L(H)$ considered before. If $V \in \Irr_L(H)$ has character $\chi$, then we set
\[
(g_1, g_2) * \chi \coloneqq \chi_{(g_1,g_2) * V} = \sigma_{g_2}(\chi(g_1^{-1} - g_1)).
\]
These actions are compatible in the following sense.
\begin{lemma}\label{lem:compatibleactions}
    For any $\chi \in \Irr(H,L)$, $h \in H$ and $\underline{\gamma} \in \Gamma \times \Gamma$,
    \[
        (\underline{\gamma}*\chi)(h) = \chi(\underline{\gamma}^{-1} * h).
    \]
\end{lemma}

\begin{proof}
    Expressing $\underline{\gamma}$ as the reduction of $(g_1, g_2) \in G \times G$, choose a basis of $V$ (the representation associated to $\chi$) where $g_1^{-1}h g_1$ acts diagonally. Then $\chi(g_1^{-1} h g_1) = \zeta_1 + \cdots + \zeta_r$ is the sum of $n$th roots of unity, using the assumption that $n$ divides the exponent of $H$. We may then compute that
    \begin{align*}
        (\underline{\gamma}*\chi)(h) &= \sigma_{g_2}(\chi(g_1^{-1} h g_1)), \\
        &= \sigma_{g_2}(\zeta_1 + \cdots + \zeta_r), \\
        &= \zeta_1^{\epsilon(\sigma_{g_2})} + \cdots + \zeta_r^{\epsilon(\sigma_{g_2})}, \\
        &= \chi((g_1^{-1} h g_1)^{\epsilon(\sigma_{g_2})}), \\
        &= \chi(g_1^{-1} h^{\epsilon(\sigma_{g_2})} g_1), \\
        &= \chi(\underline{\gamma}^{-1} * h),
    \end{align*}
    as required.
\end{proof}

Consider the $L$-vector space $\Fun(H / \! /H, L)$ of functions from $H$ to $L$ which are conjugation invariant. This has a natural action of $G \times G$ where
\[
(g_1, g_2) * f \coloneqq \sigma_{g_2}(f(g_1^{-1} - g_1)).
\]
This induces a natural action of $G / H \times G /H$ and hence of $\Gamma \times \Gamma$ on $\Fun(H /\!/ H, L)$, the subspace of functions invariant by conjugation by $H$. Because $L$ splits $H$, this has a basis given by irreducible characters $\chi \in \Irr(H,L)$ \cite[Cor.\ 3.21]{LOR}, and with respect to this basis $\Fun(H /\!/ H, L)$ is nothing but the permutation representation $L[\Irr_L(H)]$ of $\Gamma \times \Gamma$. Similarly, the centre $Z(L[H])$ of $L[H]$ has a basis given by conjugacy class sums $c_h$, which induces a permutation action of $\Gamma \times \Gamma$ on $Z(L[H])$, which is explicitly given by
\[
\underline{\gamma} \cdot \left( \sum_{h} \lambda_h h \right) = \sum_h \lambda_{\underline{\gamma}^{-1} * h} h.
\]

\begin{lemma}
    The perfect pairing
    \[
    \Fun(H /\!/ H, L) \times Z(L[H]) \rightarrow L, \qquad \left(f, \sum_h \lambda_h h\right) = \sum_h \lambda_h f(h)
    \]
    is $\Gamma \times \Gamma$-equivariant. In particular, this induces a $\Gamma \times \Gamma$-equivariant isomorphism
    \[
        \Fun(H /\!/ H, L) \xrightarrow{\sim} \Hom_L(Z(L[H]), L), \qquad f \mapsto \Phi_f.
    \]
\end{lemma}

\begin{proof}
    Expressing any $f \in \Fun(H /\!/ H, L)$ as a sum of $\chi \in \Irr_L(H)$, this follows from Lemma \ref{lem:compatibleactions}.
\end{proof}

For any group which acts on a finite set $X$, there is a equivariant perfect pairing
\[
L[X] \times L[X] \rightarrow L, \qquad \left(\sum_x \lambda_x x, \sum_y \mu_y y\right) = \sum_{x} \lambda_x \mu_x,
\]
which induces a canonical isomorphism of representations
\[
L[X] \xrightarrow{\sim} \Hom_L(L[X],L),
\]
where $x \in X$ is sent to its indicator function. In particular:

\begin{cor}\label{cor:permrepsisom}
The permutation representations of $\Gamma \times \Gamma$ associated to $\Irr_L(H)$ and $\Cl(H)$ are canonically isomorphic.
\end{cor}

\begin{remark}
Even though the permutation representations of the $\Gamma \times \Gamma$-sets $\Irr_L(H)$ and $\Cl(H)$ are the same, it is not true in general that $\Irr_L(H)$ and $\Cl(H)$ are isomorphic as $\Gamma \times \Gamma$-sets. For example, for the field extension $\bQ(i)/ \bQ$, an example of a pair $(G,H)$ of the smallest order where this fails is $\text{SmallGroup}(32,9)$ with $H$ the unique subgroup of $G$ isomorphic to $\text{SmallGroup}(16,4)$. The $\Gamma \times \Gamma = K_4$-sets $\Irr_L(H)$ and $\Cl(H)$ are non-isomorphic as one can compute that $|\Irr_L(H)^{K_4}| = 4$ whereas $|\Cl(H)^{K_4}| = 6$. This condition is natural to check, it being true that for any degree two extension $L/K$, if $\Irr_L(H)$ and $\Cl(H)$ are non-isomorphic then one must necessarily have that $|\Irr_L(H)^{K_4}| \neq |\Cl(H)^{K_4}|$ by \cite[Ex.\ 13.5]{SERRE} as $K_4$ is the only non-cyclic subgroup of $K_4$.

However, when $L / K$ is $\bC / \bR$, the $\Gamma \times \Gamma$-sets $\Irr_L(H)$ and $\Cl(H)$ \emph{are} isomorphic \cite[Cor.\ 1.5]{RUMTAY3}.
\end{remark}

Now we consider the diagonal action of $\Gamma$ on $\Irr_L(H)$ and $\Cl(H)$, which allows us to count the number of irreducible semilinear representations of $G$ from just the group theory of $G$. This generalises \cite[Thm.\ 5.6]{RUMTAY} for $\bC / \bR$, and \cite[\S 12.4, Cor.\ 2]{SERRE} in the split case (cf.\ Section \ref{sect:classicalSchurIndex}).

\begin{cor}
The number of irreducible semilinear representations of $G$ over $L$ is the same as the number of $G$-orbits on $\Cl(G)$: $|\Irr_L^{\rtimes}(G)| = |\Cl(H) / \hspace{0.1em}\Gamma|$.
\end{cor}

\begin{proof}
    This follows from \cite[Ex.\ 13.5]{SERRE}, Corollary \ref{cor:permrepsisom}, and the bijection between $\Irr_L(H) / \Gamma$ and $\Irr_L^{\rtimes}(G)$ of Theorem \ref{thm:mainthm}. 
\end{proof}

\section{Cohomological Interpretation}\label{sect:cohomology}

We can interpret the main results of Section \ref{sect:mainresults} in terms of Galois cohomology.

From the matrix form of semilinear representations described in Section \ref{sect:matrixdesc}, for any $n \geq 1$ we obtain an identification between $\HH^1(G, \GL_n(L))$ and the set of isomorphism classes of $n$-dimensional semilinear representations of $G$ over $L$. From the exact sequence
\[
1 \rightarrow H \rightarrow G \rightarrow \Gamma \rightarrow 1
\]
we have the restriction-inflation exact sequence of pointed sets \cite[\S 5.8(a)]{SERGC}
\[
1 \rightarrow \HH^1(\Gamma, \GL_n(L)^H) \rightarrow \HH^1(G,\GL_n(L)) \rightarrow \HH^1(H,\GL_n(L))^{\Gamma},
\]
which has trivial first term
\[
\HH^1(\Gamma, \GL_n(L)^H) = \HH^1(\Gamma, \GL_n(L)) = 1
\]
by Hilbert's Theorem 90. Furthermore, $\HH^1(H,\GL_n(L))$ is the set of isomorphism classes of $n$-dimensional linear representations of $H$ over $L$, and the action of $\Gamma$ on $H^1(H,\GL_n(L))$ corresponds to the action of $\Gamma$ on isomorphism classes of objects of $\Rep_{L}(H)$ as described in Definition \ref{def:Gaction}.

The fact that the image of $\HH^1(G,\GL_n(L))$ is contained in $\HH^1(H,\GL_n(L))^{\Gamma}$ amounts to the fact that any representation of $H$ obtained by restricting a semilinear representation will be fixed by the action of $\Gamma$. The fact that
\[
1 \rightarrow \HH^1(G,\GL_n(L)) \rightarrow \HH^1(H,\GL_n(L))^{\Gamma}
\]
is an exact sequence of pointed sets amounts to the fact that the only $n$-dimensional semilinear representation which restricts to the trivial $n$-dimensional linear representation of $H$ is the trivial semilinear representation. This can be considered as a weaker form of Corollary \ref{cor:detectisom}.

\subsection{One-Dimensional Representations}
We can say more when $n = 1$. In this case $\GL_1(L) = L^{\times}$ is abelian, and the pointed sets considered above are abelian groups. Furthermore, from the Lyndon-Hochschild-Serre spectral sequence there is an exact sequence
\[
0 \rightarrow \HH^1(\Gamma, (L^{\times})^H) \rightarrow \HH^1(G,L^{\times}) \rightarrow \HH^1(H,L^{\times})^{\Gamma} \xrightarrow{\fT} \HH^2(\Gamma, L^{\times}) \rightarrow \HH^2(G, L^{\times})
\]
of low degree terms which extends the inflation-restriction sequence above \cite[Prop.\ 1.6.7, Thm.\ 2.4.1]{NSW}. Again the first term
\[
\HH^1(\Gamma, (L^{\times})^H) = \HH^1(\Gamma, L^{\times}) = 0
\]
is trivial by Hilbert's Theorem 90, and so we have an exact sequence of abelian groups
\[
0 \rightarrow \HH^1(G,L^{\times}) \rightarrow \HH^1(H,L^{\times})^{\Gamma} \xrightarrow{\fT} \HH^2(\Gamma, L^{\times}).
\] 

The group $\HH^1(G,L^{\times})$ classifies $1$-dimensional semilinear representations of $G$ over $L$, and, as $H$ acts trivially on $L^{\times}$, the group $\HH^1(H,L^{\times})^{\Gamma}$ is canonically identified with the set $\Hom(H, L^{\times})^{\Gamma}$ of linear characters of $H$ over $L$ which are fixed by the action of $\Gamma$, where
\[
    (\gamma * \chi)(h) \coloneqq \gamma(\chi(g(\gamma)^{-1} h g(\gamma)))
\]
for any $\gamma \in \Gamma$, $h \in H$ and $\chi \in \Hom(H, L^{\times})$, and $g(\gamma) \in G$ is any lift of $\gamma$. The injectivity of 
\[
\HH^1(G,L^{\times}) \rightarrow \Hom(H, L^{\times})^{\Gamma}
\] 
amounts to the statement that any linear character of $H$ over $L$ which is invariant under $\Gamma$ has at most one extension of a semilinear character of $G$ over $L$, which is Corollary \ref{cor:detectisom} for $1$-dimensional representations.

\subsection{The Transgression Map}
More interesting is the \emph{transgression map}, $\fT$. By \cite[Thm.\ 2.4.3]{NSW} this agrees with the map constructed in \cite[Prop.\ 1.6.6]{NSW}, which, following this construction in our situation, is explicitly defined as follows.

Let $(g_{\gamma})_{\gamma \in \Gamma}$ be a set of lifts in $G$ of $\gamma \in \Gamma$, or in other words a set of coset representatives of $H$ in $G$, which we choose with $g_1 = 1 \in G$. For any $\gamma_1, \gamma_2 \in \Gamma$, we have that
\[
    g_{\gamma_1} g_{\gamma_2} = g_{\gamma_1 \gamma_2} h_{\gamma_1, \gamma_2},
\]
for a unique $h_{\gamma_1, \gamma_2} \in H$. Then the transgression map $\fT$ is defined by
\[
\fT \colon \Hom(H, L^{\times})^{\Gamma} \rightarrow \HH^2(\Gamma, L^{\times}), \qquad \fT(\chi) \coloneqq [f_{\chi}],
\]
where $f_{\chi}$ is the $2$-cocycle
\[
f_{\chi} \colon \Gamma \times \Gamma \rightarrow L^{\times}, \qquad f_{\chi}(\gamma_1, \gamma_2) = (\gamma_1 \gamma_2)(\chi(h_{\gamma_1, \gamma_2}))^{-1}.
\]

There is also another natural map from $\Hom(H, L^{\times})^{\Gamma}$ to $\HH^2(\Gamma, L^{\times})$, which arises from Theorem B. Writing $\Br(L/K)$ for the subgroup of classes of $\Br(K)$ which are split over the Galois extension $L$, there is an explicit group isomorphism
\[
\Psi \colon \Br(L/K) \xrightarrow{\sim} \HH^2(\Gamma, L^{\times})
\]
as described in \cite[\S IV.3]{milneCFT}. We therefore describe our map as a map
\[
\Phi \colon \Hom(H, L^{\times})^{\Gamma} \rightarrow \Br(L/K),
\]
which we define as follows. For any $\chi \in \Hom(H, L^{\times})^{\Gamma}$, let $L_{\chi} \in \Irr_{L}(H)$ denote the associated $1$-dimensional representation of $H$, and let $V_{\chi} \in \Irr_L^{\rtimes}(G)$ denote the corresponding semilinear representation of $G$ over $L$ (cf.\ Theorem B). Because $L_{\chi}$ is $1$-dimensional, we have that
\[
\End_{L[H]}(L_{\chi}) = L,
\]
and therefore the division algebra
\[
D_{\chi} \coloneqq \End_{L \rtimes G}(V_{\chi})
\]
has
\[
Z(D_{\chi}) \cong L^{\Gamma} = K
\]
by Corollary \ref{cor:splittingovergalois}, using the fact that $\Gamma$ fixes $\chi$. In particular, $D_{\chi}$ is a central division algebra over $K$, and we set 
\[
\Phi(\chi) \coloneqq [D_{\chi}] \in \Br(K).
\]
From Corollary \ref{cor:homsetisom} restriction defines an isomorphism
\[
L \otimes_K D_{\chi} \xrightarrow{\sim} \End_{L[H]}(V|_H),
\]
and by Theorem B we have that
\[
V_{\chi}|_H \cong L_{\chi}^{m(V_{\chi})},
\]
so 
\[
\End_{L[H]}(V|_H) \cong \End_{L[H]}(L_{\chi}^{m(V_{\chi})}) \cong M_{m(V_{\chi})}(L).
\]
This shows that $\Phi(\chi) \in \Br(L/K)$, and therefore that $\Phi$ is well-defined.

The following shows that the two maps $\fT$ and $\Phi$ agree.
\begin{prop}
    The transgression homomorphism factors as the composition $\fT = \Psi \circ \Phi$.
\end{prop}

\begin{proof}
    Let $\chi \in \Hom(H, L^{\times})^{\Gamma}$. To see that $\fT(\chi) = (\Psi \circ \Phi)(\chi)$, it is equivalent to show that 
    \[
    \fT(\chi)^{-1} = \Psi(\Phi(\chi))^{-1}.
    \]
    The map $\Psi$ is a homomorphism, so $\Psi(\Phi(\chi))^{-1} = \Psi(\Phi(\chi)^{-1})$, and $\Phi(\chi) = [D_{\chi}] \in \Br(L/K)$, so $\Phi(\chi)^{-1} = [D_{\chi}^{\text{op}}]$. We therefore want to show that
    \[
        \Psi([D_{\chi}^{\text{op}}]) = \fT(\chi)^{-1} = [f_{\chi}]^{-1} \in \HH^2(\Gamma, L^{\times}).
    \]
    To see this, let us recall how to compute $\Psi([D_{\chi}^{\text{op}}])$, following \cite[Thm.\ IV.3.11]{milneCFT}. First, one takes a central simple $K$-algebra $A$ which has dimension $|\Gamma|^2$ over $K$, and $[D_{\chi}^{\text{op}}] = [A]$. This can be found, following \cite[Cor.\ IV.3.6]{milneCFT}, by using the isomorphism
    \[
        L \otimes_K D_{\chi} \xrightarrow{\sim} \End_L(V_{\chi}),
    \]
    and setting $A \coloneqq C_{\End_{K}(V_{\chi})}(D_{\chi})$ to be the centraliser of $D_{\chi}$ in $\End_{K}(V_{\chi})$. Then, as described in \cite[Thm.\ IV.3.11]{milneCFT}, one takes elements $(e_{\gamma})_{\gamma \in \Gamma}$ in $A^\times$ which satisfy 
    \begin{equation}\label{eqn:semilinmilne}
        e_{\gamma} \cdot a = \gamma(a) \cdot e_{\gamma}
    \end{equation}
    for any $a \in L$, and sets $\Psi([D_{\chi}^{\text{op}}]) \coloneqq \Psi([A]) \coloneqq [f]$, where
    \[
        f \colon \Gamma \times \Gamma \rightarrow L^\times, \qquad f(\gamma_1, \gamma_2) \coloneqq e_{\gamma_1} \cdot e_{\gamma_2} \cdot e_{\gamma_1\gamma_2}^{-1} \in L^\times.
    \]
    For our $A = C_{\End_{K}(V_{\chi})}(D_{\chi})$ we can find these elements explicitly. Writing $\rho \colon G \rightarrow \End_K(V_{\chi})$ for the action map of the semilinear $G$-representation $V_{\chi}$, we may set 
    \[
        e_{\gamma} \coloneqq \rho(g_{\gamma}),
    \]
    where the $(g_{\gamma})_{\gamma \in \Gamma}$ are the lifts of elements of $\Gamma$ we used to define $f_{\chi}$ and $\fT(\chi)$ above. These commute with all elements of $D_{\chi}$ directly from the definition of $D_{\chi}$, and so lie in $A$. They also satisfy the equation (\ref{eqn:semilinmilne}). We may therefore compute that
    \begin{align*}
        f(\gamma_1, \gamma_2) &= \rho(g_{\gamma_1}) \cdot \rho(g_{\gamma_2}) \cdot \rho(g_{\gamma_1 \gamma_2})^{-1}, \\
        &= \rho(g_{\gamma_1} g_{\gamma_2}) \cdot \rho(g_{\gamma_1 \gamma_2})^{-1}, \\
        &= \rho(g_{\gamma_1\gamma_2} h_{\gamma_1, \gamma_2}) \cdot \rho(g_{\gamma_1 \gamma_2})^{-1}, \\
        &= \rho(g_{\gamma_1\gamma_2}) \cdot \chi(h_{\gamma_1, \gamma_2}) \cdot \rho(g_{\gamma_1 \gamma_2})^{-1}, \\
        &= (\gamma_1 \gamma_2)(\chi(h_{\gamma_1, \gamma_2})) \cdot  \rho(g_{\gamma_1\gamma_2})  \cdot \rho(g_{\gamma_1 \gamma_2})^{-1}, \\
        &= (\gamma_1 \gamma_2)(\chi(h_{\gamma_1, \gamma_2})).
        \end{align*}
    In particular, $f(\gamma_1, \gamma_2) = f_{\chi}(\gamma_1, \gamma_2)^{-1}$, and so $\Psi([D_{\chi}^{\text{op}}]) = \fT(\chi)^{-1} = [f_{\chi}]^{-1}$ as required.
\end{proof}

We get the following immediate consequence, which is not at all obvious from the definition of $\Phi$.

\begin{cor}\label{cor:isactuallyhom}
$\Phi \colon \Hom(H, L^{\times})^{\Gamma} \rightarrow \Br(L/K)$ is a group homomorphism.
\end{cor}

\begin{remark}
    We note that when $H \hookrightarrow G$ is split, the map $\Phi$ is zero, as explained in the comment following the proof of \cite[Thm.\ 2.4.4]{NSW}. One can also see this more directly as a consequence of Proposition \ref{prop:classicalSIdividesdim}.
\end{remark}

Corollary \ref{cor:isactuallyhom} has consequences for the numbers $m_K^L(\chi)$, for $\chi \in \Hom(H, L^{\times})^{\Gamma}$.

\begin{cor}\label{cor:1dimSLSchurrelation}
    For $\chi \in \Hom(H, L^{\times})^{\Gamma}$, the prime factors of $m_K^L(\chi)$ are a subset of the prime factors of $\Ord(\chi)$. In particular, if $\Ord(\chi)$ is coprime to $[L:K]$, then $m_K^L(\chi) = 1$. 
\end{cor}

\begin{proof}
    Because $\Phi$ is a group homomorphism, $\Ord(\Phi(\chi)) \mid \Ord(\chi)$. By \cite[Prop.\ 4.5.13]{GS}, we also have that $\Ord(\Phi(\chi)) \mid \Deg(\Phi(\chi))$ and $\Ord(\Phi(\chi))$ and $\Deg(\Phi(\chi))$ share the same prime factors. The result follows as $\Deg(\Phi(\chi)) = m_{K}^L(\chi)$ by Proposition \ref{prop:descofcentre} and $m_K^L(\chi) \mid [L:K]$ by Proposition \ref{prop:propertiesSchurIndex}.
\end{proof}

\begin{remark}
    Whenever $K$ is such that the period coincides with the index of elements of $\Br(K)$, one further has that $m_K^L(\chi) \mid \Ord(\chi)$. For example, this is true for local fields \cite[Rem.\ IV.4.4(b)]{milneCFT} and global fields \cite[Thm.\ VIII.2.6]{milneCFT}.
\end{remark}

\section{Realisation of Division Algebras}\label{sect:realisedivisionalg}

The Schur subgroup of a field $K$ is the subgroup $S(K)$ of $\Br(K)$ consisting of those classes $[A]$, where $A$ is a central simple algebra over $k$ for which there exists a surjection $K[H] \twoheadrightarrow A$ for some finite group $H$. In general $S(K) \neq \Br(K)$, and is typically much smaller, being explicitly described when $K$ has characteristic $0$ as the subgroup $C(K)$ of $\Br(K)$ consisting of classes represented by cyclotomic algebras \cite{YAM}.

In this section we show that the semilinear representations of the type we have considered in this paper are sufficiently more general than ordinary group algebras to realise all elements of $\Br(K)$. We also direct the reader towards \cite{MEIR}, where similar ideas are used to show that any central simple algebra over $K$ is Brauer equivalent to a quotient of a finite-dimensional Hopf algebra over $k$.

\begin{thm}\label{thm:realisationCSA}
    Suppose that $K$ be a field, and $A$ is a central simple algebra over $K$. Then there is a finite Galois extension $L / K$, a group $G$ and a surjective group homomorphism $G \rightarrow \Gal(L/K)$ such that $A$ is Brauer equivalent to a quotient of $L \rtimes G$. If the order of $[A]$ in $\Br(K)$ is coprime to $\charfield(K)$, then $G$ can be taken to be finite.
\end{thm}

\begin{proof}
   Let $L$ be a finite Galois extension that splits $A$, and set $\Gamma \coloneqq \Gal(L / K)$. Then from the isomorphism, $\HH^2(\Gal(L/K), L^{\times}) \cong \Br(L/K)$ \cite[Thm.\ IV.3.14]{milneCFT}, $[A] = [A(L/K, \alpha)]$ in $\Br(K)$ for some normalised cocycle $\alpha \colon \Gamma \times \Gamma \rightarrow L^{\times}$, where $A(L/K, \alpha)$ is the central simple $K$-algebra
    \[
        A(L/K, \alpha) = \bigoplus_{\gamma \in \Gamma} L \cdot u_{\gamma}
    \]
    for formal symbols $(u_{\gamma})_{\gamma \in \Gamma}$, with multiplication 
    \[
    \lambda u_{\gamma} \cdot \mu u_{\delta} = \lambda \gamma(\mu) \alpha(\gamma, \delta) u_{\gamma \delta}.
    \]
    We take
    \[
        G \coloneqq \{L^{\times} \cdot u_{\gamma} \mid \gamma \in \Gamma\} \subset A(L/K, \alpha)^\times,
    \]
    for which there is a natural surjective group homomorphism
    \[
        G \rightarrow \Gamma, \qquad \lambda \cdot u_{\gamma} \mapsto \gamma.
    \]
    Then we can define a $K$-algebra homomorphism 
    \[
       \phi \colon L \rtimes G \rightarrow A(L/K, \alpha), \qquad \lambda \cdot (\mu \cdot u_{\gamma}) \mapsto \lambda \mu \cdot u_{\gamma} 
    \]
    which is clearly surjective, and multiplicative from the definition of the action of $G$ on $L$ and the fact that $\alpha$ is normalised. This gives the first statement.

    For the second statement, write $G_K \coloneqq \Gal(K^{\text{sep}} / K)$. If $[A]$ has order $m$ in $\Br(K)$, and $m$ is coprime to the characteristic of $K$, then the natural inclusion $\HH^2(G_K, \mu_m(K^{\text{sep}})) \hookrightarrow \HH^2(G_K, (K^{\text{sep}})^{\times})$ has image $\HH^2(G_K, (K^{\text{sep}})^{\times})[m]$ \cite[Cor.\ 4.4.9]{GS}. In particular, under the isomorphism $\Br(K)[m] \cong \HH^2(G_K, (K^{\text{sep}})^{\times})[m]$ \cite[Cor.\ IV.3.16]{milneCFT}, $[A] = [A(L/K, \alpha)]$ for some finite Galois extension $L/K$ and normalised cocycle $\alpha \colon \Gamma \times \Gamma \rightarrow \mu_m(L)$, where again $\Gamma \coloneqq \Gal(L/K)$. Then we may proceed as above, taking $G$ to be the, now finite, group
    \[
    G \coloneqq \{\mu_m(L) \cdot u_{\gamma} \mid \gamma \in \Gamma\} \subset A(L/K, \alpha)^\times,
    \]
    which is a subgroup because $\alpha$ is valued in $\mu_m(L)$.
\end{proof}

This allows us to show that all central division algebras over $K$ arise as endomorphism rings of semilinear representations, and that all possible values for the semilinear Schur indices $m(V)$ occur.

\begin{cor}\label{cor:realisationdivalg}
    Suppose that $K$ be a field, and $D$ is a central division algebra over $K$. Then there is a finite Galois extension $L / K$, a group $G$, a surjective group homomorphism $G \rightarrow \Gal(L/K)$ and $V \in \Irr_L^{\rtimes}(G)$ such that $D \cong \End_{L \rtimes G}(V)$ and $m(V) = \Deg(D)$. If the order of $[D]$ in $\Br(K)$ is coprime to $\charfield(K)$, then $G$ can be taken to be finite.
\end{cor}

\begin{proof}
    Taking $A \coloneqq D^{\text{op}}$ in Theorem \ref{thm:realisationCSA}, we have a surjection $L \rtimes G \twoheadrightarrow M_n(D^{\text{op}})$ for some $n \geq 1$, and we can take $V \coloneqq (D^{\text{op}})^n$, which is irreducible as a $M_n(D^{\text{op}})$-module with endomorphism ring $D$. To see that $m(V) = \Deg(D)$, note that $V|_H$ is semisimple by Corollary \ref{cor:presss}, and that 
    \[
        L \otimes_K D = L \otimes_K \End_{L \rtimes G}(V) \xrightarrow{\sim} \End_{L[H]}(V|_H)
    \]
    by Corollary \ref{cor:homsetisom}. From how $L$ is defined in the proof of Theorem \ref{thm:realisationCSA}, $D^{\text{op}}$ is split over $L$, and therefore the same is true of $D$. In particular, $\End_{L[H]}(V|_H)$ is a matrix ring over $L$, and so the semisimple $V|_H$ has only one isotypic factor: $V|_H = W^{m(V)}$ for some $W \in \Irr_{L[H]}(H)$ with $\End_{L[H]}(W) = L$. Then the fact that $m(V) = \Deg(D)$ follows by comparing dimensions in the above isomorphism.
\end{proof}

\begin{remark}
    We point out that in Theorem \ref{thm:realisationCSA} (and therefore in Corollary \ref{cor:realisationdivalg}), the field $L$ can be taken to be any field that splits $A$. However, when one takes $G$ to be finite, one loses control the field $L$, and the field $L$ constructed by the proof in the case that $G$ is finite might have degree greater than that of a splitting field for $A$.

    For example, the group $\Br(\bQ)[2]$ parametrises quaternion algebras over $\bQ$. For any such quaternion algebra $D$, the proof realises $D$ as the endomorphism ring of some irreducible $V \in \Irr_L^{\rtimes}(G)$, for some Galois extension $L / \bQ$ and takes $H \coloneqq \mu_2(L) \cong C_2$.
    
    On the other hand, we can compute the quaternion algebras $D$ which arise from semilinear representations of degree two extensions $L / \bQ$ and group extensions $H \leq G$ of index two with $H \cong C_2$.
    
    There are two options for $H \leq G$: $C_2 \leq C_2 \times C_2$ or $C_2 \leq C_4$. The first has $m_{\bQ}^L(W) = 1$ for all $W \in \Irr_L(C_2)$ and degree two $L / \bQ$, and so does not realise any quaternion algebras over $\bQ$. For the second, $L = \bQ(\sqrt{d})$ for some $d \in \bQ^{\times} \setminus \bQ^{\times 2}$, and from Example \ref{eg:C2C4ctd} $D \cong (-1,d)_{\bQ}$ for some $a \in \bQ^{\times}$. However such $D$ are exactly those quaternion algebras which define elements of the proper subgroup $\Br(\bQ(i)/\bQ)[2]$ of $\Br(\bQ)[2]$.
\end{remark}

\begin{question}
The above remark shows that not all elements of $\Br(\bQ)[2]$ are realised when $L / \bQ$ has degree two and $H \cong C_2$, but all elements are realised when $L$ is allowed to be any extension of $\bQ$ (actually, computing explicitly with cocycles, one can always take $L$ to be at most degree $4$ over $\bQ$).

It is natural to ask if the same is true when instead the degree of $L$ is fixed and $H$ is allowed to grow: are all elements of $\Br(\bQ)[2]$ realised for degree two extensions $L / \bQ$, for a degree two extension $H \leq G$ of finite groups of arbitrary order?
\end{question}

\section{Extension to Infinite Galois Extensions}\label{sect:exttoinfinite}

We now give a brief indication of how the results of this paper can be extended to understand semilinear representations of infinite Galois extensions.

Let $L/K$ be a potentially infinite Galois extension, and suppose that $\sigma \colon G \twoheadrightarrow \Gal(L/K)$ is a continuous surjection of topological groups which is open.

\begin{eg}\label{eg:detandrednorm}
As an example to illustrate that such situations arise naturally, let $F$ be a finite extension of $\bQ_p$. From local class field theory that there is an isomorphism of topological groups 
\[
    \OO_F^\times \xrightarrow{\sim} \Gal(F^{\text{ab}} / F^{\text{nr}})
\]
where $F^{\text{nr}}$ is the maximal unramified extension of $F$, and $F^{\text{ab}}$ is the maximal abelian extension of $F$.

The determinant
\[
\det \colon \GL_n(\OO_F) \rightarrow \OO_F^{\times} \xrightarrow{\sim} \Gal(F^{\text{ab}} / F^{\text{nr}})
\]
satisfies the above assumptions, as does the reduced norm
\[
\text{Nrd} \colon \OO_D^{\times} \rightarrow \OO_F^{\times} \xrightarrow{\sim} \Gal(F^{\text{ab}} / F^{\text{nr}})
\]
for any central division algebra $D$ over $F$.
\end{eg}

\begin{defn}
We denote by 
\[
    \Rep_{L, \sm}^{\rtimes}(G)
\]
the full subcategory of $\Rep_L^{\rtimes}(G)$ of objects $V$ for which 
\[
    V = \bigcup_{H \leq_o G} V^H,
\]
as $H$ ranges over the open subgroups of $G$.
\end{defn}

\begin{eg}
When $L/K$ is finite and $G$ is any group with a surjective group homomorphism $G \twoheadrightarrow \Gal(L/K)$, one can take $G$ to have the discrete topology. This satisfies the above assumptions, and $\Rep_{L, \sm}^{\rtimes}(G)$ is nothing but $\Rep_{L}^{\rtimes}(G)$.
\end{eg}

We can characterise this in several equivalent ways.

\begin{lemma}\label{lem:smoothequiv}
    Suppose that $V \in \Rep_L^{\rtimes}(G)$. Then the following are equivalent:
    \begin{enumerate}
        \item $V \in \Rep_{L, \sm}^{\rtimes}(G)$,
        \item Every $v \in V$ is fixed by some open subgroup of $G$,
        \item $\Stab_G(v)$ is open in $G$,
        \item There is an open subgroup $H$ of $G$ such that $L \cdot V^H = V$.
    \end{enumerate}
    When $G$ has a basis of open normal subgroups, these are equivalent to:
    \begin{enumerate}
        \item[(5)] There is an open normal subgroup $N$ of $G$ such that $L \cdot V^N = V$.
    \end{enumerate}
\end{lemma}

For example, the assumption that $G$ has a basis of open normal subgroups is satisfied in the examples of Example \ref{eg:detandrednorm} above.

\begin{proof}
    The equivalence of the first three conditions is straightforward. To see that (4) implies (2), suppose that $L \cdot V^H = V$ for some open subgroup $H$ of $G$. Then any $v \in V$ is of the form
    \[
    \sum_{i} \lambda_i v_i \in L \cdot V^H.
    \]
    The $\lambda_i$ lie in some finite Galois extension $F$ of $K$, and so are fixed by some open subgroup $\Gal(L/F)$. In particular, $v$ is fixed by the open subgroup $H \cap \sigma^{-1}(\Gal(L/F))$. To see that (2) implies (4), let $e_1, ... , e_n$ be a basis of $V$ over $L$. Taking $H$ to be the intersection of the open subgroups $\Stab_G(e_i)$, we have that $L \cdot V^H = V$. The equivalence of point (5) with (4) is immediate.
\end{proof}

The link between the categories we have considered and $\Rep_{L, \sm}^{\rtimes}(G)$ is the following. If $N$ is an open normal subgroup of $G$, then $L^N / K$ is a finite Galois extension of $K$, because $\sigma(N)$ is open and normal in $\Gal(L/K)$. In particular, there is a surjective map
\[
G/N \twoheadrightarrow \Gal(L^N / K),
\]
and we may consider the category $\Rep_{L^N}^{\rtimes}(G/N)$, which satisfies the assumptions of Section \ref{sect:semilinearrepsGaloisext}, and therefore is described by the results of this paper. The following is a restatement of \cite[Cor.\ 3.9]{TAY4}.

\begin{prop}
    Suppose that $N$ is an open normal subgroup of $G$. Then the functor
    \[
        L \otimes_{L^N} - \colon \Rep_{L^N}^{\rtimes}(G/N) \rightarrow \Rep_{L, \sm}^{\rtimes}(G)
    \]
    is exact, monoidal, fully faithful, and the essential image is closed under sub-quotients. The essential image is explicitly described as those $V \in \Rep_{L, \sm}^{\rtimes}(G)$ for which $L \cdot V^N = V$, and on this full subcategory
    \[
        (-)^N \colon \Rep_{L, \sm}^{\rtimes}(G) \rightarrow \Rep_{L^N}^{\rtimes}(G/N)
    \]
defines a quasi-inverse to $L \otimes_{L^N} -$.
    \end{prop}

    In particular, the indecomposable (resp.\ irreducible) objects $V$ of $\Rep_{L, \sm}^{\rtimes}(G)$ which satisfy $L \cdot V^N = V$ are canonically identified with $\Ind_{L^N}^{\rtimes}(G/N)$ (resp.\ $\Irr_{L^N}^{\rtimes}(G/N)$), which is described by Theorem B.

    \begin{eg}
    Continuing Example \ref{eg:detandrednorm} above, then taking $N \coloneqq 1 + \pi^r M_n(\OO_F)$ for $r \geq 1$ and a uniformiser $\pi$ of $\OO_F$,
    \[
        G / N = \GL_n(\OO_F / 1 + \pi^r \OO_F),\ \text{ and } \ L^N = F^{\text{nr}}_r,
    \]
    where $F^{\text{nr}}_r$ denotes the $r$th Lubin-Tate extension of $F^{\text{nr}}$ from explicit local class field theory.
    
    For example, when $F = \bQ_p$ and $r = 1$, the surjection $G / N \rightarrow \Gal(L^N / K)$ becomes
    \[
       \det \colon \GL_n(\bF_p) \twoheadrightarrow \bF_p^{\times} \xrightarrow{\sim} \Gal(\bQ_p^{\text{nr}}(\zeta_p) / \bQ_p^{\text{nr}}).
    \]
    \end{eg}

\bibliography{biblio}{}
\bibliographystyle{plain}

\end{document}